\documentclass[11pt]{amsart}
\usepackage{fullpage}
\def\showanswers{0}

\newcommand{\hide}[1]{
\ifnum\showanswers=1
        {\color{red} #1} 
        \fi
\ifnum\showanswers=0
        \fi
}

\usepackage{hyperref}
\hypersetup{
    colorlinks=false,
    linkcolor=,
    filecolor=,      
    urlcolor=,
    pdftitle={Overleaf Example},
    pdfpagemode=FullScreen,
    }

\usepackage{comment} 

\usepackage{mathrsfs}
\usepackage{amsthm}
\usepackage{amsmath}
\usepackage{amssymb}

\def\sideremark#1{\ifvmode\leavevmode\fi\vadjust{\vbox to0pt{\vss
\hbox to 0pt{\hskip\hsize\hskip1em%
\vbox{\hsize2cm\tiny\raggedright\pretolerance10000%
\noindent {\color{red}{#1}}\hfill}\hss}\vbox to8pt{\vfil}\vss}}}%

\usepackage{arydshln}
\usepackage{tikz} 
\usepackage{wrapfig}

\theoremstyle{plain}
\newtheorem{lemma}{Lemma}[section]
\newtheorem{sublemma}{Lemma}[lemma]
\newtheorem{proposition}[lemma]{Proposition}
\newtheorem{theorem}[lemma]{Theorem}
\newtheorem{corollary}[lemma]{Corollary}
\newtheorem{maintheorem}{Theorem}

\newtheorem{maincorollary}[maintheorem]{Corollary}

\newtheorem*{keylemma*}{Key Lemma}
\newtheorem{setting}[lemma]{Setting}

\theoremstyle{definition}
\newtheorem{question}[lemma]{Question}
\newtheorem*{openquestion*}{Open Question}
\newtheorem{definition}[lemma]{Definition}
\newtheorem{example}[lemma]{Example}

\newtheorem{remark}[lemma]{Remark}
\newtheorem*{remark*}{Remark}

\newcommand{\op}{\mathrm{op}}
\newcommand{\Fr}{\mathrm{Fr}}
\newcommand{\Euc}{\mathrm{Euc}}

\usepackage[T1]{fontenc}
\usepackage{mathtools}

\newcommand{\Addresses}{{
  \bigskip
  \hide{\footnotesize}

  \textsc{Department of Mathematics, University of Oklahoma, 601 Elm Ave, Norman, OK 73019-3103, USA}\par\nopagebreak
  \textit{Email address}: \texttt{tomoya.tatsuno@ou.edu}

}}

\title{Sectional Curvature Pinching of Two-Step Nilmanifolds}
\author{Tomoya Tatsuno}
\date{}

\begin{document}

\begin{abstract}  
    We study the classical problem of sectional curvature pinching in the class of 2-step nilmanifolds, which are necessarily of mixed curvature. We show that the pinching constant of any 2-step nilmanifold lies in the compact interval $[-3, -\frac{3}{2}]$. The upper bound $-\frac{3}{2}$ is achieved by the complex Heisenberg group $\mathrm{Heis}_3(\mathbb{C})$ with a Ricci soliton metric. The upper bound exhibits rigidity: if a simply connected 2-step nilmanifold $N$ has the pinching constant $-\frac{3}{2}$, then $N$ admits a Ricci soliton complex Heisenberg group as a totally geodesic subgroup. On the other hand, the lower bound satisfies non-rigidity: any 2-step nilpotent Lie group admits a metric with pinching constant $-3$. This is derived by showing that there is an open neighborhood $U$ of $\mathrm{Heis}_3(\mathbb{R})\times \mathbb{R}^{n-3}$ in the space of $n$-dimensional 2-step nilmanifolds such that the pinching constant is $-3$ on $U$, and any 2-step nilpotent Lie group $N$ has a metric $g$ such that $(N,g)$ lies in $U$. In fact, if $N$ is not isomorphic to $\mathrm{Heis}_3(\mathbb{R})\times \mathbb{R}^{n-3}$, then there is a curve $g_t$ of metrics on $N$ with $(N,g_t)\in U$, showing that there are uncountably many left-invariant metrics on $N$ such that the pinching constant is $-3$. An algebraic characterization of a 2-step nilpotent Lie group that admits a metric with the pinching constant $-\frac{3}{2}$ is also given, and the pinching constants of various examples are computed.
\end{abstract}

\maketitle
\Addresses

\subsubsection*{Keywords}
Nilpotent Lie Groups; Pinching; Rigidity; Ricci Solitons; Mixed Curvature       

\subsubsection*{Mathematics Subject Classification} 	53C20, 53C30, 22E25, 53C25

\subsubsection*{Declarations of interest:} none

\section{Introduction} \label{section_introduction}
Pinching problems, first posed by Hopf, have been of interest to Riemannian geometers since the mid-20th century (see \cite[p. 545]{Berger_book}, \cite{Rauch}). For example, if the sectional curvature $K$ of a compact, simply connected Riemannian manifold $M$ satisfies $\frac{1}{4}<K\leq 1$, then $M$ is diffeomorphic to a sphere by the Differentiable Sphere Theorem \cite{Brendle_Schoen_Sphere_Theorem}. Admitting a metric such that $\frac{1}{4}<K\leq 1$ is equivalent to admitting a metric such that $\frac{1}{4}<\frac{K_{\min}}{K_{\max}} \leq 1$, where $K_{\min}$ and $K_{\max}$ denote the minimum and maximum of the sectional curvature, respectively. The quantity $\frac{K_{\min}}{K_{\max}}$ is called a pinching constant, which is an invariant up to isometry and scaling. 

 The Differentiable Sphere Theorem is optimal in the sense that compact rank 1 symmetric spaces have $\frac{1}{4}\leq K \leq 1$, or equivalently, $\frac{K_{\min}}{K_{\max}}=\frac{1}{4}$. Moreover, by \cite{Petersen--Tao}, for each dimension $n$ there exists $\epsilon(n)>0$ such that if a simply connected Riemannian manifold satisfies $\frac{1}{4}-\epsilon(n)\leq \frac{K_{\min}}{K_{\max}} \leq 1$, then $M$ is diffeomorphic to a sphere or a compact rank 1 symmetric space. We also note that positively curved homogeneous spaces are classified. The optimal pinching constants of even-dimensional ones were computed by Valiev \cite{Valiev} in 1979. In 1999, P{\"u}ttmann (\cite{Puttmann}) computed the optimal pinching constants of the odd-dimensional ones, developing the idea of the modified curvature operators, previously introduced by Thorpe \cite{Thorpe}, \cite{Thorpe_erratum} (see also \cite{Singer--Thorpe}).

 On the other hand, negatively curved pinched manifolds are not rigid. For any $\epsilon>0,$ there exists a compact Riemannian manifold $M$ with $-1-\epsilon\leq K \leq -1$ such that $M$ does not admit a metric with $K\equiv -1$, i.e., $M$ cannot be a hyperbolic manifold (\cite{Gromov--Thurston}). 

 However, under symmetry assumptions, one can say more. Let $M$ be a homogeneous Riemannian manifold with $K<0$. Necessarily, $M$ is a solvmanifold. Even further, $M$ is isometric to a solvable Lie group $S$ with a left-invariant metric of the form $S=\mathbb{R} \ltimes N$, where $N$ is a nilpotent Lie group with a left-invariant metric (see \cite{Heintze_Negative_Curvature}, \cite{Azencott--WilsonI}). Geometrically, $M$ contains $N$ as a hypersurface, and if $A$ denotes a unit vector spanning $\mathbb{R}$, then $A$ is a normal vector field of $N$. Algebraically, $A$ acts on the normal subgroup $N$ as an automorphism with special properties (see \cite[Proposition 2 and Theorem 3]{Heintze_Negative_Curvature}).
 
 Important examples of such solvmanifolds are noncompact rank 1 symmetric spaces, which satisfy $-4\leq K \leq -1$. If a homogeneous manifold $M$ satisfies $-4\leq K \leq -1$, then $M=\mathbb{R}\ltimes N$ described as above, and $N$ is a 2-step nilpotent Lie group with a left-invariant metric (see \cite[Theorem 5.1]{EH1} for a full result).  Conversely, any 2-step nilpotent Lie group with a left-invariant metric  (after rescaling) admits a solvable extension $\mathbb{R}\ltimes N$ which carries a left-invariant metric with $-4\leq K\leq -1$ (cf. \cite[Proposition 3.10]{EH1}). This way, there is a close relationship between quarter-pinched homogeneous manifolds of negative curvature and 2-step nilpotent Lie groups with a left-invariant metric.

 A connected Riemannian manifold admitting a transitive 2-step nilpotent Lie group of isometries is called a \textit{2-step nilmanifold}. By a result of Wilson \cite{Wilson_isometry_groups_on_homogeneous_nilmanifolds}, any 2-step nilmanifold is isometric to a 2-step nilpotent Lie group with a left-invariant metric. For this reason, we often call a 2-step nilpotent Lie group with a left-invariant metric simply a 2-step nilmanifold.

In \cite{EH1}, their analysis is reduced to the analysis of the upper and lower bounds of the sectional curvature of 2-step nilmanifolds. However, they did not study the maximum and minimum of the sectional curvature of 2-step nilmanifolds $N$, or equivalently, the pinching constant $\frac{K_{\min}}{K_{\max}}$ of $N$. As observed, there is a close relationship between 2-step nilmanifolds and negatively curved quarter-pinched manifolds. Naturally, one can ask:
 
 \begin{question}
     What can one say about $\frac{K_{\min}}{K_{\max}}$ of a 2-step nilmanifold $N$?
 \end{question}

Any non-abelian nilpotent Lie group $N$ satisfies an interesting curvature property: $N$ is of mixed curvature for any left-invariant metric, i.e., $K_{\max}>0$ and $K_{\min}<0$ (Wolf \cite{Wolf}). Similarly, its Ricci curvature has a mixed sign (\cite[Theorem 4]{Jensen_Einstein}, \cite{Milnor}). This means:
\begin{quote}
     For any left-invariant metric on a (non-abelian) nilpotent Lie group $N$, we have $K_{\max} \neq 0$ and $\frac{K_{\min}}{K_{\max}}$ is always well-defined.
\end{quote}
This makes it possible to study $\frac{K_{\min}}{K_{\max}}$ as a functional on the space of all 2-step nilpotent Lie groups with a left-invariant metric, or more precisely, on the variety of all metric 2-step nilpotent Lie algebras.
 We will compute the pinching constants for the following examples. Observe, the pinching constant is $-3$ for many metrics, except for the last example.

\begin{enumerate} 
    \item $\frac{K_{\min}}{K_{\max}}=-3$ for the 3-dimensional Heisenberg group $\mathrm{Heis}_3(\mathbb{R})$ with any left-invariant metric (Example \ref{example_heis_3}), and more generally,
    \item $\frac{K_{\min}}{K_{\max}}=-3$ for $\mathrm{Heis}_{2n+1}(\mathbb{R})\times \mathbb{R}^m$ ($m\geq 0$) with any left-invariant metric (Corollary \ref{pinching_constant_of_heis_2n+1_plus_abelian}),
    \item $\frac{K_{\min}}{K_{\max}}=-3$ for any 2-step nilpotent Lie group $N$ of dimension $\leq 5$ with any left-invariant metric (Corollary \ref{pinching_constant_up_to_dim_5}),
    \item $\frac{K_{\min}}{K_{\max}}=-\frac{3}{2}$ for the complex Heisenberg group $\mathrm{Heis}_3(\mathbb{C})$ with a left-invariant \textit{Ricci soliton} metric (Example \ref{example_heis_3_C_with_Ricci_soliton_metric}). Moreover, for any non-soliton metric of $\mathrm{Heis}_3(\mathbb{C})$, $-3 \leq \frac{K_{\min}}{K_{\max}}<-\frac{3}{2}$. 
\end{enumerate}

As one can see, having $\frac{K_{\min}}{K_{\max}}=-3$ seems ``typical'' for a 2-step nilmanifold. On the other hand, on $\mathrm{Heis}_3(\mathbb{C})$, the pinching constant functional achieves its maximum $-\frac{3}{2}$ at a Ricci soliton metric.

\begin{maintheorem} \label{maintheorem_bounds}
    Let $N$ be a 2-step nilpotent Lie group with a left-invariant metric. 
    
    Then, $-3\leq \frac{K_{\min}}{K_{\max}} \leq -\frac{3}{2}$. Both bounds are attained by some 2-step nilmanifolds.
\end{maintheorem}
\begin{remark}
    See Theorem \ref{maintheorem_rigidity} and Theorem \ref{maintheorem_characterization} for the rigidity of the maximal pinching constant $-\frac{3}{2}$. See Theorem \ref{maintheorem_invariant}, Corollary \ref{maincorollary_any_2-step_nilmanifod}, and Corollary \ref{maincorollary_open_set} for the non-rigidity of the minimal pinching constant $-3$.
\end{remark}

It is almost immediate that $\frac{K_{\min}}{K_{\max}}=-3$ for $\mathrm{Heis}_3(\mathbb{R})$ if one uses a so-called Milnor frame \cite[pp. 305--306]{Milnor}. Then it follows that $\frac{K_{\min}}{K_{\max}}=-3$ for $\mathrm{Heis}_3(\mathbb{R})\times \mathbb{R}^{m}$, which admits a unique left-invariant metric up to isometry and scaling. Beyond these examples, computing the pinching constant of 2-step nilmanifolds is non-trivial. Interestingly, the lower bound $-3$ and the upper bound $-\frac{3}{2}$ exhibit totally different behaviors. Recall, $\mathrm{Heis}_3(\mathbb{C})$ with a Ricci soliton metric achieves $\frac{K_{\min}}{K_{\max}}=-\frac{3}{2}$. The upper bound enjoys the following rigidity result.

\begin{maintheorem} \label{maintheorem_rigidity}
    Let $N$ be a simply connected 2-step nilpotent Lie group with a left-invariant metric. Suppose that the pinching constant is maximal, i.e., $\frac{K_{\min}}{K_{\max}}=-\frac{3}{2}$. Let $\pi$ be a 2-plane in $T_eN$ such that $K(\pi)=K_{\max}$.

    Then, $\pi$ is contained in the tangent space of a totally geodesic subgroup of $N$ that is isometric and isomorphic to $\mathrm{Heis}_3(\mathbb{C})$ with a Ricci soliton metric. 
\end{maintheorem}
\begin{remark}
    This implies the existence of a totally geodesic subgroup $\cong\mathrm{Heis}_3(\mathbb{C})$ in $N$, which is even stronger than the existence of a subgroup isomorphic to $\mathrm{Heis}_3(\mathbb{C})$. See Theorem \ref{maintheorem_characterization}.
\end{remark}
If $N$ is not simply connected, one can apply Theorem \ref{maintheorem_rigidity} to the simply connected cover $\tilde{N}$, as the pinching constant stays the same when taking a covering space. The plane $\pi$ may not be unique up to isometry; see Example \ref{example_non-uniqueness_of_maximal_plane}. There can exist different subgroups $\mathrm{Heis}_3(\mathbb{C})$ which are totally geodesic, and there may not exist an isometry that interchanges them.

We give the following characterization. Let $\mathrm{heis}_3(\mathbb{C})$ be the Lie algebra of $\mathrm{Heis}_3(\mathbb{C})$.

\begin{maintheorem} \label{maintheorem_characterization}
    Let $N$ be a simply connected 2-step nilpotent Lie group. Let $\mathcal{N}$ be its Lie algebra. Then, the following are equivalent.

    \begin{enumerate}
        \item $N$ admits a left-invariant metric with the maximal pinching constant, i.e., $\frac{K_{\min}}{K_{\max}}=-\frac{3}{2}$.
        \item $N$ admits a left-invariant metric such that there exists a totally geodesic subgroup isomorphic and isometric to $\mathrm{Heis}_3(\mathbb{C})$ with a Ricci soliton metric.
        \item There exist a subalgebra $\mathcal{N}_1$ and a vector subspace $\mathcal{N}_2$ such that:
        \begin{enumerate}
            \item $\mathcal{N}=\mathcal{N}_1\oplus \mathcal{N}_2$ as vector spaces,
            \item $\mathcal{N}_1$ is isomorphic to $\mathrm{heis}_3(\mathbb{C})$ as Lie algebras, and 
            \item $[\mathcal{N}_1, \mathcal{N}_2]\subseteq \mathcal{N}_2$.
        \end{enumerate}
    \end{enumerate}
\end{maintheorem}
\begin{remark}
    The third condition is purely algebraic. Note, even if $\mathcal{N}$ admits $\mathrm{heis}_3(\mathbb{C})$ as a subalgebra, $N$ may not admit a metric of the maximal pinching constant (Example \ref{almost_heis_C}).
\end{remark}

\begin{example} \label{product_of_heis_3_over_r_and_c} Let $\langle \cdot, \cdot \rangle_{\mathrm{Heis}_3(\mathbb{R})}$ be a left-invariant metric of $\mathrm{Heis}_3(\mathbb{R})$, normalized as $K_{\min}=-\frac{3}{4}$. Let $\langle \cdot, \cdot \rangle_{\mathrm{Heis}_3(\mathbb{C})}$ be a (unique) Ricci soliton left-invariant metric of $\mathrm{Heis}_3(\mathbb{C})$, normalized to be $K_{\min}=-\frac{3}{4}$. Consider $N=\mathrm{Heis}_3(\mathbb{R})\times \mathrm{Heis}_3(\mathbb{C})$ with a product left-invariant metric $\langle \cdot, \cdot \rangle_a:=\frac{1}{a^2}\langle \cdot, \cdot \rangle_{\mathrm{Heis}_3(\mathbb{R})}+\langle \cdot, \cdot \rangle_{\mathrm{Heis}_3(\mathbb{C})}$. Then,
\begin{align*}
    \frac{K_{\min}}{K_{\max}}(\langle \cdot, \cdot \rangle_a)=\begin{cases} 
-\frac{3}{2}   & \text{when } 0 < a \le 1 \\ 
-\frac{3}{2} a^2 & \text{when } 1 \le a \le \sqrt{2} \\ 
-3             & \text{when } a \ge \sqrt{2} 
\end{cases}, \text{while } \mathrm{scal}(\langle\cdot, \cdot\rangle_a)=-\frac{1}{2}(a^2+4).
\end{align*}
See Section \ref{section_product_metrics}. For all $a>0$, both factors are totally geodesic. Interestingly, the Ricci soliton metric on $N$ corresponds to $a=\sqrt{2}$, and has $\frac{K_{\min}}{K_{\max}}=-3$, i.e., a global minimum. Moreover, $\frac{K_{\min}}{K_{\max}}$ is not differentiable at this soliton metric. 
 
 The pinching functional assumes its maximum $-\frac{3}{2}$ at a non-soliton metric (for any $a\leq 1$). This is contrary to the case of $\mathrm{Heis}_3(\mathbb{C})$, where the soliton metric is a global maximum ($=-\frac{3}{2}$) of the pinching functional. This echoes P\"uttmann's result \cite{Puttmann}, where a metric achieving the maximal pinching constant is Einstein in some case, and not Einstein in other cases. 

 Since $\mathrm{scal}$ increases as $a\to 0$, $N$ admits uncountably many left-invariant metrics with $\frac{K_{\min}}{K_{\max}}=-\frac{3}{2}$ up to isometry and scaling. In particular, a left-invariant metric with $\frac{K_{\min}}{K_{\max}}=-\frac{3}{2}$ is not unique.
\end{example}

\begin{remark}
    The 7-dimensional quaternionic Heisenberg group $N$ with a Ricci soliton metric satisfies $\frac{K_{\min}}{K_{\max}}=-\frac{33}{16 }$ (see Section \ref{section_quaternionic_Heisenberg_group}). Since $\mathcal{N}:=\mathrm{Lie}N$ degenerates to $\mathrm{heis}_3(\mathbb{C})\oplus \mathbb{R}$, it turns out that $\frac{K_{\min}}{K_{\max}}\in [-3, -\frac{3}{2})$ when varying a left-invariant metric on $N$ (see also Corollary \ref{maincorollary_any_2-step_nilmanifod}). Note, $\frac{K_{\min}}{K_{\max}}=-\frac{3}{2}$ is not possible due to Theorem \ref{maintheorem_rigidity} (see Proposition \ref{proposition_ridigity_in_dim_7}). It is an interesting question when a Ricci soliton is a (local) maximum or minimum of the pinching constant functional. 
\end{remark}

We discuss the minimal pinching constant $\frac{K_{\min}}{K_{\max}}=-3$. Unlike the rigidity of the maximal pinching constant, we have non-rigidity.

First we define a new isometric invariant of 2-step nilmanifolds that characterizes $\mathrm{Heis}_3(\mathbb{R})\times \mathbb{R}^{n-3}$.
\begin{definition}
    Given a 2-step nilpotent Lie group $N$ with a left-invariant metric $\langle \cdot, \cdot \rangle$, define 
    \begin{align*}
        \mathcal{C}(N, \langle \cdot, \cdot\rangle)=\frac{-3 \mathrm{scal} + 2K_{\min}}{-K_{\min}},
    \end{align*}
    where $\mathrm{scal}$ denotes the scalar curvature. Note, $K_{\min}<0$ \cite{Wolf}. 
\end{definition}

 This is an invariant up to isometry and scaling, and has the following remarkable property.

\begin{maintheorem} \label{maintheorem_invariant}
    Let $N$ be a 2-step nilpotent Lie group of dimension 
    $n$ with a left-invariant metric. Let $\tilde{N}$ be its universal cover. Then,

    \begin{enumerate}
        \item $\mathcal{C}(N, \langle \cdot, \cdot\rangle)\geq 0$, and $\mathcal{C}(N, \langle \cdot, \cdot\rangle)=0$ if and only if $\tilde{N} \cong \mathrm{Heis}_3(\mathbb{R})\times \mathbb{R}^{n-3}$.
        \item If $\mathcal{C}(N, \langle \cdot, \cdot\rangle)<\frac{1}{4n^2(n-1)^2+4(n^2-2)^2}$, then $\frac{K_{\min}}{K_{\max}}=-3$.
        \item Suppose $\tilde{N}$ is not isomorphic to $\mathrm{Heis}_3(\mathbb{R})\times \mathbb{R}^{n-3}$. Then for any $\epsilon>0$, there exists a left-invariant metric $\langle \cdot, \cdot \rangle_{\epsilon}$ of $N$ such that $0<\mathcal{C}(N, \langle \cdot, \cdot\rangle_\epsilon)<\frac{1}{4n^2(n-1)^2+4(n^2-2)^2}$, and hence $\frac{K_{\min}}{K_{\max}}(N, \langle \cdot, \cdot\rangle_\epsilon)=-3$.
        
    \end{enumerate}
\end{maintheorem}
\begin{remark}
    If we normalize $K_{\min}=-1$, then $\mathcal{C}(N, \langle \cdot, \cdot \rangle)=-3\mathrm{scal}-2$, so geometrically, $\mathrm{scal}\leq -\frac{2}{3}$. The scalar curvature functional is maximized at $\mathrm{Heis}_3(\mathbb{R})\times \mathbb{R}^{n-3}$, and any maximum $(N, \langle\cdot, \cdot \rangle)$ of the scalar curvature functional has to be $\mathrm{Heis}_3(\mathbb{R})\times \mathbb{R}^{n-3}$ (up to covering).
\end{remark}

Recall, $\mathrm{Heis}_3(\mathbb{R})\times \mathbb{R}^{n-3}$ admits a unique left-invariant metric up to isometry and scaling. Lauret showed that $\mathrm{Heis}_3(\mathbb{R})\times \mathbb{R}^{n-3}$ and $\mathbb{R}\ltimes\mathbb{R}^{n-1}$ (with $\mathbb{R}$ acting on $\mathbb{R}^{n-1}$ via multiples of the identity) are the only simply connected Lie groups that have this property (\cite[Corollary 5.3]{Lauret_Degenerations}). 

The invariant $\mathcal{C}(N, \langle \cdot, \cdot\rangle)\geq 0$ measures how much the local geometry of a 2-step nilmanifold $(N, \langle \cdot, \cdot\rangle)$ deviates from the local geometry of this special 2-step nilmanifold $\mathrm{Heis}_3(\mathbb{R})\times \mathbb{R}^{n-3}$. The key to the proof is to observe that any simply connected 2-step nilmanifold $(N, \langle \cdot, \cdot\rangle)$ admits a totally geodesic subgroup isomorphic to $\mathrm{Heis}_3(\mathbb{R})$ (Lemma \ref{maximum_eigenspace_and_two_interpretations}). This is done by analyzing the maximum of a natural norm (called the operator norm) in the space of metric 2-step nilpotent Lie algebras that also appeared in the study of quarter-pinched homogeneous solvmanifolds \cite{EH1}.

Once one finds a totally geodesic $\mathrm{Heis}_3(\mathbb{R})$, then the Lie algebra $\mathcal{N}$ of $N$ decomposes as an orthogonal direct sum $\mathcal{N}=\mathrm{heis}_3(\mathbb{R})\oplus (\mathrm{heis}_3(\mathbb{R}))^{\perp},$ where $\mathrm{heis}_3(\mathbb{R})$ is the Lie algebra of $\mathrm{Heis}_3(\mathbb{R})$. The next key observation is that $\mathcal{C}(N, \langle \cdot, \cdot\rangle)$ is exactly measuring the size of $(\mathrm{heis}_3(\mathbb{R}))^{\perp}$ in another norm (called the Frobenius norm) in the space of metric 2-step nilpotent Lie algebras. Hence, $\mathcal{C}(N, \langle \cdot, \cdot\rangle)\geq 0$. Moreover, $\mathcal{C}(N, \langle \cdot, \cdot\rangle)= 0$ if and only if $(\mathrm{heis}_3(\mathbb{R}))^{\perp} \cong \mathbb{R}^{n-3}$, i.e., $\mathcal{N}=\mathrm{heis}_3(\mathbb{R})\oplus \mathbb{R}^{n-3}.$

The main difficulty of our result is to show that $\mathcal{C}(N, \langle \cdot, \cdot\rangle)<\frac{1}{4n^2(n-1)^2+4(n^2-2)^2}$ implies $\frac{K_{\min}}{K_{\max}}=-3$. The key is to use the modified curvature operators, following the ideas of P{\"u}ttmann \cite{Puttmann}. Our results heavily rely on the fact that 2-step nilpotent Lie algebras $\mathcal{N}$ have ``fewer structure constants'' due to $[\mathcal{N}, [\mathcal{N}, \mathcal{N}]]=0$. This enables the curvature computations. A key lemma is Lemma \ref{curv_op_and_totally_geodesic_subalgebras}, where we establish various nice properties of the curvature operator of 2-step nilmanifolds. The main idea is to show that a modified curvature operator has a block form essentially adapted to the decomposition $\mathcal{N}=\mathrm{heis}_3(\mathbb{R})\oplus (\mathrm{heis}_3(\mathbb{R}))^{\perp}$, and that the part corresponding to $(\mathrm{heis}_3(\mathbb{R}))^{\perp}$ can be analyzed purely in terms of the invariant $\mathcal{C}(N, \langle \cdot, \cdot\rangle)$. 

The fact that any $N$ admits a left-invariant metric with arbitrarily small $\mathcal{C}(N, \langle \cdot, \cdot\rangle)>0$ has the following consequence.

\begin{maincorollary} \label{maincorollary_any_2-step_nilmanifod}
    Any 2-step nilpotent Lie group $N$ admits a left-invariant metric with the minimal pinching constant $\frac{K_{\min}}{K_{\max}}=-3$. 

    Moreover, if $\tilde{N}$ is not isomorphic to $\mathrm{Heis}_3(\mathbb{R})\times \mathbb{R}^{n-3}$, then there are uncountably many left-invariant metrics on $N$ with $\frac{K_{\min}}{K_{\max}}=-3$ up to isometry and scaling.
\end{maincorollary}

The exception is made because $\mathrm{Heis}_3(\mathbb{R})\times \mathbb{R}^{n-3}$ admits a unique left-invariant metric up to isometry and scaling. Observe, once $\mathcal{C}(N, \langle \cdot, \cdot\rangle)<\frac{1}{4n^2(n-1)^2+4(n^2-2)^2}$ is established, the pinching constant stays the same $-3$, but the invariant $\mathcal{C}(N, \langle \cdot, \cdot\rangle)$ changes as $\mathcal{C}(N, \langle \cdot, \cdot\rangle)\to 0$. Hence, along this convergence, one obtains uncountably many metrics up to isometry and scaling.

Note, the condition $\mathcal{C}(N, \langle \cdot, \cdot\rangle)<\frac{1}{4n^2(n-1)^2+4(n^2-2)^2}$ is an open condition. 

\begin{maincorollary} \label{maincorollary_open_set}
    There is an open set $U$ in the space of 2-step nilmanifolds such that $\frac{K_{\min}}{K_{\max}}=-3$ on $U$.
\end{maincorollary}
To make this more precise and to explain the proof strategy, we discuss the \textit{moving bracket approach}, introduced by Heber \cite{Heber_Einstein} and developed by Lauret and others (see, e.g., \cite{Lauret_Ricci_Soliton_Nilmanifolds}, \cite{Lauret_Degenerations}). 

The curvature of $(N, \langle \cdot, \cdot \rangle)$ is determined by its metric Lie algebra $(\mathcal{N}, \langle \cdot, \cdot \rangle)$, i.e., the Lie algebra $\mathcal{N}$ with the induced inner product $\langle \cdot, \cdot \rangle$. In particular, by pulling back the Lie bracket via a linear isometry $(\mathbb{R}^n, \langle \cdot, \cdot \rangle_{\Euc}) \cong (\mathcal{N}, \langle \cdot, \cdot \rangle)$, one can model each metric Lie algebra as $(\mathbb{R}^n, \langle \cdot, \cdot \rangle_{\Euc}, \mu)$, where $\mu$ is a 2-step nilpotent Lie bracket. Here, $\langle \cdot, \cdot \rangle_{\Euc}$ denotes the Euclidean inner product.

Let $\mathcal{L}_n\subseteq (\Lambda^2 \mathbb{R}^n)^* \otimes \mathbb{R}^n$ be the algebraic subset of all 2-step nilpotent Lie brackets on $(\mathbb{R}^n, \langle \cdot, \cdot \rangle_{\Euc})$. As discussed, any 2-step nilmanifold $(N, \langle \cdot, \cdot \rangle)$ corresponds to a point $\mu\in \mathcal{L}_n$. We call the space $\mathcal{L}_n$ the space of 2-step nilpotent Lie brackets of dimension $n$. Since the 2-step nilpotent condition is a polynomial condition, $\mathcal{L}_n$ is a real affine variety. However, we consider the Euclidean topology of $\mathcal{L}_n$. Note that it is possible to compute the dimension of the moduli space, which is positive with a few exceptions (see \cite[Table 1]{Jablonski_moduli}).

Each $\mu\in \mathcal{L}_n$ is identified with a simply connected 2-step nilmanifold $N_\mu$ whose metric Lie algebra is isomorphic to $(\mathbb{R}^n, \langle \cdot, \cdot \rangle_{\Euc}, \mu)$. There is a change-of-basis action by $\mathrm{GL}_n(\mathbb{R})$ on $\mathcal{L}_n$ via $g.\mu:=g\mu(g^{-1}\cdot, g^{-1}\cdot)$. Consider the usual action of $\mathrm{GL}_n(\mathbb{R})$ on the set of inner products on $\mathbb{R}^n$ by $g.\langle \cdot, \cdot\rangle = \langle g^{-1}\cdot, g^{-1}\cdot \rangle$. Then, for any $\mu \in \mathcal{L}_n$, $g: (\mathbb{R}^n, \mu, g^{-1}.\langle \cdot, \cdot \rangle_{\mathrm{Euc}}) \to (\mathbb{R}^n, g.\mu, \langle \cdot, \cdot \rangle_{\mathrm{Euc}})$ is a metric Lie algebra isomorphism, and the corresponding simply connected 2-step nilmanifolds are isometric as Riemannian manifolds. Hence, we can study all the left-invariant metrics on a fixed 2-step nilpotent Lie group $N_\mu$ determined by $\mu \in \mathcal{L}_n$ by studying the orbit $\mathrm{GL}_n(\mathbb{R}).\mu$.

If $\lambda, \mu\in \mathcal{L}_n$, we say that $\mu$ degenerates to $\lambda$ if $\lambda\in \overline{\mathrm{GL}_n(\mathbb{R}).\mu}$, where the closure is relative to the usual vector space topology of $(\Lambda^2 \mathbb{R}^n)^* \otimes \mathbb{R}^n$. Let $\mu_{\mathrm{heis}_3(\mathbb{R})\oplus \mathbb{R}^{n-3}} \in \mathcal{L}_n$ denote a Lie bracket that corresponds to $\mathrm{heis}_3(\mathbb{R})\oplus \mathbb{R}^{n-3}$. It is well known that any $\mu\in \mathcal{L}_n$ degenerates to $\mathrm{heis}_3(\mathbb{R})\oplus \mathbb{R}^{n-3}$ (see \cite[Proof of Theorem 2.5]{Milnor}, \cite[Theorem 5.2]{Lauret_Degenerations}). 

Write $\mathcal{C}(\mu):=\mathcal{C}(N_\mu)$.

\begin{maintheorem} \label{maintheorem_open_set}
    There exists an open neighborhood $U$ of $\mu_{\mathrm{heis}_3(\mathbb{R})\oplus \mathbb{R}^{n-3}}$ in $\mathcal{L}_n$ such that for any $\mu \in U$, $\frac{K_{\min}}{K_{\max}}=-3$. Explicitly, $U=\{\mu\in \mathcal{L}_n\;|\; \mathcal{C}(\mu)<\frac{1}{4n^2(n-1)^2+4(n^2-2)^2}\}$.

    Moreover, for any $\mu \in \mathcal{L}_n$ that is not isomorphic to $\mu_{\mathrm{heis}_3(\mathbb{R})\oplus \mathbb{R}^{n-3}}$, the orbit $\mathrm{GL}_n(\mathbb{R}).\mu$ contains uncountably many elements of $U$. Therefore, any 2-step nilpotent Lie group whose Lie algebra is not isomorphic to $\mathrm{heis}_3\oplus \mathbb{R}^{n-3}$ admits uncountably many left-invariant metrics up to isometry and scaling such that $\frac{K_{\min}}{K_{\max}}=-3$.
\end{maintheorem}

In numerical experiments, one should observe that the pinching constant $-3$ is ``abundant'' among 2-step nilpotent Lie groups with a left-invariant metric. Our result gives a partial description of this phenomenon.

\subsubsection*{Organization.} Section \ref{section_preliminaries} discusses the preliminaries that will be used throughout this paper. The minimum of sectional curvature is determined in Section \ref{section_minimum}. Section \ref{section_maximum} contains the proof of the main theorems. In Section \ref{section_bounds_and_rigidity}, Theorem \ref{maintheorem_bounds} and Theorem \ref{maintheorem_rigidity} are proved. Then, Theorem \ref{maintheorem_characterization} is proved in Section \ref{section_characterization}. Theorem \ref{maintheorem_invariant} and Theorem \ref{maintheorem_open_set} are proved in Section \ref{section_invariant_and_open_set}. In Section \ref{section_examples}, various examples are presented.

\subsubsection*{Acknowledgement.} I would like to thank my advisor Michael Jablonski for his continuous support and guidance. I also thank Jorge Lauret for helpful suggestions, which led to a significant improvement of an earlier version of this work. This work was supported in part by National Science Foundation grant DMS-1906351. Part of this work was conducted during a stay at the Universidad Nacional de C\'ordoba, Argentina, and was supported by the DFCAS Dissertation Research Fellowship from the University of Oklahoma.

\section{Preliminaries} \label{section_preliminaries}

\subsection{Sectional Curvature and Curvature Operator of Riemannian Manifolds} \label{section_curv_op_of_general_riem_mfd} 

Let $(M,g)$ be a Riemannian manifold. The Riemannian curvature tensor $R$ of $(M,g)$ is given by the formula $R(X,Y)Z=\nabla_X \nabla_Y Z - \nabla_Y \nabla_X Z -\nabla_{[X,Y]} Z$.

As in \cite[p. 83]{Petersen_Riemannian_Geometry}, the curvature tensor induces a symmetric linear map called the curvature operator $R:\Lambda^2T_pM \to \Lambda^2T_pM$ such that for any $X,Y,Z,W \in T_pM$, $\langle R(X\wedge Y), Z\wedge W\rangle=R(X,Y,W,Z)$, where $\langle \cdot, \cdot \rangle$ on the left hand side denotes the inner product on $\Lambda^2T_pM$ induced by the Riemmanian metric. If $X, Y$ are orthonormal, then the sectional curvature of the 2-plane spanned by $X,Y$ is given by $K(X,Y)=\langle R(X\wedge Y), X\wedge Y\rangle$. 

Define the oriented Grassmannian by $Gr_2^+(T_pM):=\{\omega\in \Lambda^2T_pM \;|\; |\omega|=1,  \textrm{$\omega$ is decomposable} \}$, where $\omega$ being decompopsable means that there exists $X,Y \in T_pM$ such that $\omega = X\wedge Y$. The sectional curvature can be considered to be a function $K:Gr_2^+(T_pM) \to \mathbb{R}$. 
Let $\mathrm{symm}(\Lambda^2 T_pM, \langle \cdot, \cdot \rangle)$ be the space of symmetric linear maps from $\Lambda^2 T_pM$ to itself.

\begin{definition} \label{def_of_modified_curvature_operator}
    To each $T\in \Lambda^4 T_p^*M$, we define $\iota(T)\in \mathrm{symm}(\Lambda^2 T_pM, \langle \cdot, \cdot \rangle)$ by $\langle \iota(T)(\omega_1), \omega_2 \rangle=T(\omega_1\wedge \omega_2).$  The associated \textit{modified curvature operator} is defined by $R_T:=R+\iota(T)$. 
\end{definition}
\begin{remark}
    In \cite{Puttmann}, $\iota(T)$ is defined as a symmetric bilinear map $\iota(T)(\omega_1, \omega_2)=T(\omega_1\wedge \omega_2)$, and the modified curvature operator is denoted by $\hat{R}_T$.
\end{remark}

 It is important to observe that for any $\omega \in Gr_2^+(T_pM)$, we have $\langle R_T(\omega), \omega\rangle=\langle R(\omega), \omega\rangle=K(\omega)$, the sectional curvature of $\omega$, as $\omega \wedge \omega=0$.

\begin{proposition} [{\cite[Proposition 3.4]{Puttmann}}]
\label{bounds_by_curv_op_general_case}
    Let $K_{\max}$ be the maximum of the sectional curvature and $K_{\min}$ be the minimum of the sectional curvature. Let $\lambda_{\max}(R_T)$ and $\lambda_{\min}(R_T)$ be the maximum and minimum eigenvalues of the modified curvature operator $R_T$. Then, for any $T\in \Lambda^4 T_p^*M$, we have
\begin{align*}
\lambda_{\min}(R_T) \leq K_{\min} \leq K_{\max} \leq \lambda_{\max}(R_T).
\end{align*}
In particular, if there is a 2-plane $\pi$ and $T\in \Lambda^4 T_p^*M$ such that $K(\pi)=\lambda_{\max}(R_T)$, then $K_{\max}=K(\pi)$.

Moreover, $\sup \{ \lambda_{\min}(R_T)\;|\;T\in \Lambda^4(T_p^*M)\}$ and $\inf \{ \lambda_{\max}(R_T)\;|\;T\in \Lambda^4(T_p^*M)\}$ are both attained by some $T_{\max}, T_{\min}$, respectively. When $n=\dim M \leq 4$, $K_{\min}=\lambda_{\min}(R_{T_{\max}})$ and $K_{\max}=\lambda_{\max}(R_{T_{\min}})$.
\end{proposition}
The result for $n=4$ is due to Thorpe \cite{Thorpe}, \cite{Thorpe_erratum}. In dimension $\geq 5$, $K_{\min}=\lambda_{\min}(R_{T_{\max}})$ and $K_{\max}=\lambda_{\max}(R_{T_{\min}})$ are no longer true \cite[p. 314, Theorem 9.6]{Zoltek}.
\begin{proof}[Proof of Proposition \ref{bounds_by_curv_op_general_case}]
    See \cite[Proposition 3.4]{Puttmann}. Note that if there is a 2-plane $\pi$ and $T\in \Lambda^4 T_p^*M$ such that $K(\pi)=\lambda_{\max}(R_T)$, then $K(\pi)\leq K_{\max}\leq \lambda_{\max}(R_T)=K(\pi)$, which implies $K(\pi)=K_{\max}$.
\end{proof}

\subsection{The j-Map}

Let $N$ be a 2-step nilpotent real Lie group with a left-invariant metric $\langle \cdot, \cdot \rangle$. Let $\mathcal{N}$ be the Lie algebra of $N$, that is, the real vector space of left-invariant vector fields. The Lie algebra $\mathcal{N}$ inherits the inner product from $(N, \langle \cdot, \cdot \rangle)$, which is denoted by $\langle \cdot, \cdot \rangle$ as well. The pair $(\mathcal{N}, \langle \cdot, \cdot\rangle)$ is called a metric Lie algebra. We define the curvature (the curvature tensor, the covariant derivative, etc.) of a metric Lie algebra to be the Riemannian curvature (the curvature tensor, the covariant derivative, etc.) of the corresponding simply connected Lie group with a left-invariant metric.

To study the geometry of $(N, \langle \cdot, \cdot \rangle)$, we follow the approach introduced by A. Kaplan (\cite{Kaplan1}, \cite{Kaplan2}) and further developed by Eberlein (\cite{Eb1}). Let $\mathrm{so}(\mathcal{V}, \langle \cdot, \cdot \rangle)$ be the vector space of skew-symmetric linear maps of $(\mathcal{V}, \langle \cdot, \cdot \rangle)$. Let $j:\mathcal{Z} \to \mathrm{so}(\mathcal{V}, \langle \cdot, \cdot \rangle)$ be defined by $j(Z)X=({\mathrm{ad}X}^*)Z$, where $\mathcal{Z}$ is the center of $\mathcal{N}$, $\mathcal{V}$ is its orthogonal complement, $Z\in \mathcal{Z}$, and $X\in \mathcal{V}$. Here, ${\mathrm{ad}X}^*$ denotes the metric adjoint of the linear operator $\mathrm{ad}X$ with respect to the inner product $\langle \cdot, \cdot \rangle$ of $\mathcal{N}$. Equivalently, for any $Y\in \mathcal{N}$, we have
\begin{equation*}
    \langle j(Z)X, Y \rangle = \langle Z, [X,Y] \rangle.
\end{equation*}
See \cite{Eb1} for the details. We consider two norms for $j$-maps:
\begin{align*}
    ||j||_{\op}&:=\max\{|j(Z)X|\;|\;Z\in \mathcal{Z}, X\in \mathcal{V}, |Z|=1=|X|\}.\\
    ||j||_{\Fr}&:=\sqrt{\sum_{k=1}^p |j(Z_k)|_{\Fr}^2},
\end{align*}
where $p=\dim \mathcal{Z}$, $Z_1, ..., Z_p$ is an orthonormal basis of $\mathcal{Z}$, and $|j(Z_k)|_{\Fr}^2=\sum_{\alpha,\beta=1}^q |j(Z_k)_{\alpha \beta}|^2.$ Here, $q=\dim \mathcal{V}$ and $j(Z_k)_{\alpha \beta}$ is an $(\alpha, \beta)$-entry of the matrix representation of $j(Z_k)$ with respect an orthonormal basis $X_1,...,X_q$ of $\mathcal{V}$. We call $||j||_{\op}$ the \textit{operator norm} of $j$, and $||j||_{\Fr}$ the \textit{Frobenius norm} of $j$. The operator norm $||j||_{\op}$ appears naturally in the study of quarter-pinched homogeneous manifolds of negative curvature by Eberlein and Heber (\cite[p. 458]{EH1}). The Levi-Civita connection and the curvature tensor are computed in \cite[p. 620, (2.2)]{Eb1}, \cite[p. 621, (2.5)]{Eb1}, respectively. 

The curvature operator $R:\Lambda_2 \mathcal{N}\to \Lambda_2 \mathcal{N}$ is computed as follows.
\begin{lemma}\label{formulas_for_curv_op}
    Let $X, Y, X', Y'\in \mathcal{V}$ and $Z,W, Z', W'\in \mathcal{Z}$.
    Then,
    \begin{enumerate}
        \item $\langle R(X\wedge Z), Y\wedge W\rangle = \frac{1}{4}\langle j(Z)Y, j(W)X\rangle$.
        \item $\langle R(X\wedge Y), Z\wedge W \rangle = \frac{1}{4}\langle j(Z)Y, j(W)X \rangle - \frac{1}{4}\langle j(Z)X, j(W)Y \rangle$.
        \item $\langle R(X\wedge Y), X'\wedge Y' \rangle = -\frac{1}{2}\langle[X, Y], [X',Y']\rangle + \frac{1}{4}\langle [Y, X'], [X, Y']\rangle - \frac{1}{4}\langle [X, X'],[Y,Y']\rangle$.
        \item $\langle R(Z\wedge W), Z'\wedge W' \rangle = 0$.
    \end{enumerate}
\end{lemma}
\begin{proof}
    This is routine using the formulas from \cite[p. 621, (2.5)]{Eb1}.
\end{proof}

The sectional and scalar curvature of a general 2-plane in $\mathcal{N}$ are given by the following formula.

\begin{proposition}(\cite[p.459, (3.6)]{EH1}) \label{general_formula}
    Let $N$ be a 2-step nilpotent Lie group with a left-invariant metric $\langle \cdot, \cdot \rangle$. Let $\mathcal{N}$ be the Lie algebra of $N$. Let $\mathcal{Z}, \mathcal{V}, j$ as above. Let $X, X'\in \mathcal{V}$ and $Z, Z'\in \mathcal{Z}$ such that $X+Z, X'+Z'$ is orthonormal. Then, we have
    \begin{align*}
        K(X+Z,X'+Z')
        &=\frac{1}{4}|j(Z')X|^2 + \frac{1}{4}|j(Z)X'|^2 -\frac{3}{4}|[X,X']|^2 \\
        &\;\;\;+\frac{1}{2}\langle j(Z)X', j(Z')X \rangle 
        - \langle j(Z)X, j(Z')X' \rangle.
    \end{align*}

    If $X, X'\in \mathcal{V}$ are orthonormal, $K(X,X')=-\frac{3}{4}|[X,X']|^2$. If $X\in \mathcal{V}$ and $Z\in \mathcal{Z}$ are orthonormal, $K(X,Z)=\frac{1}{4}|j(Z)X|^2$.

    The scalar curvature is given by $\mathrm{scal}=-\frac{1}{4}||j||_{\Fr}^2$. 
\end{proposition}
Note that the formula (b) of (2.4) of \cite[p. 621]{Eb1} has a typo. The formula for the scalar curvature follows from \cite[p. 622, Proposition (2.5)]{Eb1}. In particular, $||j||_{\Fr}$ is independent of the choice of an orthonormal basis.

\subsection{Curvature Operator of 2-step Nilmanifolds}

Let $N$ be a 2-step nilpotent Lie group with a left-invariant metric $\langle \cdot, \cdot \rangle$. Let $\mathcal{N}=\mathrm{Lie}N$ be the Lie algebra of $N$ with the inner product $\langle \cdot, \cdot \rangle$ induced by $\langle \cdot, \cdot \rangle$ on $N$.  We identify with $\mathcal{N}$ with $T_eN$. Then, the results in Section \ref{section_curv_op_of_general_riem_mfd} apply to $\mathcal{N}\cong T_eN$. Let $\mathcal{Z}$ be the center and $\mathcal{V}=\mathcal{Z}^{\perp}$. Let $X_1, ..., X_q$ be an orthonormal basis of $\mathcal{V}$. Let $Z_1, ..., Z_p$ be an orthonormal basis of $\mathcal{Z}$. Then, $\{X_i\wedge X_j\ \;|\;i<j\}$, $\{X_i \wedge Z_j\ \;|\; i=1,\dots q,j=1,\dots p\}$, $\{Z_i \wedge Z_j \;|\; i<j\}$ form an orthonormal basis of $\Lambda^2\mathcal{N}$. We set
 \begin{align*}
     \Lambda^2\mathcal{V}&:= \mathbb{R}\textrm{-span} \{X_i\wedge X_j\ \;|\;i<j\},\\
     \Lambda^2\mathcal{Z} &:=\mathbb{R}\textrm{-span} \{Z_i \wedge Z_j \;|\; i<j\},\\
     \mathcal{V} \otimes \mathcal{Z} &:= \mathbb{R}\textrm{-span} \{X_i \wedge Z_j\ \;|\; i=1,\dots q,j=1,\dots p\}.
 \end{align*}
 
 Note that these subspaces do not depend on the choice of an orthonormal basis. For example, $\Lambda^2\mathcal{V}=\mathbb{R}\textrm{-span}\{x\wedge y\;|\;x,y\in \mathcal{V}\}.$ The decomposition $\Lambda^2\mathcal{N}=(\mathcal{V} \otimes \mathcal{Z}) \oplus(\Lambda^2\mathcal{V})\oplus  (\Lambda^2\mathcal{Z})$ is an orthogonal direct sum.
 \begin{proposition}\label{decomposition_of_the_curvature_operator}
     The subspaces $\mathcal{V} \otimes \mathcal{Z}$ and $\Lambda^2\mathcal{V}\oplus \Lambda^2\mathcal{Z}$ are invariant under the curvature operator $R$.
 \end{proposition}
 \begin{proof}
     It follows from \cite[p. 621, (2.5)]{Eb1} that $R(X_i, Z_j)X_l \in \mathcal{Z}$, so $\langle R(X_i\wedge Z_j), X_k\wedge X_l\rangle = \langle R(X_i, Z_j)X_l, X_k\rangle=0$. Similarly, $R(X_i, Z_j)Z_l \in \mathcal{V}$, so $\langle R(X_i\wedge Z_j), Z_k\wedge Z_l\rangle = \langle R(X_i,Z_j)Z_l, Z_k\rangle=0$. This shows $R(\mathcal{V}\otimes \mathcal{Z})\subseteq \mathcal{V}\otimes \mathcal{Z}$. Since $R$ is symmetric, it preserves the orthogonal complement $\Lambda^2 \mathcal{V}\oplus \Lambda^2\mathcal{Z}$ as well.
 \end{proof}
\begin{remark} \label{remark_block_form_of_curvature_operator}
    When we work with the basis, we use the lexicographical order in each subspace. Explicitly, the order of the basis is arranged as follows. Within $\Lambda^2\mathcal{V}$, the order is $X_1\wedge X_2, X_1\wedge X_3, ..., X_1\wedge X_q$, $X_2\wedge X_3, ..., X_2 \wedge X_q, ..., X_{q-1}\wedge X_q$. The basis of $\Lambda^2\mathcal{Z}$ is ordered as: $Z_1\wedge Z_2,..., Z_1\wedge Z_p, Z_2\wedge Z_3,..., Z_2\wedge Z_p,...,Z_{p-1}\wedge Z_p$. The basis of $\mathcal{V}\otimes \mathcal{Z}$ is ordered as $X_1\wedge Z_1, ..., X_q\wedge Z_1, X_1\wedge Z_2, ..., X_q\wedge Z_2,..., X_1\wedge Z_p, ..., X_q\wedge Z_p$.
\end{remark} 

\begin{remark} \label{curv_op_on_V_wedge_Z_using_j_map}
    Let $\alpha, \beta \in \{1,...,p\}$ and $i,j \in \{1,...,q\}$. By Lemma \ref{formulas_for_curv_op}, we have $\langle  R(X_j \wedge Z_\beta), X_i \wedge Z_\alpha \rangle=\langle -\frac{1}{4}j(Z_\beta)j(Z_\alpha)X_j, X_i\rangle$. This is useful when computing $R|_{\mathcal{V}\otimes \mathcal{Z}}$. For example, if $p=2$, i.e., $Z_1, Z_2$ form an orthonormal basis of the center $\mathcal{Z}$, then the matrix representation of $R|_{\mathcal{V}\otimes \mathcal{Z}}$ is given by 
    \begin{align*}
    R|_{\mathcal{V}\otimes \mathcal{Z}} = \begin{bmatrix}
            -\frac{1}{4}[j(Z_1)]^2 & -\frac{1}{4}[j(Z_2)][j(Z_1)] \\
            -\frac{1}{4}[j(Z_1)][j(Z_2)] & -\frac{1}{4}[j(Z_2)]^2
        \end{bmatrix},
    \end{align*}
    where $[j(Z_i)], i=1,2,$ denotes the matrix representation of $j(Z_i)$ with respect to $X_1, ..., X_q$.
\end{remark}

\begin{example} \label{example_heis_3}
    Let $\mathcal{N}=\mathrm{heis}_3(\mathbb{R})$ be the 3-dimensional Heisenberg algebra. It has a basis $e_1, e_2, e_3$ with $[e_1,e_2]=e_3$, and the center $\mathcal{Z}$ is spanned by $e_3$. We show that $K_{\max}=\frac{1}{4}||j||_{\op}^2$, $K_{\min}=-\frac{3}{4}||j||_{\op}^2$, and hence $\frac{K_{\min}}{K_{\max}}=-3$ for any inner product on the 3-dimensional Heisenberg algebra.

    Note, the 3-dimensional Heisenberg algebra admits a unique inner product up to isometry and scaling; there is no need to work with an arbitrary inner product. However, it is illustrative to include the computations with respect to any inner product.
    
    Let $\langle \cdot, \cdot \rangle$ be any inner product on $\mathcal{N}$. Let $Z_1$ be a unit vector in $\mathcal{Z}$. Let $\mathcal{V}=\mathcal{Z}^{\perp}$ be the orthogonal complement so that $\mathcal{N}=\mathcal{V}\oplus \mathcal{Z}$. Let $X_1, X_2$ be any orthonormal basis of $\mathcal{V}$, so $X_1,X_2,Z_1$ form an orthonormal basis of $\mathcal{N}$. Then, the $j$-map is given by $j(Z_1)=\begin{bmatrix}
        0 & -\lambda \\
        \lambda & 0
    \end{bmatrix}$ with respect to $X_1, X_2$, where $\lambda\in \mathbb{R}$. This is because $j(Z_1)$ is skew-symmetric. This implies $[X_1, X_2]=\lambda Z_1$. Also, $\lambda\neq 0$ (otherwise $\mathcal{N}$ would be abelian). Now, $\Lambda^2 \mathcal{N}$ has an orthonormal basis $X_1\wedge Z_1, X_2\wedge Z_1, X_1 \wedge X_2$. By Proposition \ref{decomposition_of_the_curvature_operator} and Lemma \ref{formulas_for_curv_op}, the curvature operator is given by 
        $R=\begin{bmatrix}
        \frac{1}{4}\lambda^2 & 0 & 0 \\
        0 & \frac{1}{4}\lambda^2 & 0 \\
        0 & 0 & -\frac{3}{4}\lambda^2
    \end{bmatrix}$
    with respect to $X_1\wedge Z_1, X_2\wedge Z_1, X_1 \wedge X_2$. 
    
    Observe, $K_{\max}\geq K(X_1,Z_1)=\frac{1}{4}\lambda^2$ and $K_{\min}\leq K(X_1,X_2)= -\frac{3}{4}\lambda^2$. On the other hand, the maximum and minimum eigenvalues of $R$ provide bounds of $K_{\max}, K_{\min}$, respectively (Proposition \ref{bounds_by_curv_op_general_case} with $T=0$), and it follows that $K_{\max}=\frac{1}{4}\lambda^2$ and $K_{\min}=-\frac{3}{4}\lambda^2$. Moreover, since $j(Z_1)^2=-\lambda^2 \mathrm{id}_{\mathcal{V}}$, we have $||j||_{\op}^2=\lambda^2$, so $K_{\max}=\frac{1}{4}||j||_{\op}^2$ and $K_{\min}=-\frac{3}{4}||j||_{\op}^2$ for any left-invariant metric on the 3-dimensional Heisenberg group.
\end{example}

\subsection{Totally Geodesic Subalgebras}

Let us define a totally geodesic subalgebra.

\begin{definition}\label{def_of_aligned_totally_geodesic_subalgebras}
    Let $(\mathcal{N}, \langle \cdot, \cdot \rangle)$ be a metric 2-step nilpotent Lie algebra. Let $\mathcal{Z}$ be the center, and $\mathcal{V}=\mathcal{Z}^{\perp}$ be the orthogonal complement of $\mathcal{Z}$. Let $\nabla$ be the Levi-Civita connection of $\mathcal{N}$.

    \begin{enumerate}
        \item (\cite[p. 623]{Eb1}) A subalgebra $\mathcal{N}_1$ of $\mathcal{N}$ is called a \textit{totally geodesic subalgebra} if for any $\xi_1, \xi_2\in\mathcal{N}_1$, $\nabla_{\xi_1} \xi_2 \in \mathcal{N}_1$.
        \item A subalgebra $\mathcal{N}_1$ is called \textit{aligned} if there exist subspaces $\mathcal{V}_1$ of $\mathcal{V}$ and $\mathcal{Z}_1$ of $\mathcal{Z}$ such that $\mathcal{N}_1=\mathcal{V}_1 \oplus \mathcal{Z}_1$ and $\mathcal{Z}_1=\mathfrak{z}(\mathcal{N}_1)$, the center of $\mathcal{N}_1$.
        \item We say that a totally geodesic subalgebra $\mathcal{N}_1$ is \textit{aligned} or an \textit{aligned totally geodesic subalgebra} if the subalgebra $\mathcal{N}_1$ is also aligned in addition to being totally geodesic.
    \end{enumerate}       
\end{definition}
This means the connected Lie subgroup with Lie algebra $\mathcal{N}_1$ is a totally geodesic submanifold (\cite[p. 623]{Eb1}). Not all totally geodesic subalgebras are aligned (see \cite[Example 4.4 (iii)]{Decoste_Demeyer}).

\begin{lemma} \label{criterion_of_totally_geodesic_when_it_is_aligned}
    Let $\mathcal{N}$ be a metric 2-step nilpotent Lie algebra. Let $\mathcal{Z}$ be the center, and $\mathcal{V}=\mathcal{Z}^{\perp}$ be the orthogonal complement of $\mathcal{Z}$.
    
    Let $\mathcal{N}_1$ be an aligned subalgebra, i.e., there exist subspaces $\mathcal{V}_1$ of $\mathcal{V}$ and $\mathcal{Z}_1$ of $\mathcal{Z}$ such that $\mathcal{N}_1=\mathcal{V}_1 \oplus \mathcal{Z}_1$ and $\mathcal{Z}_1=\mathfrak{z}(\mathcal{N}_1)$, the center of $\mathcal{N}_1$.

    Then, $\mathcal{N}_1$ is an aligned totally geodesic subalgebra if and only if $\mathcal{V}_1$ is invariant under $j(z)$ for any $z\in \mathcal{Z}_1$.

    In this case, the $j$-map $j_1$ of $\mathcal{N}_1$ is given by the restriction of the $j$-map $j$ of $\mathcal{N}$, i.e., $j_1(z)=j(z)
    |_{\mathcal{V}_1}$ for $z\in \mathcal{Z}_1$, and  we often write $j$ for the $j$-map of both $\mathcal{N}$ and $\mathcal{N}_1$.
\end{lemma}

\begin{proof}
    If $\mathcal{N}_1$ is an aligned totally geodesic subalgebra, then let $z\in \mathcal{Z}_1$ and $x\in \mathcal{V}_1$. Since $\mathcal{N}_1$ is totally geodesic, the Levi-Civita connection $\nabla$ of $\mathcal{N}$ coincides with the Levi-Civita connection of $\mathcal{N}_1$. Then, by \cite[p. 620, (2.2)]{Eb1}, we have $-\frac{1}{2}j_1(z)x=\nabla_z x=-\frac{1}{2}j(z)x$ as $\mathcal{N}_1$ is aligned. In particular, $j(z)x\in \mathcal{V}_1$ and $j_1$ is the restriction of $j$. The converse follows from the linearity of $\nabla$, \cite[p. 620, (2.2)]{Eb1}, and the fact that $\mathcal{N}_1$ is a subalgebra.
\end{proof}

 A 2-step nilpotent Lie algebra $\mathcal{N}$ is called \textit{nonsingular} if $ad X:\mathcal{N} \to \mathcal{Z}$ is surjective for any $X\in \mathcal{N}\setminus \mathcal{Z}$, where $\mathcal{Z}$ is the center. Suppose $\mathcal{N}$ is equipped with an inner product. Then, by \cite[(1.8) Lemma]{Eb1}, $\mathcal{N}$ is nonsingular if and only if $j(Z):\mathcal{V}\to \mathcal{V}$ is invertible for any $Z\in \mathcal{Z}\setminus \{0\}$, where $\mathcal{V}=\mathcal{Z}^{\perp}$ (\cite[(1.8) Lemma, p. 620]{Eb1}). Moreover, if $\mathcal{N}$ is nonsingular, then $[\mathcal{N},\mathcal{N}]=\mathcal{Z}$. In particular, if $\mathcal{N}$ is of Heisenberg type, i.e., $j(Z)^2=-|Z|^2 \mathrm{id}_{\mathcal{V}}$ for any $Z\in \mathcal{Z}$ (\cite{Kaplan1}, \cite{Kaplan2}), then $\mathcal{N}$ is nonsingular.

\begin{proposition}\label{nonsingular_ones_are_aligned}
    Let $(\mathcal{N}, \langle \cdot, \cdot \rangle)$ be a metric 2-step nilpotent Lie algebra. Let $\mathcal{N}_1$ be a totally geodesic subalgebra. Let $\mathfrak{z}(\mathcal{N}_1)$ be the center of the subalgebra $\mathcal{N}_1$.
    
    Suppose $[\mathcal{N}_1, \mathcal{N}_1]=\mathfrak{z}(\mathcal{N}_1)$ and $j_1(Z_1)$ is invertible for some $Z_1\in \mathfrak{z}(\mathcal{N}_1)$, where $j_1$ denotes the $j$-map of $\mathcal{N}_1$. Then $\mathcal{N}_1$ is aligned. 
    
    In particular, if $\mathcal{N}_1$ is nonsingular, then $\mathcal{N}_1$ is aligned.
\end{proposition}

\begin{proof}[Proof of Proposition \ref{nonsingular_ones_are_aligned}]
Let $\mathcal{Z}=\mathfrak{z}(\mathcal{N})$ be the center of $\mathcal{N}$, and let $\mathcal{V}$ be the orthogonal complement of $\mathcal{Z}$. Let $\mathcal{Z}_1=\mathcal{Z}\cap \mathcal{N}_1$. Then, $[\mathcal{N}_1, \mathcal{N}_1]\subseteq [\mathcal{N},\mathcal{N}]\cap \mathcal{N}_1\subseteq \mathcal{Z}_1 \subseteq \mathfrak{z}(\mathcal{N}_1)$. Since $[\mathcal{N}_1, \mathcal{N}_1]=\mathfrak{z}(\mathcal{N}_1)$, it follows that $\mathcal{Z}_1=\mathfrak{z}(\mathcal{N}_1)$. 

Let $\mathcal{V}_1=\mathcal{Z}_1^{\perp}\cap \mathcal{N}_1$, so that $\mathcal{N}_1=\mathcal{V}_1\oplus \mathcal{Z}_1$ (an orthogonal direct sum). We will show $\mathcal{V}_1\subseteq \mathcal{V}$. Let $\xi\in \mathcal{V}_1$. Take an element $0\neq Z_1\in \mathcal{Z}_1$ such that $j_1(Z_1):\mathcal{V}_1\to \mathcal{V}_1$ is invertible, where $j_1$ is the $j$-map of $\mathcal{N}_1$. Let $\eta\in \mathcal{V}_1$ such that $\xi=j_1(Z_1)(\eta)$. Write $\eta=x+z$, where $x\in \mathcal{V}$ and $z\in \mathcal{Z}$. 

Let $\nabla$ be the Levi-Civita connection of the ambient Lie algebra $\mathcal{N}$. Since $\mathcal{N}_1$ is totally geodesic, $\nabla_{Z_1} \eta \in \mathcal{N}_1$, and $\nabla$ coincides with the Levi-Civita connection of $\mathcal{N}_1$. In particular, it follows from \cite[p. 620, (2.2)]{Eb1} that $\nabla_{Z_1}\eta = -\frac{1}{2}j_1(Z_1)\eta=-\frac{1}{2}\xi$. On the other hand, again it follows from \cite[p. 620, (2.2)]{Eb1} that $\nabla_{Z_1}\eta=\nabla_{Z_1}x + \nabla_{Z_1} z=-\frac{1}{2}j(Z_1)x\in \mathcal{V}$, where this $j$ denotes the $j$-map of the ambient Lie algebra $\mathcal{N}$. Thus, $\xi\in \mathcal{V}$, so $\mathcal{V}_1\subseteq \mathcal{V}$. This completes the proof. 
\end{proof}

For any subspace $\mathcal{A}$ of $\mathcal{V}$ and $\mathcal{B}$ of $\mathcal{Z}$, write $\mathcal{A}\otimes \mathcal{B}=\{A\wedge B\;|\;A\in \mathcal{A}, B\in \mathcal{B}\}$. We are ready to state a technical lemma that will be used repeatedly in this article. The setting is as follows. 

\begin{setting} \label{setting_totally_geodesic_T}
    Let $(\mathcal{N}, \langle \cdot, \cdot \rangle)$ be a metric 2-step nilpotent Lie algebra. Let $\mathcal{Z}$ be the center of $\mathcal{N}$, and $\mathcal{V}=\mathcal{Z}^{\perp}$ be the orthogonal complement of $\mathcal{Z}$.
    
    Let $\mathcal{N}_1$ be a totally geodesic subalgebra of $\mathcal{N}$ that is aligned, i.e., there exist subspaces $\mathcal{V}_1$ of $\mathcal{V}$ and $\mathcal{Z}_1$ of $\mathcal{Z}$ such that $\mathcal{N}_1=\mathcal{V}_1 \oplus \mathcal{Z}_1$ and $\mathcal{Z}_1$ is the center of $\mathcal{Z}$. Let $\mathcal{V}_2:=\mathcal{V}_1^{\perp}\cap \mathcal{V}$ and $\mathcal{Z}_2:=\mathcal{Z}_1^{\perp}\cap \mathcal{Z}$ so that $\mathcal{V}=\mathcal{V}_1\oplus \mathcal{V}_2$ and $\mathcal{Z}=\mathcal{Z}_1\oplus \mathcal{Z}_2$. Notice, $\mathcal{V}\otimes \mathcal{Z}$ decomposes as an orthogonal direct sum of the subspaces $\mathcal{V}_1\otimes \mathcal{Z}_1, \mathcal{V}_1\otimes \mathcal{Z}_2, \mathcal{V}_2\otimes \mathcal{Z}_1$, and $\mathcal{V}_2\otimes \mathcal{Z}_2$.

    Let $X_1,...,X_s$ be an orthonormal basis of $\mathcal{V}_1$, and $X_{s+1},..., X_q$ an orthonormal basis of $\mathcal{V}_2$. Let $Z_1,..., Z_t$ be an orthonormal basis of $\mathcal{Z}_1$, and $Z_{t+1},..., Z_p$ an orthonormal basis of $\mathcal{Z}_2$, so that $\{X_i\wedge Z_\alpha\;|\;X_i \in \mathcal{V}_a, Z_\alpha \in \mathcal{Z}_b\}$ form an orthonormal basis of $\mathcal{V}_a\otimes \mathcal{Z}_b$ for $a=1, 2$ and $b=1,2$. 
    
    For $i\in \{1,...,s\}, \alpha \in \{1,...,t\}, j\in \{s+1,...,q\}, \beta\in \{t+1,...,p\}$, set $c_{i\alpha j \beta}:=\langle R(X_i\wedge Z_\alpha), X_j\wedge Z_\beta \rangle$. Let $X_i=e_i, i=1,...,q$ and $Z_\alpha=e_{\alpha+q}, \alpha=1,...,p$, so that $\{e_1,...,e_{p+q}\}$ forms an orthonormal basis of $\mathcal{N}$. Then, $\{e_a\wedge e_b\wedge e_c\wedge e_d\;|\;a<b<c<d \}$ forms an orthonormal basis of $\Lambda^4 \mathcal{N}$, where $a,b,c,d\in \{1,...,p+q\}$. Form the dual basis $\{(e_a\wedge e_b\wedge e_c\wedge e_d)^*\;|\; a<b<c<d\}$ of $\Lambda^4 \mathcal{N}^*=(\Lambda^4 \mathcal{N})^*$.
\end{setting}

Aligned totally geodesic subalgebras behave well with respect to the curvature operator, as the next lemma indicates.

\begin{lemma} \label{curv_op_and_totally_geodesic_subalgebras}
Let $(\mathcal{N}, \langle \cdot, \cdot \rangle)$ be a metric 2-step nilpotent Lie algebra. Let $\mathcal{N}_1$ be an aligned totally geodesic subalgebra of $\mathcal{N}$. We use the notations $\mathcal{V}_1, \mathcal{Z}_1, \mathcal{V}_2, \mathcal{Z}_2, X_i,Z_\alpha$, and $c_{i \alpha j \beta}$ as in Setting \ref{setting_totally_geodesic_T}.

    Define $T=\sum_{i,\alpha,j,\beta} c_{i \alpha j \beta} (e_i\wedge e_j\wedge e_{\alpha+q}\wedge e_{\beta+q})^*$, where the sum is over the indices $i\in \{1,...,s\}, \alpha \in \{1,...,t\}, j\in \{s+1,...,q\}, \beta \in \{t+1,...,p\}$. Notice, $i<j<\alpha+q<\beta+q$, so $T$ is well defined. Observe, $T=\sum_{i,\alpha,j,\beta} c_{i \alpha j \beta} (X_i\wedge X_j\wedge Z_{\alpha}\wedge Z_{\beta})^*$ with $X_i\in \mathcal{V}_1, X_j\in \mathcal{V}_2, Z_\alpha\in \mathcal{Z}_1, Z_\beta\in \mathcal{Z}_2$.
    
    Then, 
    \begin{enumerate}
        \item For any $z\in \mathcal{Z}_1$, $j(z)$ leaves $\mathcal{V}_1$ and $\mathcal{V}_2$ invariant.
        \item For any $w\in \mathcal{Z}_2$, $j(w)(\mathcal{V}_1)\subseteq \mathcal{V}_2$.
        \item $R$ leaves $\mathcal{V}\otimes \mathcal{Z}$ invariant, $\langle R(\mathcal{V}_1\otimes \mathcal{Z}_1), \mathcal{V}_1\otimes \mathcal{Z}_2 \rangle =0$, and $\langle R(\mathcal{V}_1\otimes \mathcal{Z}_1), \mathcal{V}_2\otimes \mathcal{Z}_1 \rangle =0$. Hence, $R(\mathcal{V}_1\otimes \mathcal{Z}_1)\subseteq \mathcal{V}_1\otimes \mathcal{Z}_1 \oplus \mathcal{V}_2\otimes \mathcal{Z}_2$.
        \item  $\iota(T)$ leaves $\mathcal{V}\otimes \mathcal{Z}$ invariant, $R_T=R+\iota(T)$ leaves $\mathcal{V}\otimes \mathcal{Z}$, and $\iota(T)(\mathcal{V}_1\otimes \mathcal{Z}_1) \subseteq \mathcal{V}_2\otimes \mathcal{Z}_2$.
        \item $R_T$ leaves $\mathcal{V}_1\otimes \mathcal{Z}_1$ invariant. Hence, $R_T$ leaves each summund of the decomposition $\Lambda^2\mathcal{N}=(\mathcal{V}_1\otimes \mathcal{Z}_1)\oplus ((\mathcal{V}_1\otimes \mathcal{Z}_1)^{\perp}\cap (\mathcal{V}\otimes \mathcal{Z})) \oplus (\Lambda^2\mathcal{V}\oplus \Lambda^2\mathcal{Z})$ invariant. 
        \item For any $x,y\in \mathcal{V}_1, z, w\in \mathcal{Z}_1$, $\langle R_T(x\wedge z), y\wedge w\rangle =\langle R(x\wedge z), y\wedge w\rangle  $.
    \end{enumerate}

\end{lemma}

\begin{proof} By Lemma \ref{criterion_of_totally_geodesic_when_it_is_aligned}, the item 1 follows, as $j(z)$ is skey symmetric. The item 2 follows as $\mathcal{N}_1$ is a subalgebra. By Proposition \ref{decomposition_of_the_curvature_operator}, $R$ leaves $\mathcal{V}\otimes \mathcal{Z}$ invariant. By the items 1 and 2 of this lemma and the item 1 of Lemma \ref{formulas_for_curv_op}, the item 3 follows.
    
    We show the item 4. To see $\iota(T)$ leaves $\mathcal{V}\otimes \mathcal{Z}$ invariant, recall $T=\sum_{i,\alpha,j,\beta} c_{i \alpha j \beta} (X_i\wedge X_j\wedge Z_{\alpha}\wedge Z_{\beta})^*$. Let $i'\in \{1,...,q\}, \alpha' \in \{1,...,p\}$. For $x,y\in \mathcal{N}$, we compute $\langle \iota(T)(X_{i'}\wedge Z_{\alpha'}), x\wedge y\rangle=T(X_{i'} \wedge Z_{\alpha'} \wedge x\wedge y)$. If $x,y\in \{X_1,...,X_q\}$ or $x,y\in\{Z_1,...,Z_p\}$, then $T(X_{i'} \wedge Z_{\alpha'} \wedge x\wedge y)=0$. This means $\iota(T)(X_{i'} \wedge Z_{\alpha'})\in (\mathcal{V}\wedge \mathcal{V} \oplus \mathcal{Z}\wedge \mathcal{Z})^{\perp}=\mathcal{V}\otimes \mathcal{Z}$. Thus, $\iota(T)$ leaves $\mathcal{V}\otimes \mathcal{Z}$ invariant. By the item 3, we see $R_T$ leaves $\mathcal{V}\otimes \mathcal{Z}$ invariant. Next, we show that $\iota(T)(\mathcal{V}_1\otimes \mathcal{Z}_1)\subseteq \mathcal{V}_2\otimes \mathcal{Z}_2$. Indeed, if $i' \in \{1,...,s\}$ and $\alpha'\in \{1,...,t\}$, then $\langle \iota(T)(X_{i'}\wedge Z_{\alpha'}), x\wedge y\rangle=T(X_{i'}\wedge Z_{\alpha'} \wedge x\wedge y)=0$ when $x\wedge y\in \mathcal{V}_1\otimes \mathcal{Z}_1 \oplus \mathcal{V}_1\otimes \mathcal{Z}_2 \oplus \mathcal{V}_2\otimes \mathcal{Z}_1$. This proves the item 4.
    
    Let us show the item 5. We show that $R_T$ leaves $\mathcal{V}_1\otimes \mathcal{Z}_1$ invariant. By the items 3 and 4, we know $R_T(\mathcal{V}_1\otimes \mathcal{Z}_1)\subseteq \mathcal{V}_1\otimes \mathcal{Z}_1 \oplus \mathcal{V}_2\otimes \mathcal{Z}_2.$ We show that $\langle R_T(\mathcal{V}_1\otimes \mathcal{Z}_1), \mathcal{V}_2\otimes \mathcal{Z}_2 \rangle =0$.
    
    For $i\in \{1,...,s\}, \alpha\in \{1,...,t\}, j\in \{s+1,...,q\}, \beta \in \{t+1,...,p\}$, we have $\langle \iota(T)(X_i\wedge Z_\alpha), X_j\wedge Z_\beta \rangle = T(X_i\wedge Z_\alpha \wedge X_j\wedge Z_\beta)=-c_{i \alpha j \beta}=-\langle R(X_i\wedge Z_\alpha), X_j\wedge Z_\beta \rangle$ by definition of $T$. Thus, $\langle R_T(X_i\wedge Z_\alpha), X_j\wedge Z_\beta \rangle = \langle R(X_i\wedge Z_\alpha)+\iota(T)(X_i\wedge Z_\alpha), X_j\wedge Z_\beta \rangle = 0$ . Therefore, $R_T(\mathcal{V}_1\otimes \mathcal{Z}_1)$ is orthogonal to $\mathcal{V}_2\otimes \mathcal{Z}_2$. Thus, $R_T$ leaves $\mathcal{V}_1\otimes \mathcal{Z}_1$ invariant. Since $R_T$ leaves $\mathcal{V}\otimes \mathcal{Z}$ invariant and $R_T$ is symmetric, the item 5 follows. The item 6 immediately follows from the item 4 of this lemma.
\end{proof}

\subsection{Useful Lemmas}

\begin{lemma} \label{rotation_in_2-plane} (\cite[p.460]{EH1})
     Let $\mathcal{N}$ be a metric 2-step nilpotent Lie algebra. Let $\mathcal{Z}$ the center, $\mathcal{V}=\mathcal{Z}^{\perp}$, and $j$ be the $j$-map. Then, for any 2-plane $\pi$ of $\mathcal{N}$, there exist $X, X'\in \mathcal{V}$ and $Z, Z'\in \mathcal{Z}$ such that $X+Z$ and $X'+Z'$ form an orthonormal basis of $\pi$, and $\langle X, X' \rangle = 0 = \langle Z, Z' \rangle.$  
\end{lemma}

We have the following preliminary lemma. 
\begin{lemma}\label{maximum_eigenspace_and_two_interpretations}
    Let $\mathcal{N}$ be a 2-step nilpotent Lie algebra with an inner product. Let $\mathcal{Z}$ be the center of $\mathcal{N}$, and $\mathcal{V}=\mathcal{Z}^{\perp}$.
    \begin{itemize}
        \item There exist $Z\in \mathcal{Z}$, $|Z|=1$, and $X\in \mathcal{V}, |X|=1,$ such that $|j(Z)X|=||j||_{\op}$.
        \item Let $Z\in \mathcal{Z}$, $|Z|=1$, and $X\in \mathcal{V}, |X|=1,$ be any vectors such that $|j(Z)X|=||j||_{\op}$. Set $Y:=\frac{j(Z)X}{||j||_{\op}}$.  Then, 
        
    \begin{enumerate}
        \item $-j(Z)^2 X = ||j||_{\op}^2 X$ and $[X, j(Z)X] = ||j||_{\op}^2 Z.$
        \item $Y$ is a unit vector orthogonal to $X$, $[X,Y]=||j||_{\op}Z$, and $|[X,Y]|=||j||_{\op}$.
        \item $K(X,Z)=\frac{1}{4}||j||_{\op}^2$ and $K(X,Y)=-\frac{3}{4}||j||_{\op}^2$. In particular, $K_{\max}\geq \frac{1}{4}||j||_{\op}^2$ and $K_{\min}\leq -\frac{3}{4}||j||_{\op}^2$.
        \item Let $\mathcal{N}_1=\mathbb{R}\textrm{-span}\{X, Y, Z\}$. Then, $\mathcal{N}_1$ is an aligned totally geodesic subalgebra of $\mathcal{N}$ isomorphic to $\mathrm{heis}_3(\mathbb{R})$.
    \end{enumerate}
    \end{itemize}
    
In particular, any simply connected 2-step nilmanifold admits a totally geodesic subgroup $\cong\mathrm{Heis}_3(\mathbb{R})$.
\end{lemma}
\begin{proof}
    By compactness, there exist unit vectors $Z\in \mathcal{Z}$ and $X\in \mathcal{V}$ such that $|j(Z)X|=||j||_{\op}$. Assume $Z_0 \in \mathcal{Z}, |Z_0|=1$ and $X_0\in \mathcal{V}, |X_0|=1$ satisfy $|j(Z_0)X_0|=||j||_{\op}$. Observe, $||j||_{\op}^2 = \max \{|j(Z)X|^2 \;|\; Z\in \mathcal{Z}, X\in \mathcal{V}, |Z|=1=|X|\}$. The point is that we can consider $|j(Z)X|^2$ as the quadratic form associated to a symmetric operator in two ways as follows. 

    We show the item 1. First, for any $X\in \mathcal{V}, |X|=1$, we have $\langle -j(Z_0)^2 X, X \rangle =|j(Z_0)X|^2 \leq ||j||_{\op}^2$. Thus, $||j||_{\op}^2$ is the maximum eigenvalue of the symmetric linear map $-j(Z_0)^2:\mathcal{V}\to \mathcal{V}$ and $X_0$ is an eigenvector corresponding to $||j||_{\op}^2$, so we have $-j(Z_0)^2 X_0 = ||j||_{\op}^2 X_0$. Next, for any $Z\in \mathcal{Z}, |Z|=1$, we have $\langle \mathrm{ad} X_0 \circ (\mathrm{ad} X_0)^* Z, Z \rangle =\langle [X_0, j(Z)X_0], Z \rangle 
        =|j(Z)X_0|^2
        \leq ||j||_{\op}^2.$ Thus, $||j||_{\op}^2$ is the maximum eigenvalue of the symmetric linear map $\mathrm{ad} X_0 \circ (\mathrm{ad} X_0)^*:\mathcal{Z}\to \mathcal{Z}$ and $Z_0$ is an eigenvector corresponding to $||j||_{\op}^2$, so we have $\mathrm{ad} X_0 \circ (\mathrm{ad} X_0)^* Z_0 = ||j||_{\op}^2 Z_0$. This means $[X_0, j(Z_0)X_0] = ||j||^2 Z_0$, so the item 1 follows. 
    
    For the item 2, since $j(Z_0)$ is skew-symmetric, the vector $Y_0=\frac{j(Z_0)X_0}{||j||_{\op}}$ is orthogonal to $X_0$. Since $[X_0, j(Z_0)X_0] = ||j||^2 Z_0$, we get $[X_0,Y_0]=||j||_{\op}Z_0$. It follows that $|[X_0,Y_0]|=||j||_{\op}$. Then, the item 3 follows immediately from Proposition \ref{general_formula}.

 Lastly, we show the item 4. Since $[X_0, Y_0]=||j||_{\op}Z_0$, it is a subalgebra, and by definition of aligned subalgebras, it is aligned. We have $j(Z_0)X_0=||j||_{\op}Y_0$ and $j(Z_0)Y_0=-||j||_{\op}X_0$. By Lemma \ref{criterion_of_totally_geodesic_when_it_is_aligned}, it is an aligned totally geodesic subalgebra. Since the dimension is $3$, it is isomorphic to $\mathrm{heis}_3(\mathbb{R})$.
\end{proof}

\begin{remark}
    \cite[p. 1606, Theorem 6.3]{Kerr--Payne} shows that a nontrivial totally geodesic subalgebra of a certain filiform nilpotent Lie algebra (whose step size is $\geq3$) must be abelian.
\end{remark}

Another useful lemma gives a bound of the Lie bracket, borrowed from \cite[p.458]{EH1}.

\begin{lemma} \label{upper_bound_of_Lie_bracket}(\cite[p.458]{EH1})
    If $X,Y\in \mathcal{V}$, then $|[X,Y]|\leq ||j||_{\op}|X||Y|$. Similarly, if $X\in \mathcal{V}$ and $Z\in \mathcal{Z}$, then $|j(Z)X|\leq ||j||_{\op}|Z||X|$.
\end{lemma}

\section{Minimum of Sectional Curvature} \label{section_minimum}
In this section, we find the minimum of the sectional curvature.
\begin{theorem}\label{minimum_of_sectional_curvature_of_2-step_nilpotent_Lie_groups}
Let $N$ be a 2-step nilpotent Lie group with a left-invariant metric $\langle \cdot, \cdot \rangle$. Let $\mathcal{N}$ be the Lie algebra of $N$. Let $\mathcal{Z}$ be the center, and $\mathcal{V}=\mathcal{Z}^{\perp}$. Then, 
$K_{min}=-\frac{3}{4}||j||_{\op}^2$. Moreover, if $K(\pi)=K_{min}$, then $\pi \subseteq \mathcal{V}$.
\end{theorem}

\begin{proof}
Let $\pi$ be a 2-plane in $\mathcal{N}$, and let $X+Z, X'+Z'$ be an orthonormal basis of $\pi$, where $X, X'\in \mathcal{V}$ and $Z,Z'\in \mathcal{Z}$. We may assume that $\langle X, X' \rangle=0=\langle Z, Z' \rangle$ by Lemma \ref{rotation_in_2-plane}. By Proposition \ref{general_formula}, we have $K(X+Z,X'+Z') 
        =\frac{1}{4}|j(Z)X'+j(Z')X|^2-\frac{3}{4}|[X,X']|^2 - \langle j(Z)X, j(Z')X' \rangle 
        \geq -\frac{3}{4}|[X,X']|^2 - \langle j(Z)X, j(Z')X' \rangle 
        = -\frac{3}{4}(|[X,X']|^2 + \frac{4}{3}\langle j(Z)X, j(Z')X' \rangle).$
By Lemma \ref{upper_bound_of_Lie_bracket}, we have $|[X,X']| \leq ||j||_{\op} |X||X'|$. Using the Cauchy--Schwarz inequality and Lemma \ref{upper_bound_of_Lie_bracket}, we have $\langle j(Z)X, j(Z')X' \rangle \leq |j(Z)X||j(Z')X'|\leq ||j||_{\op}^2|X||Z||X'||Z'|$. 

Therefore, $|[X,X']|^2 + \frac{4}{3}\langle j(Z)X, j(Z')X' \rangle
    \leq ||j||_{\op}^2|X|^2|X'|^2 +  \frac{4}{3}||j||_{\op}^2|X||Z||X'||Z'|.$ Since $X+Z$ and $X'+Z'$ are unit vectors, we have $|X|^2+|Z|^2=1$ and $|X'|^2+|Z'|^2=1$. Thus, $|[X,X']|^2 + \frac{4}{3}\langle j(Z)X, j(Z')X' \rangle\leq ||j||_{\op}^2(|X|^2|X'|^2+ \frac{4}{3}|X||X'|\sqrt{1-|X|^2}\sqrt{1-|X'|^2})$. Let $f(x,y):=x^2y^2 +  \frac{4}{3}xy\sqrt{1-x^2}\sqrt{1-y^2}$. Then, for any $x,y \in [0,1]$,  $f(x,y)\leq 1$ and $f(x,y)=1$ if and only if $x=y=1$. Thus, $K(X+Z, X'+Z') \geq -\frac{3}{4}||j||_{\op}^2$. If $K(X+Z, X'+Z') = -\frac{3}{4}||j||_{\op}^2$, then in particular, $f(|X|,|X'|)=1$. Thus, $|X|=|X'|=1$, and we have $\pi \subseteq \mathcal{V}$. Furthermore, by Lemma \ref{maximum_eigenspace_and_two_interpretations}, there exists a 2-plane $\pi$ such that $K(\pi)=-\frac{3}{4}||j||_{\op}^2$. This completes the proof.
\end{proof}

\section{Maximum of Sectional Curvature} \label{section_maximum}

\subsection{Proof of Theorem \ref{maintheorem_bounds} and Theorem \ref{maintheorem_rigidity}} \label{section_bounds_and_rigidity}

We prove the bounds that appear in Theorem \ref{maintheorem_bounds}.

\begin{proof}[Proof of the bounds of Theorem \ref{maintheorem_bounds}]  

Let $\mathcal{N}$ be the Lie algebra of $N$. Let $\pi$ be any 2-plane of $\mathcal{N}$.
    Let $X+Z, X'+Z'$ be an orthonormal basis of the 2-plane $\pi$, where $X,X'\in \mathcal{V}$ and $Z, Z'\in \mathcal{Z}$. Then, 
    we may assume that $\langle X, X' \rangle=0=\langle Z, Z' \rangle$ by Lemma \ref{rotation_in_2-plane}. By Proposition \ref{general_formula}, we have $K(X+Z,X'+Z')
        =\frac{1}{4}|j(Z')X|^2 + \frac{1}{4}|j(Z)X'|^2 -\frac{3}{4}|[X,X']|^2 +\frac{1}{2}\langle j(Z)X', j(Z')X \rangle 
        - \langle j(Z)X, j(Z')X' \rangle
        =\frac{1}{4}|j(Z)X'+j(Z')X|^2-\frac{3}{4}|[X,X']|^2 - \langle j(Z)X, j(Z')X' \rangle.$

    By applying Lemma \ref{upper_bound_of_Lie_bracket} and the Cauchy--Schwarz inequality repeatedly, we have
\begin{align}
       K(X+Z,X'+Z') \nonumber&=\frac{1}{4}|j(Z)X'+j(Z')X|^2-\frac{3}{4}|[X,X']|^2 - \langle j(Z)X, j(Z')X' \rangle \\ \label{first_ineq_in_rigidity_of_maximum}
        &\leq \frac{1}{4}|j(Z)X'+j(Z')X|^2 - \langle j(Z)X, j(Z')X' \rangle \\        
        \label{second_ineq_in_rigidity_of_maximum}
        &\leq \frac{1}{4}|j(Z)X'+j(Z')X|^2 + |j(Z)X||j(Z')X'| \\        
        \label{third_ineq_in_rigidity_of_maximum}
        &\leq \frac{1}{4}|j(Z)X'+j(Z')X|^2 + ||j||_{\op}^2|Z||X||Z'||X'| \\
        \label{fourth_ineq_in_rigidity_of_maximum}
        &\leq \frac{1}{4}(|j(Z)X'|+|j(Z')X|)^2 + ||j||_{\op}^2|Z||X||Z'||X'| \\
        \label{fifth_ineq_in_rigidity_of_maximum}
        &\leq \frac{1}{4}||j||_{\op}^2(|Z||X'|+|Z'||X|)^2 + ||j||_{\op}^2|Z||X||Z'||X'|. 
\end{align}
Set $|X|=\cos \theta$ and $|X'|=\cos \varphi$, where $0\leq \theta, \varphi \leq \frac{\pi}{2}$. Then, the last line can be calculated as $\frac{1}{4}||j||_{\op}^2(|Z||X'|+|Z'||X|)^2 + ||j||_{\op}^2|Z||X||Z'||X'|=\frac{1}{4}||j||_{\op}^2(( \cos\varphi \sin\theta + \cos\theta \sin\varphi)^2 + 4\cos\theta \cos\varphi \sin\theta \sin\varphi) =\frac{1}{4}||j||_{\op}^2(\sin^2(\theta+\varphi)+\sin 2\theta \sin 2\varphi)$. Now we use the last inequality:

\begin{equation}\label{sixth_ineq_in_rigidity_of_maximum}
    \frac{1}{4}||j||_{\op}^2(\sin^2(\theta+\varphi)+\sin 2\theta \sin 2\varphi)  
    \leq \frac{1}{4}||j||_{\op}^2(1+1)
    =\frac{1}{2}||j||_{\op}^2.
\end{equation}

Therefore, $K_{\max}\leq \frac{1}{2}||j||_{\op}^2$. On the other hand, by the item 3 of Lemma \ref{maximum_eigenspace_and_two_interpretations}, we always have a 2-plane whose sectional curvature is $\frac{1}{4}||j||_{\op}^2$. Therefore, $K_{\max}\geq \frac{1}{4}||j||_{\op}^2$. Combining this with Theorem \ref{minimum_of_sectional_curvature_of_2-step_nilpotent_Lie_groups} proves the bounds.
\end{proof}

We show $\mathrm{Heis}_3(\mathbb{C})$ with a Ricci soliton metric satisfies $K_{\max}=\frac{1}{2}||j||_{\op}^2$, $K_{\min}=-\frac{3}{4}||j||_{\op}^2$, and hence the pinching constant is given by $\frac{K_{\min}}{K_{\max}}=-\frac{3}{2}$.

\begin{example} \label{example_heis_3_C_with_Ricci_soliton_metric}

    The 3-dimensional complex Heisenberg algebra $\mathrm{heis}_3(\mathbb{C})$ is a complex 2-step nilpotent Lie algebra that has a $\mathbb{C}$-basis $e_1, e_2, e_3$ with the Lie bracket relation $[e_1, e_2]=e_3$, and $e_3$ is in the center. The corresponding simply connected Lie group $\mathrm{Heis}_3(\mathbb{C})$ is the 3-dimensional complex Heisenberg group. The algebra $\mathrm{heis}_3(\mathbb{C})$ has a $\mathbb{R}$-basis $\{e_1,\sqrt{-1}e_1, e_2, \sqrt{-1}e_2, e_3, \sqrt{-1}e_3\}$. Set $X_1=e_1, X_2=\sqrt{-1}e_1, X_3=e_2, X_4=-\sqrt{-1}e_2, Z_1=e_3, Z_2=-\sqrt{-1}e_3$. Then, $[X_1, X_2]=0, [X_1, X_3]=Z_1, [X_1, X_4]=Z_2, [X_2, X_3]=-Z_2,
    [X_2, X_4]=Z_1, [X_3, X_4]=0$. Let $\langle \cdot, \cdot \rangle$ be the inner product on $\mathrm{heis}_3(\mathbb{C})$ for which $X_1,X_2,X_3,X_4,Z_1,Z_2$ form an orthonormal basis. The center $\mathcal{Z}$ is given by $\mathcal{Z}=\mathbb{R}\textrm{-span}\{Z_1, Z_2\}$. The orthogonal complement $\mathcal{V}=\mathcal{Z}^{\perp}$ is given by $ \mathcal{V}=\mathbb{R}\textrm{-span}\{X_1,X_2,X_3,X_4\}$.

    The $j$-map of $\mathrm{heis}_3(\mathbb{C})$ with respect to the basis $\{X_1, X_2,X_3,X_4\}$ is given by
 
        $j(Z_1) =A_1:= \begin{bmatrix}
        0 & 0 & -1 & 0 \\
        0 & 0 & 0 & -1 \\
        1 & 0 & 0 & 0\\
        0 & 1 & 0 & 0
    \end{bmatrix}, 
    j(Z_2) = A_2:= \begin{bmatrix}
        0 & 0 & 0 & -1 \\
        0 & 0 & 1 & 0 \\
        0 & -1 & 0 & 0\\
        1 & 0 & 0 & 0
    \end{bmatrix}.$
    It follows that $j(z)^2=-|z|^2 \mathrm{id}_{\mathcal{V}}$ for any $z\in \mathcal{Z}$, so one sees that this metric becomes of Heisenberg type. In particular, this metric is a Ricci soliton metric (\cite[Example 3.3 (i)]{Lauret_Ricci_Soliton_Nilmanifolds}). Moreover, it follows that $||j||_{\op}=1$. We show that $K_{\max}=\frac{1}{2}||j||_{\op}^2=\frac{1}{2}$.

    Let $X=\frac{1}{\sqrt{2}}X_1, X'=\frac{1}{\sqrt{2}}X_2, Z=\frac{1}{\sqrt{2}}Z_1, Z'=\frac{1}{\sqrt{2}}Z_2$. Then, $\{X+Z, X'+Z'\}$ is orthonormal. Observe, $j(Z)X=\frac{1}{2}j(Z_1)X_1=\frac{1}{2}X_3, j(Z')X'=\frac{1}{2}j(Z_2)X_2=-\frac{1}{2}X_3, j(Z)X'=\frac{1}{2}j(Z_1)X_2=\frac{1}{2}X_4, j(Z')X=\frac{1}{2}j(Z_2)X_1=\frac{1}{2}X_4$, and $[X,X']=\frac{1}{2}[X_1,X_2]=0$. By Proposition \ref{general_formula}, it follows that$ K(X+Z,X'+Z')=\frac{1}{2}=\frac{1}{2}||j||_{\op}^2.$
    Thus, $\frac{1}{2}||j||_{\op}^2\leq K_{\max} \leq \frac{1}{2}||j||_{\op}^2$ from the bounds in Theorem \ref{maintheorem_bounds} that were proved above. Therefore, $K_{\max}=\frac{1}{2}||j||_{\op}^2$. By Theorem \ref{minimum_of_sectional_curvature_of_2-step_nilpotent_Lie_groups}, we have $K_{\min}=-\frac{3}{4}||j||_{\op}^2$. Therefore, the pinching constant is given by $\frac{K_{\min}}{K_{\max}}=-\frac{3}{2}$.
\end{example}

\begin{remark}\label{rescaling_heis_3_c_ricci_soliton}
    Let $\mathcal{N}=\mathrm{heis}_3(\mathbb{C})$. Define an inner product $\langle \cdot, \cdot\rangle_\xi$ on $\mathcal{N}$ so that $\xi X_1, ..., \xi X_4, \xi Z_1, \xi Z_2$ form an orthonormal basis, where $\xi>0$. Then the corresponding $j$-map (denoted by $j_\xi$) is given by  

        $j_\xi(\xi Z_1) = \begin{bmatrix}
        0 & 0 & -\xi & 0 \\
        0 & 0 & 0 & -\xi \\
        \xi & 0 & 0 & 0\\
        0 & \xi & 0 & 0
    \end{bmatrix}, 
    j_\xi(\xi Z_2) = \begin{bmatrix}
        0 & 0 & 0 & -\xi \\
        0 & 0 & \xi & 0 \\
        0 & -\xi & 0 & 0\\
        \xi & 0 & 0 & 0
    \end{bmatrix}$ with respect to $\xi X_1, ..., \xi X_4$.

    It follows that any metric 2-step nilpotent Lie algebra with the $j$-map given by the matrices above is $\mathrm{heis}_3(\mathbb{C})$ with a Ricci soliton inner product. In this basis, the Lie bracket relations are given by 
    \begin{align*}
        [\xi X_1, \xi X_2]&=0, & [\xi X_1, \xi X_3]&=\xi(\xi Z_1), & [\xi X_1, \xi X_4]&=\xi (\xi Z_2) \\
        [\xi X_2, \xi X_3]&=-\xi (\xi Z_2), & [\xi X_2, \xi X_4]&=\xi(\xi Z_1), & [\xi X_3, \xi X_4]&=0.
    \end{align*}

    For later use, it is useful to compute the matrix representation of $R|_{\mathcal{V}\otimes \mathcal{Z}}$ using the formula in Remark \ref{curv_op_on_V_wedge_Z_using_j_map}: $R|_{\mathcal{V}\otimes \mathcal{Z}}=   \begin{bmatrix}
           -\frac{1}{4}[j_\xi(\xi Z_1)]^2 & -\frac{1}{4} [j_\xi(\xi Z_2)][j_\xi(\xi Z_1) ] \\
           -\frac{1}{4}[j_\xi(\xi Z_1)][j_\xi(\xi Z_2)] & -\frac{1}{4}[j_\xi(\xi Z_2)]^2
       \end{bmatrix}$, where $[j_\xi(Z_i)]$ denotes the matrix representation of $j_\xi(Z_i)$ ($i=1,2$) with respect to $\xi X_1, ..., \xi X_4$. In particular, $-\frac{1}{4}[j_\xi(\xi Z_1)]^2=\frac{\xi^2}{4}I_4=-\frac{1}{4}[j_\xi(\xi Z_2)]^2$ and $-\frac{1}{4}[j_\xi(\xi Z_1)][j_\xi(\xi Z_2)]  =\begin{bmatrix} 0 & -\frac{\xi^2}{4} & 0 & 0\\ \frac{\xi^2}{4} & 0 & 0 & 0\\ 0 & 0 & 0 & \frac{\xi^2}{4}\\ 0 & 0 & -\frac{\xi^2}{4} & 0 \end{bmatrix}=-(-\frac{1}{4}[j_\xi(\xi Z_2)][j_\xi(\xi Z_1)])$. 

       In this case, the maximum of the sectional curvature is given by $\frac{1}{2}\xi^2$.
\end{remark}
\begin{proof}[Proof of Theorem \ref{maintheorem_bounds}]
    We already proved the bounds (see above). Example \ref{example_heis_3} and Example \ref{example_heis_3_C_with_Ricci_soliton_metric} give examples that achieve those bounds.
\end{proof}

We give a proof of Theorem \ref{maintheorem_rigidity} here. We first reduce it to the Lie algebra level. Suppose that $\pi$ is contained in a totally geodesic subalgebra $\mathcal{N}_1$ isomorphic to $\mathrm{heis}_3(\mathbb{C})$ with a Ricci soliton inner product as metric Lie algebras. 

Let $N_1=\exp(\mathcal{N}_1)$, where $\exp:\mathcal{N}\to N$ is a Lie group exponential map. Since $N$ is simply connected, $\exp$ is a diffeomorphism. Hence, $N_1$ is closed and simply connected in $N$. By the Baker--Campbell--Hausdorff formula together with $[\mathcal{N}, [\mathcal{N}, \mathcal{N}]]=0$ (see \cite[p. 616, (1.1)]{Eb1}), $N_1$ is a closed Lie subgroup of $N_1$. Since the Lie algebra of $N_1$ is $\mathcal{N}_1$, $N_1$ is isomorphic and isometric to $\mathrm{Heis}_3(\mathbb{C})$ with a Ricci soliton metric, as $N_1$ is simply connected. Now the result follows.

Therefore, it suffices to show that $\pi$ is contained in a totally geodesic subalgebra $\mathcal{N}_1$ isomorphic to $\mathrm{heis}_3(\mathbb{C})$ with a Ricci soliton inner product as metric Lie algebras. 

Let $\mathcal{Z}$ be the center of $\mathcal{N}$ and $\mathcal{V}=\mathcal{Z}^{\perp}$. Let $\pi$ be a 2-plane such that $K(\pi)=\frac{1}{2}||j||_{\op}^2$.
    Let $X+Z, X'+Z'$ be an orthonormal basis of the 2-plane $\pi$, where $X,X'\in \mathcal{V}$ and $Z, Z'\in \mathcal{Z}$. Then, 
    we may assume that $\langle X, X' \rangle=0=\langle Z, Z' \rangle$ by Lemma \ref{rotation_in_2-plane}. 
    
Since $K(\pi)=K(X+Z,X'+Z')=\frac{1}{2}||j||_{\op}^2$, all the 6 inequalities (\ref{first_ineq_in_rigidity_of_maximum}), (\ref{second_ineq_in_rigidity_of_maximum}), (\ref{third_ineq_in_rigidity_of_maximum}), (\ref{fourth_ineq_in_rigidity_of_maximum}), (\ref{fifth_ineq_in_rigidity_of_maximum}), (\ref{sixth_ineq_in_rigidity_of_maximum}) in the proof of Theorem \ref{maintheorem_bounds} must be the equalities. 

\begin{lemma} \label{lemma_consequence_of_sixth_ineq}
    We have
\begin{enumerate}
    \item $|X|=|X'|=|Z|=|Z'|=\frac{1}{\sqrt{2}}.$
    \item $|j(Z)X|=||j||_{\op}|Z||X|=\frac{1}{2}||j||_{\op}$ and $|j(Z')X'|=||j||_{\op}|Z'||X'|=\frac{1}{2}||j||_{\op}.$
    \item $|j(Z)X'|=||j||_{\op}|Z||X'|=\frac{1}{2}||j||_{\op}$ and $|j(Z')X|=||j||_{\op}|Z'||X|=\frac{1}{2}||j||_{\op}$.
\end{enumerate}
\end{lemma}

\begin{proof}
    The equality for the sixth inequality (\ref{sixth_ineq_in_rigidity_of_maximum}) together with the condition $\theta, \varphi \in [0,\frac{\pi}{2}]$ imply that $\theta=\varphi=\frac{\pi}{4}$, and the item 1 follows. The equality for the third inequality (\ref{third_ineq_in_rigidity_of_maximum}) and the fifth inequality (\ref{fifth_ineq_in_rigidity_of_maximum}) 
    together with the item 1 imply the items 2 and 3.
\end{proof}

Define

\begin{align} \label{basis_equations_in_rigidity}
    Z_1 = \frac{Z}{|Z|},\;\;\; 
    Z_2 = \frac{Z'}{|Z'|},\;\;\; 
    X_1 = \frac{X}{|X|}, \;\;\;
    X_2 = \frac{X'}{|X'|}, \;\;\;
    X_3 = \frac{j(Z)X}{|j(Z)X|}, \;\;\;
    X_4 = \frac{j(Z')X}{|j(Z')X|}.
\end{align}

Let $\mathcal{N}_1=\mathbb{R}\textrm{-span}\{Z_1, Z_2, X_1, X_2, X_3, X_4\}$. We will show that $\mathcal{N}_1$ is the desired totally geodesic subalgebra isomorphic to $\mathrm{heis}_3(\mathbb{C})$ with a Ricci soliton inner product as metric Lie algebras. More specifically, we will see the following.
\begin{proposition} \label{proposition_key_for_rigidity} The following statements hold.
    \begin{enumerate}
    \item \label{j-map_of_key_proposition_for_rigidity} The $j$-map satisfies 
    \begin{align*}
    j(Z_1)X_1 &=  ||j||_{\text{op}}X_3, & j(Z_1)X_2 &=  ||j||_{\text{op}}X_4, & j(Z_1)X_3 &= -||j||_{\text{op}}X_1, & j(Z_1)X_4 &= -||j||_{\text{op}}X_2, \\
    j(Z_2)X_1 &=  ||j||_{\text{op}}X_4, & j(Z_2)X_2 &= -||j||_{\text{op}}X_3, & j(Z_2)X_3 &=  ||j||_{\text{op}}X_2, & j(Z_2)X_4 &= -||j||_{\text{op}}X_1.
\end{align*}
    \item \label{Lie_brackets_of_key_proposition_for_rigidity} $\mathcal{N}_1$ is closed under the Lie bracket, satisfying
    \begin{align*}
    [X_1, X_2] &= 0,                & [X_1, X_3] &= ||j||_{\text{op}} Z_1,  & [X_1, X_4] &= ||j||_{\text{op}} Z_2, \\
    [X_2, X_3] &= -||j||_{\text{op}} Z_2, & [X_2, X_4] &= ||j||_{\text{op}} Z_1,  & [X_3, X_4] &= 0.
\end{align*}
\end{enumerate}
\end{proposition}

\begin{proof}[Proof of Theorem \ref{maintheorem_rigidity} assuming Proposition \ref{proposition_key_for_rigidity}] 

Let $\mathcal{V}_1$ be the span of $X_1, ..., X_4$. Let $\mathcal{Z}_1$ be the span of $Z_1, Z_2$. The Lie bracket relations and the $j$-map given in Proposition \ref{proposition_key_for_rigidity} coincide with the ones given in Remark \ref{rescaling_heis_3_c_ricci_soliton} with $\xi=||j||_{\op}$. Hence, the subalgebra $\mathcal{N}_1$ is isomorphic to $\mathrm{heis}_3(\mathbb{C})$ with a Ricci soliton inner product as metric Lie algebras. In particular, $\mathcal{Z}_1$ is the center and $\mathcal{N}_1$ is an aligned subalgebra by design.

For any $z\in \mathcal{Z}_1$, $j(z)(\mathcal{V}_1)\subseteq \mathcal{V}_1$ by the item \ref{j-map_of_key_proposition_for_rigidity} of Proposition \ref{proposition_key_for_rigidity}, so it follows from Lemma \ref{criterion_of_totally_geodesic_when_it_is_aligned} that $\mathcal{N}_1$ is an aligned totally geodesic subalgebra of $\mathcal{N}$. By the discussion above, the result follows.
\end{proof}

We prove Proposition \ref{proposition_key_for_rigidity} by putting together all the conditions we gain from the inequalities (\ref{first_ineq_in_rigidity_of_maximum}), (\ref{second_ineq_in_rigidity_of_maximum}), (\ref{third_ineq_in_rigidity_of_maximum}), (\ref{fourth_ineq_in_rigidity_of_maximum}), (\ref{fifth_ineq_in_rigidity_of_maximum}), analyzing them in the backward order.

Using the inequality (\ref{fifth_ineq_in_rigidity_of_maximum}), we have the following lemma.

\begin{lemma} \label{lemma_consequence_of_fifth_ineq}
    \begin{enumerate}
        \item $|j(Z_1)X_2|=||j||_{\op}$ and $|j(Z_2)X_1|=||j||_{\op}$.
        \item $-j(Z_1)^2 X_2 = ||j||_{\op}^2 X_2$ and $[X_2, j(Z_1)X_2] =||j||_{\op}^2 Z_1$.
        \item $-j(Z_2)^2 X_1 =||j||_{\op}^2 X_1$ and $[X_1, j(Z_2)X_1]=||j||_{\op}^2 Z_2$.
        \item $X_4=\frac{1}{||j||_{\op}}j(Z_2)X_1$ and $[X_1, X_4]=||j||_{\op}Z_2$.
    \end{enumerate}
\end{lemma}

\begin{proof}
    The items 1 is a direct consequence of Lemma \ref{lemma_consequence_of_sixth_ineq}.  Then, the items 2 and 3 follow from Lemma \ref{maximum_eigenspace_and_two_interpretations}. The item 4 follows by definition of $X_4$.
\end{proof}

By the inequality (\ref{fourth_ineq_in_rigidity_of_maximum}), we show the following.

\begin{lemma}\label{lemma_consequence_of_fourth_ineq}
    \begin{enumerate}
        \item $j(Z')X=j(Z)X'$ and $j(Z_2)X_1=j(Z_1)X_2=||j||_{\op}X_4$
        \item $X_4=\frac{1}{||j||_{\op}}j(Z_1)X_2$ and $[X_2, X_4]=||j||_{\op}Z_1$.
        \item $j(Z_1)X_4=-||j||_{\op}X_2$ and $j(Z_2)X_4=-||j||_{\op}X_1$.
    \end{enumerate}
\end{lemma}

\begin{proof}
    We show the item 1. By the inequality (\ref{fourth_ineq_in_rigidity_of_maximum}), we have $\langle j(Z)X', j(Z')X\rangle = |j(Z)X'||j(Z')X|$, so two vectors $j(Z')X$ and $j(Z)X'$ are linearly dependent. We already know that  $|j(Z')X|=|j(Z)X'|=\frac{1}{2}||j||_{\op}^2$ from Lemma \ref{lemma_consequence_of_sixth_ineq}, so we have $j(Z')X=\pm j(Z)X'$. The equation above implies that the inner product of the two vectors is positive, so we have $j(Z')X=j(Z)X'$. Using Lemma \ref{lemma_consequence_of_sixth_ineq}, we have $j(Z_2)X_1=j(Z_1)X_2$. By the item 4 of Lemma \ref{lemma_consequence_of_fifth_ineq}, $j(Z_2)X_1=j(Z_1)X_2=||j||_{\op}X_4$. 
    
    We show the item 2. Using the item 2 of Lemma \ref{lemma_consequence_of_fifth_ineq}, we have $||j||_{\op}^2Z_1=[X_2, j(Z_1)X_2]=[X_2, ||j||_{\op}X_4]$. Thus, the item 2 follows. 
    
    By combining Lemma \ref{lemma_consequence_of_fifth_ineq} and the item 1 of this lemma, the item 3 follows.
\end{proof}
The inequality (\ref{third_ineq_in_rigidity_of_maximum}) implies the following.
\begin{lemma} \label{lemma_consequence_of_third_ineq}
    \begin{enumerate}
        \item $|j(Z_1)X_1|=||j||_{\op}$ and $|j(Z_2)X_2|=||j||_{\op}$.
        \item $-j(Z_1)^2 X_1 = ||j||_{\op}^2 X_1$ and $[X_1, j(Z_1)X_1] = ||j||_{\op}^2 Z_1$.
        \item $-j(Z_2)^2 X_2 = ||j||_{\op}^2 X_2$ and $[X_2, j(Z_2)X_2] = ||j||_{\op}^2 Z_2.$
        \item $X_3=\frac{1}{||j||_{\op}}j(Z_1)X_1$ and $[X_1, X_3] = ||j||_{\op}Z_1.$
    \end{enumerate}
\end{lemma}

\begin{proof}
    Lemma \ref{lemma_consequence_of_sixth_ineq} implies the item 1. Then the items 2 and 3 follow from Lemma \ref{maximum_eigenspace_and_two_interpretations}. The item 4 follows by definition of $X_3$.
\end{proof}

The equality for the second inequality (\ref{second_ineq_in_rigidity_of_maximum}) shows the following.

\begin{lemma}\label{lemma_consequence_of_second_ineq}
\begin{enumerate}
    \item $-j(Z')X'=j(Z)X$ and $-j(Z_2)X_2=j(Z_1)X_1=||j||_{\op}X_3$. 
    \item $[X_2, X_3]=-||j||_{\op} Z_2$.
    \item $j(Z_1)X_3=-||j||_{\op}X_1$ and $j(Z_2)X_3=||j||_{\op}X_2$.
\end{enumerate}
    
\end{lemma}

\begin{proof}
The equality for the second inequality (\ref{second_ineq_in_rigidity_of_maximum}) implies $-\langle j(Z)X, j(Z')X' \rangle =|j(Z)X||j(Z')X'|.$ By a similar argument to the proof of the item 1 of Lemma \ref{lemma_consequence_of_fourth_ineq}, the item 1 of this lemma follows. 

The item 2 follows from the item 3 of Lemma \ref{lemma_consequence_of_third_ineq} together with $j(Z_2)X_2=-||j||_{\op}X_3$ from the item 1 of this lemma. By combining Lemma \ref{lemma_consequence_of_third_ineq} and the item 1 of this lemma, the item 3 follows.
\end{proof}

The first inequality (\ref{first_ineq_in_rigidity_of_maximum}) implies:
\begin{lemma}\label{lemma_consequence_of_first_ineq}
    $[X_1, X_2]=0$.
\end{lemma}
This is because $-\frac{3}{4}|[X,X']|^2=0$ implies $[X,X']=0$.

\begin{lemma} \label{lemma_orthogonality}
    $X_1, X_2, X_3, X_4, Z_1, Z_2$ are orthonormal.
\end{lemma}
\begin{proof}
    By definition, each of the vectors has unit length. Since $X, X'$ are orthogonal, so are $X_1, X_2$. Similarly, the orthogonality of $Z, Z'$ implies the orthogonality of $Z_1, Z_2$. Since $\mathcal{V}\perp\mathcal{Z}$, we show that $X_1,X_2,Z_1,Z_2$ are orthogonal.
    
    Since $j(Z_1)$ and $j(Z_2)$ are skew-symmetric, $\langle X_1, X_3 \rangle=0$ and $\langle X_2, X_3 \rangle=0$ follow from the item 1 of Lemma \ref{lemma_consequence_of_second_ineq}. Similarly, $\langle X_1, X_4 \rangle=0, \langle X_2, X_4 \rangle=0$ by the item 1 of Lemma \ref{lemma_consequence_of_fourth_ineq}.

    We show that $\langle X_3, X_4\rangle=0$.
    Observe, $X_4=\frac{1}{||j||_{\op}}j(Z_1)X_2$ by the item 1 of Lemma \ref{lemma_consequence_of_fourth_ineq}, and $X_3=\frac{1}{||j||_{\op}}j(Z_1)X_1$ by the item 4 of Lemma \ref{lemma_consequence_of_third_ineq}. Using the item 2 of Lemma \ref{lemma_consequence_of_third_ineq}, we have $\langle X_3, X_4 \rangle=\frac{1}{||j||_{\op}^2}\langle -j(Z_1)^2 X_1, X_2\rangle=\langle X_1,X_2\rangle=0$. This completes the proof.
\end{proof}

We are ready to prove the item \ref{j-map_of_key_proposition_for_rigidity} of Proposition \ref{proposition_key_for_rigidity}.

\begin{proof}[Proof of the item \ref{j-map_of_key_proposition_for_rigidity} of Proposition \ref{proposition_key_for_rigidity}]
    These are shown in Lemma \ref{lemma_consequence_of_fifth_ineq}, Lemma \ref{lemma_consequence_of_fourth_ineq}, Lemma \ref{lemma_consequence_of_third_ineq}, and
    Lemma \ref{lemma_consequence_of_second_ineq}.
\end{proof} 

\begin{lemma} \label{first_five_brackets}
    The first 5 Lie bracket relations of Proposition \ref{Lie_brackets_of_key_proposition_for_rigidity} hold.
\end{lemma}
\begin{proof}
 This is due to Lemma \ref{lemma_consequence_of_first_ineq}, the item 4 of Lemma \ref{lemma_consequence_of_third_ineq}, the item 4 of Lemma \ref{lemma_consequence_of_fifth_ineq}, the item 2 of Lemma \ref{lemma_consequence_of_second_ineq}, and the item 2 of Lemma \ref{lemma_consequence_of_fourth_ineq}.
\end{proof}

Our goal is to show that $[X_3, X_4]=0$. However, this requires more work.

\begin{lemma}\label{lemma_last_bracket}
    $[X_3, X_4]=0$.
\end{lemma}

We give a proof of Lemma \ref{lemma_last_bracket} here. To begin with, we show:

\begin{sublemma} \label{sublemma_Z_3_is_orthogonal}
    $[X_3, X_4]\in \mathbb{R}\textrm{-span}\{Z_1, Z_2\}^{\perp}$, that is, $[X_3, X_4]$ is orthogonal to the subspace spanned by $Z_1, Z_2$. 
\end{sublemma}

\begin{proof}
    Recall the item \ref{j-map_of_key_proposition_for_rigidity} of Proposition \ref{proposition_key_for_rigidity} and Lemma \ref{lemma_orthogonality}. We have $\langle [X_3, X_4], Z_1 \rangle =\langle j(Z_1)X_3, X_4 \rangle$ by definition, and$\langle j(Z_1)X_3, X_4 \rangle=-||j||_{\op}\langle X_1, X_4 \rangle=0.$
    
    Similarly, $\langle [X_3, X_4], Z_2 \rangle =\langle j(Z_2)X_3, X_4 \rangle =||j||_{\op}\langle X_2, X_4 \rangle=0$.
\end{proof}

Now, we prove $[X_3, X_4]=0$ by contradiction. Suppose that $[X_3, X_4]\neq 0$, and set $Z_3 = \frac{[X_3, X_4]}{|[X_3, X_4]|}$. Let $\mathcal{N}_0$ be the subspace spanned by $X_1, X_2, X_3, X_4, Z_1, Z_2$ and $Z_3$.  By Lemma \ref{first_five_brackets}, $\mathcal{N}_0$ is a subalgebra of $\mathcal{N}$. Since $Z_3\in \mathcal{Z}\perp \mathcal{V}$, it follows from Lemma \ref{sublemma_Z_3_is_orthogonal} that $\{X_1, X_2, X_3, X_4, Z_1, Z_2, Z_3\}$ forms an orthonormal basis of $\mathcal{N}_0$. Extend $X_1,...,X_4$ to an orthonormal basis $X_1,...,X_q$ of $\mathcal{V}$ and $Z_1,Z_2,Z_3$ to $Z_1,...,Z_p$ of $\mathcal{Z}$.

Our strategy is to show that the norm $||j||_{\op}<|j(uZ_1+vZ_3)|$ for some $u,v$ with $u^2+v^2=1$, which would lead to a contradiction.

Write each $j(Z_i)$ as a block form $j(Z_i) = \begin{bmatrix}
        A_i & -^tB_i  \\
        B_i & C_i  \\
    \end{bmatrix}$, where the partition is with respect to $\mathbb{R}\textrm{-span}\{X_1,...,X_4\}$ and $\mathbb{R}\textrm{-span}\{X_5,...,X_q\}$, and $^tB_i$ is the transpose of $B_i$.

\begin{sublemma} \label{sublemma_matrices}
    \begin{align*}
    A_1&=\begin{bmatrix}
        0 & 0 & -||j||_{\op} & 0 \\
        0 & 0 & 0 & -||j||_{\op} \\
        ||j||_{\op} & 0 & 0 & 0\\
        0 & ||j||_{\op} & 0 & 0
    \end{bmatrix};
    A_2=\begin{bmatrix}
        0 & 0 & 0 & -||j||_{\op} \\
        0 & 0 & ||j||_{\op} & 0 \\
        0 & -||j||_{\op} & 0 & 0\\
        ||j||_{\op} & 0 & 0 & 0
    \end{bmatrix},\\
    A_3&=\begin{bmatrix}
        0 & 0 & 0 & 0 \\
        0 & 0 & 0 & 0 \\
        0 & 0 & 0 & -|[X_3,X_4]|\\
        0 & 0 & |[X_3,X_4]| & 0
    \end{bmatrix}; B_1=B_2=0.
\end{align*}
\end{sublemma}

\begin{proof}
     The item \ref{j-map_of_key_proposition_for_rigidity} of Proposition \ref{proposition_key_for_rigidity} computes $A_1$ and $A_2$ and shows $B_1=B_2=0$. Since $[X_3,X_4]=|[X_3,X_4]|Z_3$, it follows from Lemma \ref{first_five_brackets} that the only non-zero $\langle Z_3, [X_i, X_j]\rangle$ are $i=3, j=4$ or $i=4, j=3$. Thus, $A_3$ has the form above.
\end{proof}

Note, $B_3$ could be nonzero. Now, let $a=||j||_{\op}$ and $c=|[X_3,X_4]|$. Let $u,v\in \mathbb{R}$ with $u^2+v^2=1$. For any $x\in \mathbb{R}\textrm{-span}\{X_1,...,X_4\}$ with $|x|=1$, observe that $B_3x \in \mathbb{R}\textrm{-span}\{X_5, ..., X_q\}$ is orthogonal to $(uA_1+vA_3)x\in \mathbb{R}\textrm{-span}\{X_1, X_2, X_3, X_4\}$. Thus, we have $a^2=||j||_{\op}^2\geq |j(uZ_1+vZ_3)x|^2=|(uA_1+vA_3)x+vB_3x|^2=|(uA_1+vA_3)x|^2+|vB_3x|^2\geq |(uA_1+vA_3)x|^2.$ If $A=uA_1+vA_3$, then $a^2\geq |(uA_1+vA_3)x|^2=\langle -A^2 x, x\rangle$. 

\begin{sublemma}\label{sublemma_contradiction}
    There exists $u,v\in \mathbb{R}$ such that $u^2+v^2=1$ and a unit vector $x\in \mathbb{R}\textrm{-span}\{X_1, ..., X_4\}$ such that $|(uA_1+vA_3)x|^2=\langle -A^2 x, x\rangle > a^2$, where $A=uA_1+vA_3$, which leads to a contradiction.
\end{sublemma}
\begin{proof}
    Let $x=\frac{1}{\sqrt{2}}(X_2+X_3)$. By computations, one gets $\langle -A^2x,x\rangle=\frac{1}{2}( 2a^2 u^2 + 2acuv + c^2 v^2)$.  Let $u=\cos \theta$ and $v=\sin \theta$. Then, $\langle -A^2x,x\rangle=\frac{1}{2}(\frac{2a^2+c^2}{2} + \frac{2a^2-c^2}{2}\cos 2\theta +ac\sin 2\theta).$ There exists some $\theta'\in \mathbb{R}$ such that $\frac{2a^2-c^2}{2}\cos 2\theta +ac\sin 2\theta=\sqrt{\left(\frac{2a^2-c^2}{2}\right)^2+(ac)^2}\sin (2\theta + \theta')$. Pick $\theta$ so that $\sin (2\theta + \theta')=1$. Note that $c\leq |j(Z_3)X_3|\leq a$ so $2a^2-c^2>0$. Then, $a^2\geq \langle -A^2x,x\rangle=\frac{1}{2}\left(\frac{2a^2+c^2}{2} +\sqrt{\left(\frac{2a^2-c^2}{2}\right)^2+(ac)^2}\right)>\frac{1}{2}(\frac{2a^2+c^2}{2}+\frac{2a^2-c^2}{2})=a^2$ as we have $ac>0$. This is a contradiction. 
\end{proof}
Therefore, $c=0$, and $[X_3, X_4]=0.$ This proves Lemma \ref{lemma_last_bracket}. Together with Lemma \ref{first_five_brackets}, this finishes the proof of Proposition \ref{proposition_key_for_rigidity}, and Theorem \ref{maintheorem_rigidity} follows.

\subsection{Proof of Theorem \ref{maintheorem_characterization}} \label{section_characterization}

If a 2-step nilpotent Lie algebra $\mathcal{N}$ admits an inner product with $\frac{K_{\min}}{K_{\max}}=-\frac{3}{2}$, then $\mathcal{N}$ admits a subalgebra isomorphic to $\mathrm{heis}_3(\mathbb{C})$ by Theorem \ref{maintheorem_rigidity}. The following example shows the converse is not true.

\begin{example} \label{almost_heis_C}
    We construct an example of a 2-step nilpotent Lie algebra $\mathcal{N}$ that admits a subalgebra isomorphic to $\mathrm{heis}_3(\mathbb{C})$ such that $\mathcal{N}$ does not admit an inner product with $\frac{K_{\min}}{K_{\max}}=-\frac{3}{2}$.
    
    Let $\mathcal{V}$ be an inner product space with an orthonormal basis $\{X_1,X_2,X_3,X_4,X_5\}$. Let $\mathcal{Z}$ be an inner product space with an orthonormal basis $\{Z_1,Z_2\}$. Define a metric 2-step nilpotent Lie algebra $\mathcal{N}=\mathcal{V}\oplus\mathcal{Z}$ by declaring that $j(Z_1) =\begin{bmatrix}
            A_1 & -^te_1 \\
            e_1 & 0 
        \end{bmatrix}, 
    j(Z_2) = \begin{bmatrix}
        A_2 & 0 \\
        0 & 0
    \end{bmatrix}$
    where $A_1, A_2$ are the $j$-maps of $\mathrm{heis}_3(\mathbb{C})$ given in Example \ref{example_heis_3_C_with_Ricci_soliton_metric} and $e_1=[1,0,0,0,0]$. We have the following Lie bracket relations: $[X_1, X_2]=0, [X_1, X_3]=Z_1, [X_1, X_4]=Z_2, [X_1, X_5]=Z_1, [X_2, X_3]=-Z_2, [X_2, X_4]=Z_1, [X_2, X_5]=0, [X_3, X_4]=0, [X_3,X_5]=0, [X_4,X_5]=0$. Thus, the center is $\mathcal{Z}$, and $[\mathcal{N}, \mathcal{N}]=\mathcal{Z}$.
    
     In particular, $\mathbb{R}\textrm{-span}\{X_1,X_2,X_3,X_4,Z_1,Z_2\}$ is a subalgebra isomorphic to $\mathrm{heis}_3(\mathbb{C})$. If it admits an inner product with $\frac{K_{\min}}{K_{\max}}=-\frac{3}{2}$, then it follows from Proposition \ref{proposition_ridigity_in_dim_7} below that $[\mathcal{N}, \mathcal{N}]\neq \mathcal{Z}$, which is a contradiction. Thus, $\mathcal{N}$ does not admit an inner product with $\frac{K_{\min}}{K_{\max}}=-\frac{3}{2}$.

\end{example}

\begin{proposition} \label{proposition_ridigity_in_dim_7}
    If a 7-dimensional 2-step nilpotent Lie algebra $\mathcal{N}$ admits an inner product $\langle \cdot, \cdot \rangle$ with $\frac{K_{\min}}{K_{\max}}=-\frac{3}{2}$, then $(\mathcal{N}, \langle \cdot, \cdot\rangle)$ is isomorphic to the orthogonal direct sum of Lie algebras $\mathrm{heis}_3(\mathbb{C})\oplus \mathbb{R}$ as metric Lie algebras, where $\mathrm{heis}_3(\mathbb{C})$ be equipped with a Ricci soliton inner product. 
\end{proposition}
\begin{proof}
By Theorem \ref{maintheorem_rigidity}, there exists a totally geodesic subalgebra $\mathcal{N}_1$ isomorphic to $\mathrm{heis}_3(\mathbb{C})$ with a Ricci soliton metric as metric Lie algebras. Let $\mathcal{Z}$ be the center of $\mathcal{N}$, and $\mathcal{V}=\mathcal{Z}^{\perp}$. Since $\mathcal{N}_1$ is aligned by Proposition \ref{nonsingular_ones_are_aligned}, $\mathcal{N}_1=\mathcal{V}_1\oplus \mathcal{Z}_1$ for some subspace $\mathcal{V}_1\subseteq \mathcal{V}$ and $\mathcal{Z}_1\subseteq \mathcal{Z}$, where $\mathcal{Z}_1$ is the center of $\mathcal{N}_1$. Let $\mathcal{V}_2=\mathcal{V}_1^{\perp}\cap \mathcal{V}$ and $\mathcal{Z}_2=\mathcal{Z}_1^{\perp}\cap \mathcal{Z}$.

Suppose $\mathcal{V}_2\neq 0$. Then $\dim \mathcal{V}_2=1$ and $\mathcal{Z}_2=0$ by $\dim \mathcal{N}=7$. Let $x_0$ be a vector spanning $\mathcal{V}_2$. For any $z\in \mathcal{Z}_1=\mathcal{Z}$, $x\in \mathcal{V}_1$, $\langle [x, x_0], z\rangle=\langle j(z)x, x_0\rangle=0$ as $j(z)x\in \mathcal{V}_1$ by Lemma \ref{criterion_of_totally_geodesic_when_it_is_aligned}. Thus, $x_0$ is central but not in $\mathcal{Z}$, which is a contradiction. Thus, $\mathcal{V}_2=0$, so $\dim \mathcal{Z}_2=1$. Since $[\mathcal{N}_1, \mathcal{Z}_2]=0$ and $\mathcal{N}_1\perp \mathcal{Z}_2$, $\mathcal{N}=\mathcal{N}_1\oplus \mathcal{Z}_2$ is the desired orthogonal direct sum of Lie algebras.
\end{proof}

We give a proof of Theorem \ref{maintheorem_characterization}.

\subsubsection{Proof of the equivalence of the items 2 and 3 of Theorem \ref{maintheorem_characterization}.}

We split the proof into two lemmas.
\begin{lemma}
    The item 2 implies the item 3.
\end{lemma}
\begin{proof}
    At the Lie algebra level, $\mathcal{N}$ admits a subalgebra $\mathcal{N}_1$ and an inner product $\langle \cdot, \cdot \rangle$ such that $\mathcal{N}_1$ is a totally geodesic subalgebra isomorphic to $\mathrm{heis}_3(\mathbb{C})$ with a Ricci soliton inner product as metric Lie algebras. Since $\mathrm{heis}_3(\mathbb{C})$ is nonsingular, it follows from Proposition \ref{nonsingular_ones_are_aligned} that $\mathcal{N}_1$ is an aligned totally geodesic subalgebra. Let $\mathcal{V}$ be the orthogonal complement of the center $\mathcal{Z}$ of $\mathcal{N}$ so that $\mathcal{N}=\mathcal{V}\oplus \mathcal{Z}$.

    Since $\mathcal{N}_1$ is an aligned totally geodesic subalgebra, there exist subspaces $\mathcal{V}_1$ of $\mathcal{V}$ and $\mathcal{Z}_1$ of $\mathcal{Z}$ such that $\mathcal{N}_1=\mathcal{V}_1\oplus \mathcal{Z}_1$ and $\mathcal{Z}_1$ is the center of $\mathcal{N}_1$.  Let $\mathcal{V}_2=\mathcal{V}_1^{\perp} \cap \mathcal{V}$, $\mathcal{Z}_2 =\mathcal{Z}_1^{\perp} \cap \mathcal{Z}$, and $\mathcal{N}_2=\mathcal{V}_2\oplus \mathcal{Z}_2$. 

    We have $[\mathcal{N}_1, \mathcal{N}_2]\subseteq [\mathcal{V}_1, \mathcal{V}_2]$. For any $z\in \mathcal{Z}_1, x\in \mathcal{V}_1, y\in \mathcal{V}_2$, we have $\langle z, [x,y]\rangle = \langle j(z)x, y \rangle=0$ as $j(z)x\in \mathcal{V}_1$ by Lemma \ref{criterion_of_totally_geodesic_when_it_is_aligned}. Hence, $[\mathcal{V}_1, \mathcal{V}_2]\subseteq \mathcal{Z}_2\subseteq \mathcal{N}_2$.
\end{proof}

\begin{lemma}
    The item 3 implies the item 2.
\end{lemma}

\begin{proof}
    Let $\varphi:\mathcal{N}_1\to \mathrm{heis}_3(\mathbb{C})$ be the Lie algebra isomorphism from the item (b) of the item 3. Let $\mathrm{heis}_3(\mathbb{C})$ be equipped with a Ricci soliton inner product. Let $\langle \cdot, \cdot \rangle_1$ be the pullback of a Ricci soliton inner product on $\mathrm{heis}_3(\mathbb{C})$ so that $\varphi$ is a metric Lie algebra isomorphism. Let $\mathcal{Z}_1=\varphi^{-1}(\mathfrak{z}(\mathrm{heis}_3(\mathbb{C})))$ and $\mathcal{V}_1=\varphi^{-1}(\mathfrak{z}(\mathrm{heis}_3(\mathbb{C}))^{\perp})$, where $\mathfrak{z}(\mathrm{heis}_3(\mathbb{C}))$ is the center of $\mathrm{heis}_3(\mathbb{C})$. Note, $\mathcal{Z}_1=[\mathcal{N}_1, \mathcal{N}_1]$, and $\mathcal{N}_1=\mathcal{V}_1\oplus \mathcal{Z}_1$. Let $\mathcal{Z}$ be the center of $\mathcal{N}$.
    
    Let $\mathcal{Z}_2=\mathcal{Z}\cap \mathcal{N}_2$. Since $\mathcal{Z}_2$ is a subspace of $\mathcal{N}_2$, there exists a subspace $\mathcal{V}_2$ of $\mathcal{N}_2$ such that $\mathcal{N}_2=\mathcal{V}_2\oplus \mathcal{Z}_2$. Since $\mathcal{N}=\mathcal{N}_1\oplus \mathcal{N}_2$ as a vector space, we have $\mathcal{N}=\mathcal{Z}_1\oplus \mathcal{Z}_2 \oplus \mathcal{V}_1 \oplus \mathcal{V}_2$. Take any inner product $\langle \cdot, \cdot\rangle$ on $\mathcal{N}$ such that $\mathcal{Z}_1, \mathcal{Z}_2, \mathcal{V}_1$, and $\mathcal{V}_2$ are orthogonal, and the restriction of $\langle \cdot, \cdot \rangle$ on $\mathcal{N}_1=\mathcal{V}_1\oplus \mathcal{Z}_1$ is $\langle \cdot, \cdot \rangle_1$.

    We prove that $\mathcal{Z}_1\oplus \mathcal{Z}_2=\mathcal{Z}$, i.e., the center of $\mathcal{N}$. Observe that $\mathcal{Z}_2\subseteq \mathcal{Z}$ by definition. Since the center and the commutator ideal coincides for $\mathrm{heis}_3(\mathbb{C})$, we have $\mathcal{Z}_1=[\mathcal{N}_1, \mathcal{N}_1]\subseteq [\mathcal{N}, \mathcal{N}]\subseteq \mathcal{Z}$. Here, we used the fact that $\mathcal{N}$ is 2-step nilpotent. Thus, we have $\mathcal{Z}_1\oplus \mathcal{Z}_2 \subseteq \mathcal{Z}$. 

    To show $\mathcal{Z}\subseteq \mathcal{Z}_1\oplus \mathcal{Z}_2$, let $x\in \mathcal{Z}$. Since $x\in \mathcal{N}=\mathcal{N}_1\oplus \mathcal{N}_2$, we can write $x=x_1+x_2$, where $x_1\in \mathcal{N}_1$ and $x_2\in \mathcal{N}_2$. For any $y\in \mathcal{N}_1$, since $x$ is in the center of $\mathcal{N}$, we have $0=[x,y]=[x_1,y]+[x_2,y]$. Since $\mathcal{N}_1$ is a subalgebra, we have $[x_1, y]\in \mathcal{N}_1$. By $[\mathcal{N}_1, \mathcal{N}_2]\subseteq \mathcal{N}_2$, we have $[x_2, y]\in \mathcal{N}_2$. Since $\mathcal{N}=\mathcal{N}_1\oplus \mathcal{N}_2$ is a direct sum, one obtains that $[x_1, y]=0=[x_2,y]$. In particular, $x_1$ is in the center of $\mathcal{N}_1$, i.e., $x_1\in \mathcal{Z}_1$. Then, $x_2=x-x_1 \in \mathcal{Z}$ because $\mathcal{Z}_1\subseteq \mathcal{Z}$ as we observed above. On the other hand, $x_2\in \mathcal{N}_2$, so $x_2\in \mathcal{N}_2\cap \mathcal{Z}=\mathcal{Z}_2$ by definition of $\mathcal{Z}_2$. Thus, $x=x_1+x_2\in \mathcal{Z}_1\oplus \mathcal{Z}_2$. This shows $\mathcal{Z}=\mathcal{Z}_1\oplus \mathcal{Z}_2$.

    Set $\mathcal{V}:=\mathcal{Z}^{\perp}=\mathcal{V}_1\oplus \mathcal{V}_2$. We have $[\mathcal{V}_1, \mathcal{V}_2]\subseteq [\mathcal{N}_1, \mathcal{N}_2] \subseteq \mathcal{N}_2\cap [\mathcal{N}, \mathcal{N}]\subseteq \mathcal{N}_2\cap \mathcal{Z}=\mathcal{Z}_2$. Thus, for any $z\in \mathcal{Z}_1$, $j(z)(\mathcal{V}_1)\subseteq \mathcal{V}_1$. By definition, $\mathcal{N}_1$ is an aligned subalgebra, so by Lemma \ref{criterion_of_totally_geodesic_when_it_is_aligned}, $\mathcal{N}_1$ is an aligned totally geodesic subalgebra isomorphic to $\mathrm{heis}_3(\mathbb{C})$ with a Ricci soliton inner product as metric Lie algebras. 
\end{proof}

\subsubsection{Proof of the equivalence of the items 1 and 2 of Theorem \ref{maintheorem_characterization}} Let  $(\mathcal{N}, [\cdot,\cdot], \langle\cdot, \cdot\rangle)$ be the metric Lie algebra of $(N, \langle \cdot, \cdot \rangle)$ of the item 2. Let $\mathcal{Z}$ be the center of $\mathcal{N}$, and let $\mathcal{V}=\mathcal{Z}^{\perp}$ be the orthogonal complement of $\mathcal{Z}$. Let $N_1$ be the totally geodesic subgroup isomorphic and isometric to $\mathrm{Heis}_3(\mathbb{C})$ with a Ricci soliton metric. Let $\mathcal{N}_1=\mathrm{Lie} N_1 \subseteq \mathcal{N}$. First, $\mathcal{N}_1$ is nonsingular. By Proposition \ref{nonsingular_ones_are_aligned}, $\mathcal{N}_1$ is an aligned totally geodesic subalgebra. Hence, there are subspaces $\mathcal{Z}_1$ of $\mathcal{Z}$ and $\mathcal{V}_1$ of $\mathcal{V}$ such that $\mathcal{N}_1=\mathcal{V}_1\oplus \mathcal{Z}_1$ and $\mathcal{Z}_1$ is the center of $\mathcal{N}_1$. Let $\mathcal{V}_2=\mathcal{V}_1^{\perp}\cap \mathcal{V}$ be the orthogonal complement of $\mathcal{V}_1$ inside $\mathcal{V}$. Let $\mathcal{Z}_2=\mathcal{Z}_1^{\perp}\cap \mathcal{Z}$ be the orthogonal complement of $\mathcal{Z}_1$ inside $\mathcal{Z}$. For each $\lambda>0$, define $\varphi \in \mathrm{GL}(\mathcal{N})$ by setting 
\begin{align*}
    \varphi&=\lambda \mathrm{Id} \text{ on $\mathcal{V}_2$}, \;\varphi=\mathrm{Id} \text{ on $\mathcal{V}_2^{\perp}=\mathcal{V}_1\oplus \mathcal{Z}$.}  
\end{align*}
\begin{proposition} \label{proposition_key_for_characterization}
    For each $\lambda>0$, define a Lie bracket $[\cdot,\cdot]^{(\lambda)}:=\varphi^{-1}[\varphi\cdot, \varphi\cdot]$ on $\mathcal{N}$. Then, the metric 2-step nilpotent Lie algebra $\mathcal{N}^{(\lambda)}:=(\mathcal{N}, [\cdot,\cdot]^{(\lambda)},\langle\cdot,\cdot\rangle)$ satisfies the following statements.
    \begin{enumerate}
        \item \label{item_curvature_operator_prop_for_char} There exists a constant $\xi>0$ such that for $\lambda>0$ sufficiently small, $||j^{(\lambda)}||_{\op}=\xi$, where $j^{(\lambda)}$ denotes the $j$-map of $\mathcal{N}^{(\lambda)}$.
        \item \label{item_sectional_curvature_prop_for_char} There exists a 2-plane $\pi\subseteq \mathcal{N}_1$ such that for any $\lambda>0$, $K^{(\lambda)}(\pi)=\frac{1}{2}\xi^2$, where $K^{(\lambda)}$ denotes the sectional curvature of $\mathcal{N}^{(\lambda)}$.
        \item \label{item_pinching_constant_for_char} For $\lambda>0$ sufficiently small, $K_{\max}=\frac{1}{2}||j^{(\lambda)}||_{\op}^2, K_{\min}=-\frac{3}{4}||j^{(\lambda)}||_{\op}^2$, and   $\frac{K_{\min}}{K_{\max}}=-\frac{3}{2}$.
    \end{enumerate}
\end{proposition}

\begin{proof}[Proof of Theorem \ref{maintheorem_characterization} assuming Proposition \ref{proposition_key_for_characterization}]

We showed the equivalence of the items 2 and 3 above, and the item 1 implies the item 2 by Theorem \ref{maintheorem_rigidity}. We will assume the item 2 and show the item 1. Let $\langle \cdot, \cdot \rangle^{(\lambda)}:=\langle \varphi^{-1}\cdot, \varphi^{-1}\cdot \rangle$. Then $\varphi:\mathcal{N}^{(\lambda)}\to (\mathcal{N}, [\cdot, \cdot], \langle \cdot, \cdot \rangle^{(\lambda)})$ is an isomorphism of metric Lie algebras. Thus, for $\lambda>0$ sufficiently small, the corresponding left-invariant metric on $N$ has $\frac{K_{\min}}{K_{\max}}=-\frac{3}{2}$.
\end{proof}

\begin{remark}
    As $\mathcal{N}^{(1)}=\mathcal{N}$,  the degeneration $\lambda\to 0$ is a degeneration of $\mathcal{N}$ to $\mathrm{heis}_3(\mathbb{C})\oplus \mathbb{R}^{\dim \mathcal{N}-6}$.
\end{remark}

Thus, it suffices to show Proposition \ref{proposition_key_for_characterization}. We begin the proof of Proposition \ref{proposition_key_for_characterization} here. 
     Let $j$ denote the $j$-map of $\mathcal{N}$. By Lemma \ref{criterion_of_totally_geodesic_when_it_is_aligned}, $\mathcal{V}_1$ is invariant under $j(z)$ for any $z\in \mathcal{Z}_1$, and the restriction $j(z)|_{\mathcal{V}_1}$ is the $j$-map of $\mathcal{N}_1$. 

    A Ricci soliton inner product on any nilpotent Lie algebra is unique up to isometry and scaling (\cite[Theorem 3.5]{Lauret_Ricci_Soliton_Nilmanifolds}). Thus, one can find orthonormal bases $X_1, ..., X_4$ and $Z_1, Z_2$ of $\mathcal{V}_1$ and $\mathcal{Z}_1$, respectively, such that 
    \begin{align}\label{matrices_heis_3_C_for_characterizing_maximum}
        j(Z_1)|_{\mathcal{V}_1}=A_1:=\begin{bmatrix}
        0 & 0 & -\xi & 0 \\
        0 & 0 & 0 & -\xi \\
        \xi & 0 & 0 & 0\\
        0 & \xi & 0 & 0
    \end{bmatrix}, 
    j(Z_2)|_{\mathcal{V}_1}=A_2:=\begin{bmatrix}
        0 & 0 & 0 & -\xi \\
        0 & 0 & \xi & 0 \\
        0 & -\xi & 0 & 0\\
        \xi & 0 & 0 & 0
    \end{bmatrix}
    \end{align}
    with respect to $X_1,...,X_4$. The constant $\xi$ is due to scaling (see Remark \ref{rescaling_heis_3_c_ricci_soliton}).

    Let $X_5,...,X_q$ and $Z_3,...,Z_p$ be orthonormal bases of $\mathcal{V}_2$ and $\mathcal{Z}_2$, respectively.     Let $e_1=X_1, ..., e_q=X_q$ and $e_{q+1}=Z_1, ..., e_{q+p}=Z_p$, so that $e_1,...,e_{p+q}$ form an orthonormal basis of $\mathcal{N}$. Then, $e_a\wedge e_b\wedge e_c\wedge e_d$, $a<b<c<d$, form an orthonormal basis of $\Lambda^4 \mathcal{N}$. Form the dual basis $(e_a\wedge e_b\wedge e_c\wedge e_d)^*$. Then, Setting \ref{setting_totally_geodesic_T} is satisfied.

    Recall that if $\langle \cdot, \cdot \rangle^{(\lambda)}:=\langle \varphi^{-1}\cdot, \varphi^{-1}\cdot \rangle$, then $\varphi:\mathcal{N}^{(\lambda)}\to (\mathcal{N}, [\cdot, \cdot], \langle \cdot, \cdot \rangle^{(\lambda)})$ is an isomorphism of metric Lie algebras. Let $\mathcal{N}_1^{(\lambda)}=(\mathcal{N}_1, [\cdot, \cdot]^{(\lambda)}, \langle \cdot, \cdot \rangle)=\varphi^{-1}(\mathcal{N}_1, [\cdot,\cdot], \langle \cdot, \cdot\rangle_{\lambda})$ denote the subalgebra $\mathcal{N}_1$ with the Lie bracket induced from $[\cdot,\cdot]^{(\lambda)}$. Let $j^{(\lambda)}$ be the $j$-map of $(\mathcal{N}, [\cdot,\cdot]^{(\lambda)}, \langle \cdot, \cdot\rangle)$.

\begin{lemma} \label{lemma_N_1_lambda_is_N_1}
    The subalgebra $\mathcal{N}_1^{(\lambda)}=(\mathcal{N}_1, [\cdot, \cdot]^{(\lambda)}, \langle \cdot, \cdot \rangle)$ coincides with $(\mathcal{N}_1, [\cdot, \cdot], \langle \cdot, \cdot \rangle)$ as $\varphi$ is the identity on $\mathcal{N}_1$. In particular, $\mathcal{N}_1^{(\lambda)}$ is an aligned subalgebra of $\mathcal{N}^{(\lambda)}=(\mathcal{N}, [\cdot, \cdot]^{(\lambda)}, \langle \cdot, \cdot \rangle)$ isomorphic to $\mathrm{heis}_3(\mathbb{C})$ with a Ricci soliton inner product as metric Lie algebras. The center of $\mathcal{N}^{(\lambda)}$ is the subspace $\mathcal{Z}$, and the orthogonal complement of the center of $\mathcal{N}^{(\lambda)}$ is the subspace $\mathcal{V}$.

\end{lemma}
\begin{proof} 
This is because $\varphi$ is a metric Lie algebra isomorphism, and the statements are true for $(\mathcal{N}, [\cdot,\cdot], \langle \cdot,\cdot \rangle^{(\lambda)})$.
\end{proof}

For each $\alpha\in \{1,...,p\}$, we write $j(Z_\alpha)$ and $j^{(\lambda)}(Z_\alpha)$ in 2-by-2 block form relative to the subspaces $\mathcal{V}_1$ and $\mathcal{V}_2$.

\begin{lemma} \label{matrix_representation_of_j_lambda}
    For $\alpha\in\{1,2\}, \beta\in \{3,...,p\}$, write $j(Z_\alpha)=\begin{bmatrix}
        A_\alpha & 0 \\
        0 & D_\alpha
    \end{bmatrix}$, $j(Z_\beta)=\begin{bmatrix}
        0 & -B_{\beta}^T \\
        B_\beta & D_\beta
    \end{bmatrix}$ relative to the subspaces $\mathcal{V}_1$ and $\mathcal{V}_2$. Then, for $\alpha \in \{1,2\}, \beta\in \{3,...,p\}$, we have $j^{(\lambda)}(Z_\alpha)=\begin{bmatrix}
        A_\alpha & 0 \\
        0 & \lambda^2 D_\alpha
    \end{bmatrix}$, 
    $j^{(\lambda)}(Z_\beta)=\begin{bmatrix}
        0 & -\lambda B_{\beta}^T \\
        \lambda B_\beta & \lambda^2 D_\beta 
    \end{bmatrix}$
    relative to the subspaces $\mathcal{V}_1$ and $\mathcal{V}_2$. The $j$-map of $\mathcal{N}_1^{(\lambda)}$ is given by $j(Z_\alpha)|_{\mathcal{V}_1}=A_\alpha$, $\alpha\in \{1,2\}$. 

    In particular, $\mathcal{N}_1^{(\lambda)}$ is an aligned totally geodesic subalgebra. With respect to the orthonormal basis $X_1,...,X_q$ of $\mathcal{V}$ and $Z_1,...,Z_p$ of $\mathcal{Z}$, Setting \ref{setting_totally_geodesic_T} is satisfied for $\mathcal{N}^{(\lambda)}$ and $\mathcal{N}_1^{(\lambda)}$.
\end{lemma}
\begin{proof}
    Since $\mathcal{N}_1$ is an aligned totally geodesic subalgebra, Lemma \ref{curv_op_and_totally_geodesic_subalgebras} explains the block form of $j(Z_\alpha)$ and $j(Z_\beta)$. Recall that $\varphi$ is the identity on $\mathcal{Z}$, and $[\mathcal{N},\mathcal{N}]\subseteq \mathcal{Z}$. For any $x,y\in \mathcal{V}$, $[x,y]^{(\lambda)}=\varphi^{-1}([\varphi(x),\varphi(y)])=[\varphi(x),\varphi(y)]$. Use this to get the block form of $j^{(\lambda)}(Z_\alpha)$ and $j^{(\lambda)}(Z_\beta)$. Then Lemma \ref{criterion_of_totally_geodesic_when_it_is_aligned} implies that $\mathcal{N}_1^{(\lambda)}$ is an aligned totally geodesic subalgebra.
\end{proof}

Let $R_\lambda$ be the curvature operator of $\mathcal{N}^{(\lambda)}$. Recall from Section \ref{section_curv_op_of_general_riem_mfd} that any element $T\in \Lambda^4(\mathcal{N}^{(\lambda)})^*$ gives rise to a modified curvature operator $R_{\lambda}+\iota(T)$, which is a symmetric operator, and $\langle\iota(T)(x\wedge y),z\wedge w\rangle=T(x\wedge y\wedge z \wedge w)$ by definition. 

We will show that $\xi^2=||j^{(\lambda)}||_{\op}^2$ for a sufficiently small $\lambda>0$. For $i\in \{1,...,4\}, j\in \{5,...,q\}, \alpha \in \{1,2\}, \beta\in \{3,...,p\}$, set $c_{i\alpha j\beta}(\lambda):=\langle R_\lambda(X_i\wedge Z_\alpha), X_j\wedge Z_\beta \rangle$. Define $T_{\lambda}=\sum c_{i\alpha j\beta}(\lambda) (X_i  \wedge X_j\wedge Z_\alpha \wedge Z_\beta)^*$, where the sum is over the indices $i\in \{1,...,4\}, \alpha\in \{1,2\}, j\in \{5,...,q\}, \beta\in \{3,...,p\}$. Let $T_{\lambda}'=-\frac{\xi^2}{4}(X_1\wedge Z_1\wedge X_2\wedge Z_2)^*+\frac{\xi^2}{4}(X_3\wedge Z_1\wedge X_4\wedge Z_2)^*$. (This $T_\lambda'$ is independent of $\lambda$, so we could write $T_\lambda'=T'$.) Let $S_{\lambda}=T_\lambda+T_{\lambda}'$.

\begin{lemma} \label{lemma_existence_of_appropriate_T} We have the following.   
    \begin{enumerate}
        \item $R_\lambda$ and $\iota(T_\lambda)$ leaves $\mathcal{V}\otimes \mathcal{Z}$ invariant, and the modified curvature operator $R_\lambda + \iota(T_{\lambda})$ leaves $\mathcal{V}\otimes \mathcal{Z}$, $\mathcal{V}_1\otimes \mathcal{Z}_1$, and $(\mathcal{V}_1\otimes \mathcal{Z}_1)^{\perp}\cap (\mathcal{V}\otimes \mathcal{Z})$ invariant. Moreover, $T_\lambda\to 0$ as $\lambda\to 0$.
        \item $\iota(T_{\lambda}')$ leaves $\mathcal{V}\otimes \mathcal{Z}$, $\mathcal{V}_1\otimes \mathcal{Z}_1$, $(\mathcal{V}_1\otimes \mathcal{Z}_1)^{\perp}\cap (\mathcal{V}\otimes \mathcal{Z})$ invariant, and $\iota(T_{\lambda}')=0$ on $(\mathcal{V}_1\otimes \mathcal{Z}_1)^{\perp}\cap (\mathcal{V}\otimes \mathcal{Z})$.
        \item $R_\lambda+\iota(S_\lambda)=R_\lambda+\iota(T_\lambda)+\iota(T_{\lambda}')$ leaves $\mathcal{V}\otimes \mathcal{Z}$, $\mathcal{V}_1\otimes \mathcal{Z}_1$, and $(\mathcal{V}_1\otimes \mathcal{Z}_1)^{\perp}\cap (\mathcal{V}\otimes \mathcal{Z})$ invariant.
        \item $(R_\lambda+\iota(S_{\lambda}))|_{\mathcal{V}_1\otimes \mathcal{Z}_1}=\frac{\xi^2}{4}\mathrm{id}|_{\mathcal{V}_1\otimes \mathcal{Z}_1}$.
        \item $(R_\lambda+\iota(S_{\lambda}))|_{(\mathcal{V}_1\otimes \mathcal{Z}_1)^{\perp}\cap (\mathcal{V}\otimes \mathcal{Z})}=(R_\lambda+\iota(T_\lambda))|_{(\mathcal{V}_1\otimes \mathcal{Z}_1)^{\perp}\cap (\mathcal{V}\otimes \mathcal{Z})}\to 0$ as $\lambda\to 0$.
    \end{enumerate}
    In particular, for $\lambda>0$ sufficiently small, the maximum eigenvalue of $(R_\lambda + \iota(S_{\lambda}))|_{\mathcal{V}\otimes \mathcal{Z}}$ is $\frac{\xi^2}{4}$.
\end{lemma}

\begin{proof} We show the item 1. By Lemma \ref{formulas_for_curv_op} and Lemma \ref{matrix_representation_of_j_lambda}, we have 
\begin{align*}
    c_{i\alpha j\beta}(\lambda)= \frac{1}{4}\langle j^{(\lambda)}(Z_\alpha)X_j, j^{(\lambda)}(Z_\beta)X_i\rangle= \frac{1}{4}\langle \lambda^2 j(Z_\alpha)X_j, \lambda j(Z_\beta)X_i\rangle \to 0
\end{align*}
as $\lambda\to 0$. Hence, $T_{\lambda}\to 0$ as $\lambda\to 0$. By Lemma \ref{matrix_representation_of_j_lambda}, Setting \ref{setting_totally_geodesic_T} is satisfied and we can use Lemma \ref{curv_op_and_totally_geodesic_subalgebras}. The $T_\lambda$ defined above corresponds to the one in Lemma \ref{curv_op_and_totally_geodesic_subalgebras}. The item 1 follows from Lemma \ref{curv_op_and_totally_geodesic_subalgebras}.

To show the items 2, 3, and 4, we make the following observation. For $i, i'\in \{1,..,4\}$ and $\alpha, \alpha'\in \{1,2\}$, we have $\langle (R_\lambda+\iota(T_\lambda))(X_i\wedge Z_\alpha), X_{i'}\wedge Z_{\alpha'}\rangle=\langle R_\lambda(X_i\wedge Z_\alpha), X_{i'}\wedge Z_{\alpha'}\rangle$ by Lemma \ref{curv_op_and_totally_geodesic_subalgebras}. By Lemma \ref{formulas_for_curv_op} and Lemma \ref{matrix_representation_of_j_lambda}, we have $\langle R_\lambda(X_i\wedge Z_\alpha), X_{i'}\wedge Z_{\alpha'}\rangle=\frac{1}{4}\langle j^{(\lambda)}(Z_{\alpha'})X_{i}, j^{(\lambda)}(Z_{\alpha})X_{i'}\rangle=\frac{1}{4}\langle j(Z_{\alpha'})X_{i}, j(Z_{\alpha})X_{i'}\rangle$, which is exactly the entries of the curvature operator of $\mathrm{heis}_3(\mathbb{C})$ with respect to a Ricci soliton inner product. By this fact, the matrix representation of $(R_\lambda+\iota(T_\lambda))|_{\mathcal{V}_1\otimes \mathcal{Z}_1}$ with respect to $X_1\wedge Z_1, ..., X_4 \wedge Z_1, X_1,\wedge Z_2, ..., X_4\wedge Z_2$ is given by the 8 by 8 matrix given in Remark \ref{rescaling_heis_3_c_ricci_soliton}, which is $(R_\lambda+\iota(T_\lambda))|_{\mathcal{V}_1\otimes \mathcal{Z}_1}= \begin{bmatrix}
           -\frac{1}{4}A_1^2 & -\frac{1}{4} A_2 A_1 \\
           -\frac{1}{4}A_1 A_2 & -\frac{1}{4} A_2^2
       \end{bmatrix}$.

Observe that $\iota(T_{\lambda}')(\mathcal{V}\otimes \mathcal{Z})\subseteq \mathcal{V}\otimes \mathcal{Z}$. By definition, we see $\iota(T_{\lambda}')(\mathcal{V}_1\otimes \mathcal{Z}_1)\subseteq \mathcal{V}_1\otimes \mathcal{Z}_1$ and $\iota(T_{\lambda}')=0$ on $(\mathcal{V}_1\otimes \mathcal{Z}_1)^{\perp}\cap (\mathcal{V}\otimes \mathcal{Z})$, since for $x,y\in \mathcal{V}, z,w\in \mathcal{Z}, \langle \iota(T_{\lambda}')(x\wedge z), y \wedge w\rangle=T_{\lambda}'(x\wedge z\wedge y\wedge w)\neq 0$ only when $x,y\in \mathcal{V}_1$ and $w,z\in \mathcal{Z}_1$. Since $S_\lambda=T_\lambda+T_{\lambda}'$, $R_\lambda + \iota(S_\lambda)$ leaves $\mathcal{V}_1\otimes \mathcal{Z}_1$ invariant because both of $R_\lambda+\iota(T_\lambda)$ and $\iota(T_{\lambda}')$ leave $\mathcal{V}_1\otimes \mathcal{Z}_1$ invariant. This shows the items 2 and 3.

    The item 4 follows from a direct computation using the definition of $T_{\lambda}'$, where one shows the matrix representation of $\iota(T_{\lambda}')|_{\mathcal{V}_1\otimes \mathcal{Z}_1}$ with respect to $X_1\wedge Z_1, ..., X_4 \wedge Z_1, X_1,\wedge Z_2, ..., X_4\wedge Z_2$ is given by $\begin{bmatrix}
          0 & \frac{1}{4} A_2 A_1 \\
           \frac{1}{4}A_1 A_2 & 0
       \end{bmatrix}$. For example, $\langle \iota(T_\lambda')(X_1\wedge Z_1), X_2\wedge Z_2\rangle=T_\lambda'(X_1\wedge Z_1\wedge X_2\wedge Z_2)=-\frac{\xi^2}{4}$.
    Therefore, $(R_\lambda+\iota(S_\lambda))|_{\mathcal{V}_1\otimes \mathcal{Z}_1}=(R_\lambda+\iota(T_\lambda))|_{\mathcal{V}_1\otimes \mathcal{Z}_1} + \iota(T_{\lambda}')|_{\mathcal{V}_1\otimes \mathcal{Z}_1}= \frac{\xi^2}{4} \mathrm{id}_{\mathcal{V}_1\otimes \mathcal{Z}_1}$.

         We show the item 5. We have $(R_\lambda + \iota(S_\lambda))|_{(\mathcal{V}_1\otimes \mathcal{Z}_1)^{\perp}\cap (\mathcal{V}\otimes \mathcal{Z})}=(R_\lambda + \iota(T_\lambda))|_{(\mathcal{V}_1\otimes \mathcal{Z}_1)^{\perp}\cap (\mathcal{V}\otimes \mathcal{Z})}$ by the item 2. We have $T_\lambda\to 0$ as $\lambda\to 0$ by the item 1. We show $R_\lambda|_{(\mathcal{V}_1\otimes \mathcal{Z}_1)^{\perp}\cap (\mathcal{V}\otimes \mathcal{Z})}\to 0$ as $\lambda\to 0$. For $z, w \in \mathcal{Z}$ and $x,y\in \mathcal{V}$, we have $\langle R_\lambda(x\wedge z), y\wedge w \rangle=\frac{1}{4}\langle j^{(\lambda)}(z)y, j^{(\lambda)}(w)x\rangle$. Notice, a basis vector $X_j\wedge Z_\beta$ is in $(\mathcal{V}_1\otimes \mathcal{Z}_1)^{\perp}\cap (\mathcal{V}\otimes \mathcal{Z})$ if and only if $j\geq 5$ or $\beta\geq 3$. Let $x=X_j, z=Z_\beta, y=X_i, w=Z_\alpha$, where $X_j\wedge Z_\beta \in (\mathcal{V}_1\otimes \mathcal{Z}_1)^{\perp}\cap (\mathcal{V}\otimes \mathcal{Z})$. If $\beta\geq 3$, then $j^{(\lambda)}(Z_\beta)X_i\to 0$ as $\lambda \to 0$ by Lemma \ref{matrix_representation_of_j_lambda}. If $j\geq 5$, then $j^{(\lambda)}(Z_\alpha)X_j\to 0$ as $\lambda\to 0$ again by Lemma \ref{matrix_representation_of_j_lambda}. Either way, $\frac{1}{4}\langle j^{(\lambda)}(z)y, j^{(\lambda)}(w)x\rangle\to 0$ as $\lambda \to 0$. This shows the item 5.
         
      By combining the items 4 and 5, for $\lambda>0$ sufficiently small, the maximum eigenvalue of $(R_\lambda+\iota(S_\lambda))|_{\mathcal{V}\otimes \mathcal{Z}}$ is $\frac{\xi^2}{4}$. This proves Lemma \ref{lemma_existence_of_appropriate_T}.
\end{proof}

\begin{proof}[Proof of Proposition \ref{proposition_key_for_characterization}]
     To show the item \ref{item_curvature_operator_prop_for_char}, fix a sufficiently small $\lambda>0$ from Lemma \ref{lemma_existence_of_appropriate_T} so that the maximum eigenvalue of $(R_\lambda + \iota(S_{\lambda}))|_{\mathcal{V}\otimes \mathcal{Z}}$ is $\frac{\xi^2}{4}$. By compactness, there exists some unit vectors $x\in \mathcal{V}$ and $z\in \mathcal{Z}$ such that $||j^{(\lambda)}||_{\op}^2=|j^{(\lambda)}(z)x|^2$. By Proposition \ref{general_formula}, we have $K^{(\lambda)}(x,z)=\frac{1}{4}|j^{(\lambda)}(z)x|^2$, so $||j^{(\lambda)}||_{\op}^2=4K^{(\lambda)}(x,z)=4\langle (R_\lambda+\iota(S_{\lambda})|_{\mathcal{V}\otimes \mathcal{Z}}(x\wedge z), x\wedge z\rangle \leq \xi^2$. Here, $K^{(\lambda)}$ denotes the sectional curvature of $\mathcal{N}^{(\lambda)}$, and we used $\langle \iota(S_\lambda)(X\wedge Z), X\wedge Z\rangle=0$. On the other hand, $\xi^2=|j^{(\lambda)}(Z_1)X_1|^2\leq ||j^{(\lambda)}||_{\op}^2$. Therefore, $\xi^2=||j^{(\lambda)}||_{\op}^2$. 

    To see the item \ref{item_sectional_curvature_prop_for_char}, recall that $\mathcal{N}_1^{(\lambda)}$ coincides with $\mathcal{N}_1$ by Lemma \ref{lemma_N_1_lambda_is_N_1}, and it is isomorphic to $\mathrm{heis}_3(\mathbb{C})$ with a Ricci soliton inner product whose $j$-map is given by (\ref{matrices_heis_3_C_for_characterizing_maximum}). Let $\pi$ be a 2-plane in $\mathcal{N}_1^{(\lambda)}$ that achieves the maximum sectional curvature, i.e., $\frac{1}{2}\xi^2$ (Remark \ref{rescaling_heis_3_c_ricci_soliton}). By  Lemma \ref{matrix_representation_of_j_lambda}, $\mathcal{N}_1^{(\lambda)}$ is totaly geodesic in $\mathcal{N}^{(\lambda)}$, and the sectional curvature $K^{(\lambda)}(\pi)$ of $\pi$ inside the ambient metric Lie algebra $\mathcal{N}^{(\lambda)}$ is given by $K^{(\lambda)}(\pi)=\frac{1}{2}\xi^2$.
    
    To show the item \ref{item_pinching_constant_for_char}, let $K^{(\lambda)}_{\max}$ denote the maximum of the sectional curvature of $\mathcal{N}^{(\lambda)}$. By (\ref{item_curvature_operator_prop_for_char}) and (\ref{item_sectional_curvature_prop_for_char}) of Proposition \ref{proposition_key_for_characterization}, for $\lambda>0$ sufficiently small, we have $\frac{1}{2}||j^{(\lambda)}||_{\op}^2=\frac{1}{2}\xi^2=K^{(\lambda)}(\pi)\leq K^{(\lambda)}_{\max}\leq \frac{1}{2}||j^{(\lambda)}||_{\op}^2$, so $K^{(\lambda)}_{\max}= \frac{1}{2}||j^{(\lambda)}||_{\op}^2$. Note, $K^{(\lambda)}_{\max}\leq \frac{1}{2}||j^{(\lambda)}||_{\op}^2$ follows from Theorem \ref{maintheorem_bounds} and Theorem \ref{minimum_of_sectional_curvature_of_2-step_nilpotent_Lie_groups}. By Theorem \ref{minimum_of_sectional_curvature_of_2-step_nilpotent_Lie_groups}, $K_{\min}=-\frac{3}{4}||j^{(\lambda)}||_{\op}^2$ and $\frac{K_{\min}}{K_{\max}}=-\frac{3}{2}$. This completes the proof.
\end{proof}

\subsection{Proof of Theorem \ref{maintheorem_invariant} and Theorem \ref{maintheorem_open_set}} \label{section_invariant_and_open_set}
First, we prove the item 1 of Theorem \ref{maintheorem_invariant}. 
\begin{lemma} \label{lemmma_invariant_using_norms}
    Let $N$ be a 2-step nilpotent Lie group with a left-invariant metric $\langle \cdot, \cdot \rangle$. Then, $\mathcal{C}(N, \langle \cdot, \cdot \rangle)=\frac{||j||_{\Fr}^2-2||j||_{\op}^2}{||j||_{\op}}$.
\end{lemma}
\begin{proof}
    This follows from Theorem \ref{minimum_of_sectional_curvature_of_2-step_nilpotent_Lie_groups} and $\mathrm{scal}=-\frac{1}{4}||j||_{\Fr}^2$ (Proposition \ref{general_formula}).
\end{proof}

Throughout the proof of Theorem \ref{maintheorem_invariant}, we use the following setup.

\begin{setting} \label{setting_for_the_proof_of_lower_bound_theorem}
    Let $(\mathcal{N}, \langle \cdot, \cdot \rangle)$ be an arbitrary 2-step nilpotent Lie algebra with any inner product. Let $\mathcal{Z}$ be the center of $\mathcal{N}$. Let $\mathcal{V}=\mathcal{Z}^{\perp}$. 

    By Lemma \ref{maximum_eigenspace_and_two_interpretations}, there exist unit vectors $X_1, X_2 \in \mathcal{V}$ and $Z_1\in \mathcal{Z}$ such that $|j(Z_1)X_1|=||j||_{\op}$, $j(Z_1)^2X_1=-||j||_{\op}^2X_1$, $X_2=\frac{j(Z_1)X_1}{||j||_{\op}}$, and $[X_1,X_2]=||j||_{\op}Z_1$. Extend $X_1,X_2$ to an orthonormal basis $X_1, X_2, ..., X_q$ of $\mathcal{V}$. Extend $Z_1$ to an orthonormal basis $Z_1,...,Z_p$ of $\mathcal{Z}$. We fix these bases and use them throughout the proof.

    Let $\mathcal{N}_1 =\mathbb{R}\textrm{-span}\{X_1,X_2, Z_1\}$. Then, it follows from Lemma \ref{maximum_eigenspace_and_two_interpretations} that $\mathcal{N}_1$ is an aligned totally geodesic subalgebra isomorphic to $\mathrm{heis}_3(\mathbb{R})$. Moreover, $j(Z_1)X_1=||j||_{\op}X_2$ and $j(Z_1)X_2=-||j||_{\op}X_1$.

We order the basis as $e_1=X_1, ..., e_q=X_q, e_{q+1}=Z_1, ..., e_{p+q}=Z_p$, and form a basis $e_a\wedge e_b\wedge e_c\wedge e_d$ with $a<b<c<d$ of $\Lambda^4 \mathcal{N}$. Then consider the dual basis $(e_a\wedge e_b\wedge e_c\wedge e_d)^*$ with $a<b<c<d$.

Let $\mathcal{V}_1=\mathbb{R}\textrm{-span}\{X_1, X_2\}$ and $\mathcal{V}_2=\mathbb{R}\textrm{-span}\{X_3,...,X_q\}$. Let $\mathcal{Z}_1=\mathbb{R}\textrm{-span}\{Z_1\}$ and $\mathcal{Z}_2=\mathbb{R}\textrm{-span}\{Z_2,...,Z_p\}$. Write $\mathcal{V}_1\otimes \mathcal{Z}_1 = \mathbb{R}\textrm{-span}\{X_1\wedge Z_1, X_2\wedge Z_1\}$. Under this setting, Setting \ref{setting_totally_geodesic_T} is satisfied.
\end{setting}

\begin{lemma}\label{lemma_key_estimate}
For the orthonormal basis described in Setting \ref{setting_for_the_proof_of_lower_bound_theorem}, we have
    \begin{align*}
         ||j||_{\Fr}^2-2||j||_{\op}^2=\sum_{\textrm{$s\geq 3$ or $t\geq 3$}} |j(Z_1)_{st}|^2 + \sum_{\alpha=2}^p \sum_{s,t=1}^q |j(Z_\alpha)_{st}|^2,
    \end{align*}
    where $j(Z_\alpha)X_t=\sum_{s=1}^q j(Z_\alpha)_{st}X_s$, $j(Z_\alpha)_{st}\in \mathbb{R}$.
    In particular, if $\alpha\geq 2$, then for any $j=1,...,q$, $|j(Z_\alpha)X_j|\leq \sqrt{||j||_{\Fr}^2-2||j||_{\op}^2}$. If $j\geq 3$, then for any $\alpha=1,...,p$, $|j(Z_\alpha)X_j|\leq \sqrt{||j||_{\Fr}^2-2||j||_{\op}^2}$.
\end{lemma}

\begin{proof}
    By definition, $||j||_{\Fr}^2=\sum_{\alpha=1}^p |j(Z_\alpha)|_{\Fr}^2=\sum_{\alpha=1}^p \sum_{s,t=1}^q |j(Z_\alpha)_{st}|^2$. By Setting \ref{setting_for_the_proof_of_lower_bound_theorem}, we have $j(Z_1)X_1=||j||_{\op}X_2$ and $j(Z_1)X_2=-||j||_{\op}X_1$. Thus, $j(Z_1)_{21}=||j||_{\op}$, $j(Z_1)_{12}=-||j||_{\op}$, and $j(Z_1)_{11}=0=j(Z_1)_{22}$. This means $\sum_{\textrm{$s\leq 2$ and $t\leq 2$}}|j(Z_1)_{st}|^2=2||j||_{\op}^2$. Thus, $ 2||j||_{\op}^2+ \sum_{\textrm{$s\geq 3$ or $t\geq 3$}} |j(Z_1)_{st}|^2 + \sum_{\alpha=2}^p \sum_{s,t=1}^q |j(Z_\alpha)|_{st}^2 = ||j||_{\Fr}^2$. This shows the first identity. 
    
    If $\alpha\geq 2$, then for any $j=1,...,q$, $|j(Z_\alpha)X_j|^2=\sum_{s=1}^q |j(Z_\alpha)_{sj}|^2\leq ||j||_{\Fr}^2-2||j||_{\op}^2$. If $j\geq 3$ and $\alpha=1$, then $|j(Z_1)X_j|^2=\sum_{s=1}^q |j(Z_1)_{sj}|^2 \leq \sum_{\textrm{$s\geq 3$ or $t\geq 3$}} |j(Z_1)_{st}|^2 \leq ||j||_{\Fr}^2-2||j||_{\op}^2$. This completes the proof.
\end{proof}
\begin{remark}
    It is crucial that we are using the orthonormal basis $X_1,...,X_q$ and $Z_1,...,Z_p$ described in Setting \ref{setting_for_the_proof_of_lower_bound_theorem}.
\end{remark}

\begin{proof}[Proof of the item 1 of Theorem \ref{maintheorem_invariant}]
    We use Setting \ref{setting_for_the_proof_of_lower_bound_theorem}. We have $\mathcal{C}(N, \langle \cdot, \cdot \rangle)\geq 0$ by Lemma \ref{lemmma_invariant_using_norms} and Lemma \ref{lemma_key_estimate}. The invariant $\mathcal{C}(N, \langle \cdot, \cdot \rangle)$ is zero if and only if the right hand side of Lemma \ref{lemma_key_estimate} is zero, i.e., $j(Z_1)_{st}=0$ if $s\geq 3$ or $t\geq 3$, and $j(Z_\alpha)=0$ for any $\alpha \in \{2, ..., p\}$. The only non-trivial entries of the $j$-map are given by $j(Z_1)X_1=||j||_{\op}X_2$ and $j(Z_1)X_2=-||j||_{\op}X_1$. Thus, $\mathcal{C}(N, \langle \cdot, \cdot \rangle)=0$ if and only if $(\mathcal{N}, \langle \cdot, \cdot \rangle)$ is isomorphic to $\mathrm{heis}_3(\mathbb{R})\oplus \mathbb{R}^{n-3}$ as metric Lie algebras. This shows the item 1 of Theorem \ref{maintheorem_invariant}.
\end{proof}

\subsubsection{Strategy}\label{strategy}
The strategy to prove the items 2 and 3 of Theorem \ref{maintheorem_invariant} and Theorem \ref{maintheorem_open_set} is as follows. Recall, $p=\dim \mathcal{Z}, q=\dim \mathcal{V}$. Let $m=\dim \Lambda^2 \mathcal{V} \oplus \Lambda^2\mathcal{Z}=\frac{1}{2}p(p-1)+\frac{1}{2}q(q-1)$, and $C(p,q)=16m^2+4(pq-2)^2$. Since $p\leq n, q\leq n$, and $m\leq \dim \Lambda^2\mathcal{N}=\frac{1}{2}n(n-1)$, we have $C(p,q)\leq 4n^2(n-1)^2+4(n^2-2)^2$. Thus, if $\mathcal{C}(N, \langle \cdot, \cdot \rangle) \leq \frac{1}{4n^2(n-1)^2+4(n^2-2)^2}$, then $\mathcal{C}(N, \langle \cdot, \cdot \rangle)\leq \frac{1}{C(p,q)}$.

We will show that if $\mathcal{C}(N, \langle \cdot, \cdot \rangle)\leq \frac{1}{C(p,q)}$, then $\frac{K_{\min}}{K_{\max}}=-3$.

\subsubsection*{Step 1.} We find an upper bound of the sectional curvature using a modified curvature operator $R_T$ introduced in Definition \ref{def_of_modified_curvature_operator}, where $T\in \Lambda^4 \mathcal{N}^*$. By Proposition \ref{bounds_by_curv_op_general_case}, $K_{\max}$ is bounded above by the maximum eigenvalue of $R_T$ for any $T\in \Lambda^4 \mathcal{N}^*$.

\subsubsection*{Step 2.} We find a suitable $T\in \Lambda^4 \mathcal{N}^*$ so that we can control the eigenvalues of the corresponding modified curvature operator $R_T$. 

Our choice of $R_T$ will have the following block form relative to $\mathcal{V}_1\otimes \mathcal{Z}_1$, $(\mathcal{V}_1\otimes \mathcal{Z}_1)^{\perp} \cap (\mathcal{V}\otimes \mathcal{Z})$, and $\Lambda^2\mathcal{V} \oplus \Lambda^2 \mathcal{Z}$:

\begin{align*}
    R_T=\begin{bmatrix}
        M & 0 & 0 \\
        0 & A & 0 \\
        0 & 0 & B
    \end{bmatrix}.
\end{align*}
The matrix representation of $M$ with respect to an orthonormal basis $X_1\wedge Z_1, X_2\wedge Z_1$ of $\mathcal{V}_1 \otimes \mathcal{Z}_1$ is given by a 2-by-2 matrix: $M=\begin{bmatrix}
    \frac{1}{4}||j||_{\op}^2 & 0\\
    0 & \frac{1}{4}||j||_{\op}^2
\end{bmatrix}$. 
 
\subsubsection*{Step 3.} We control the maximum eigenvalues of $A$ and $B$, and the goal of Step 3 is to show that the maximum eigenvalues of $A$ and $B$ are less than or equal to $C||j||_{\op}\sqrt{||j||_{\Fr}^2-2||j||_{\op}^2}$, where $C>0$ is a constant such that $\frac{1}{C(p,q)}\leq \frac{1}{16C^2}$. For this step, we take two different approaches for $A$ and $B$, respectively. In both cases, our basic strategy is to control each entry of the matrix by $||j||_{\op}\sqrt{||j||_{\Fr}^2-2||j||_{\op}^2}$.  

\subsubsection*{Step 4.} We use the estimate from Step 3 to conclude that the maximum eigenvalue of $R_T$ is $\frac{1}{4}||j||_{\op}^2$ if $\frac{||j||_{\Fr}^2-2||j||_{\op}^2}{||j||_{\op}^2}\leq \frac{1}{C(p,q)}$. By the item 3 of Lemma \ref{maximum_eigenspace_and_two_interpretations} and Step 1, it follows that $\frac{1}{4}||j||_{\op}^2=K(X_1,Z_1)\leq K_{\max}\leq \frac{1}{4}||j||_{\op}^2$. This implies $K_{\max}=\frac{1}{4}||j||_{\op}^2$. By Theorem \ref{minimum_of_sectional_curvature_of_2-step_nilpotent_Lie_groups}, we conclude that $\frac{K_{\min}}{K_{\max}}=-3$. 

\begin{proof}[Proof of Theorem \ref{maintheorem_invariant} and Theorem \ref{maintheorem_open_set} given Steps 1--4]
    The item 1 of Theorem \ref{maintheorem_invariant} is proved above. Once Steps 1--4 are carried out, the item 2 of Theorem \ref{maintheorem_invariant} immediately follows, and $U=\{\mu\in \mathcal{L}_n\;|\; \mathcal{C}(\mu)<\frac{1}{16n^2(n-1)^2+4(n^2-2)^2}\}$ is an open set in $\mathcal{L}_n$ such that $\frac{K_{\min}}{K_{\max}}=-3$ on $U$. Let $N$ be any 2-step nilpotent Lie group with a left-invariant metric, and let $\mu\in \mathcal{L}_n$ be the corresponding element. Suppose $\tilde{N}$ is not isomorphic to $\mathrm{Heis}_3(\mathbb{R})\times \mathbb{R}^{n-3}$, or equivalently, $\mu \in \mathcal{L}_n$ is not isomorphic to $\mu_{\mathrm{heis}_3(\mathbb{R})\oplus \mathbb{R}^{n-3}}$. Then, the item 1 of Theorem \ref{maintheorem_invariant} (shown above) implies that $C(\mu)=C(N, \langle \cdot, \cdot \rangle)>0$. As explained in Introduction, $\mu$ degenerates to $\mu_{\mathrm{heis}_3(\mathbb{R})\oplus \mathbb{R}^{n-3}}$: for any $\epsilon>0$, there exists $\varphi_\epsilon\in \mathrm{GL}(n)$ such that $\lim_{\epsilon\to 0} \varphi_{\epsilon}^{-1}.\mu=\mu_{\mathrm{heis}_3(\mathbb{R})\oplus \mathbb{R}^{n-3}}$ (see the proof of \cite[Theorem 5.2]{Lauret_Degenerations}). For $\epsilon>0$ sufficiently small, we have $0<C(\varphi_{\epsilon}^{-1}.\mu)$ and $\varphi_{\epsilon}^{-1}.\mu \in U$. Each of them gives rise to a left-invariant metric on $N$ with $\frac{K_{\min}}{K_{\max}}=-3$. The invariant $\mathcal{C}(\varphi_{\epsilon}^{-1}.\mu)\to 0$ while $\frac{K_{\min}}{K_{\max}}=-3$ as $\epsilon \to 0$, so all of these metrics are not isometric up to scaling. This proves the item 3 of Theorem \ref{maintheorem_invariant} and Theorem \ref{maintheorem_open_set}.
\end{proof}
We carry out Steps 1--4 below, and we use the following notation. Given a matrix $U\in M_n(\mathbb{R})$, let $||U||_{\op}=\max_{|x|=1} |\langle Ux,x \rangle|$ (operator norm) and $||U||_{\Fr}=\sqrt{\sum_{s,t=1}^n |U_{st}|^2}$ (Frobenius norm). Here, $U_{st}$ denotes the $(s,t)$-entry of $U$. 

Suppose $U$ is symmetric. Let $\lambda_1\leq ... \leq \lambda_n$ be its eigenvalues. Then, observe that $||U||_{\op}=\max_{j} |\lambda_j|$, and $||U||_{\Fr}=\sqrt{\sum_{j=1}^n \lambda_j^2}$. Thus, $\max_{j} \lambda_j \leq ||U||_{\op}\leq ||U||_{\Fr}$. 

Moreover, suppose there exists a constant $C$ such that for any $s,t\in \{1,...,n\}, |U_{st}|\leq C$. Then, $||U||_{\Fr}^2 = \sum_{s,t=1}^n |U_{st}|^2\leq n^2 C^2.$ Thus, $||U||_{\Fr}\leq nC$.
\subsubsection{Carrying out Step 1.} For any $T\in \Lambda^4 \mathcal{N}^*$, it follows from Proposition \ref{bounds_by_curv_op_general_case} that 
$K_{\max} \leq \lambda_{\max}(R_T)$.

\subsubsection{Carrying out Step 2.}
 We use the notations in Setting \ref{setting_for_the_proof_of_lower_bound_theorem}. The desired $T\in \Lambda^4 \mathcal{N}^*$ described in Step 2 of Strategy \ref{strategy} is given in the following lemma.

\begin{lemma} \label{lemma_T_for_lower_bound}
    Let $c_{i1j\alpha}=\langle R(X_i\wedge Z_1), X_j\wedge Z_\alpha \rangle$ for $i\in \{1,2\}, j\in \{3,...,q\}, \alpha\geq 2$. Set 
    \begin{align*}
        T=\sum_{i,j,\alpha} c_{i1j\alpha}(X_i\wedge X_j\wedge Z_1\wedge Z_\alpha)^*,
    \end{align*}
    where the sum runs over the indices $i\in \{1,2\}, j\in \{3,...,q\}, \alpha\geq 2$.
    Then, $\iota(T)$ leaves $\mathcal{V}\otimes \mathcal{Z}$ invariant, and ${R_T}|_{\mathcal{V}\otimes \mathcal{Z}}=(R+\iota(T))|_{\mathcal{V}\otimes \mathcal{Z}}$ leaves $\mathcal{V}_1\otimes \mathcal{Z}_1 = \mathbb{R}\textrm{-span}\{X_1\wedge Z_1, X_2\wedge Z_1\}$ invariant.

    In particular, $R_T$ leaves $\mathcal{V}_1\otimes \mathcal{Z}_1$, $(\mathcal{V}_1\otimes \mathcal{Z}_1)^{\perp} \cap (\mathcal{V}\otimes \mathcal{Z})$, and $\Lambda^2\mathcal{V} \oplus \Lambda^2\mathcal{Z}$.
\end{lemma}

\begin{proof}
    As pointed out in Setting \ref{setting_for_the_proof_of_lower_bound_theorem}, $\mathcal{N}_1$ is an aligned totally geodesic subalgebra of $\mathcal{N}$. Also, observe that Setting \ref{setting_for_the_proof_of_lower_bound_theorem} is a special case of Setting \ref{setting_totally_geodesic_T}. In particular, Lemma \ref{curv_op_and_totally_geodesic_subalgebras} applies to our setting, and the result follows.
\end{proof}
By the item 6 of Lemma \ref{curv_op_and_totally_geodesic_subalgebras} and Lemma \ref{formulas_for_curv_op}, $\langle R_T(X_1\wedge Z_1), X_2\wedge Z_1 \rangle=\langle R(X_1\wedge Z_1), X_2\wedge Z_1 \rangle=\frac{1}{4}\langle j(Z_1)X_1, j(Z_1)X_2\rangle=0$. The other entries of $M$ are computed similarly.

\subsubsection{Carrying out Step 3.} Recall $m=\dim \Lambda^2\mathcal{V}\oplus\Lambda^2 \mathcal{Z}$. In Step 3, we show that the maximum eigenvalues of the matrices $A$ and $B$ above are less than or equal to $\frac{1}{2}(pq-2)||j||_{\op}\sqrt{||j||_{\Fr}^2-2||j||_{\op}^2}$ and $m||j||_{\op}\sqrt{||j||_{\Fr}^2-2||j||_{\op}^2}$, respectively.  For this step, we take two different approaches for $A$ and $B$, respectively. In both cases, our basic strategy is to control each entry of the matrix by $||j||_{\op}\sqrt{||j||_{\Fr}^2-2||j||_{\op}^2}$.

Note that $\{X_i\wedge Z_\alpha\;|\;i\in \{1,...,q\}, \alpha \in\{2,...,p\}\}\cup \{X_3\wedge Z_1, ..., X_q\wedge Z_1\}$ span the subspace $(\mathcal{V}_1\otimes \mathcal{Z}_1)^{\perp} \cap (\mathcal{V}\otimes \mathcal{Z})$, so 
\begin{align*}
    (\mathcal{V}_1\otimes \mathcal{Z}_1)^{\perp} \cap (\mathcal{V}\otimes \mathcal{Z})=\mathbb{R}\textrm{-span}\{X_i\wedge Z_\alpha \;|\;i\in \{1,...,q\}, \alpha \in\{1,...,p\}, i\geq 3 \text{ or } \alpha \geq 2\}.
\end{align*}

First, we will use the lexicographical order of the basis elements for the matrix representations $A$ and $B$. More precisely, let $A$ be the matrix representation of $R_T|_{(\mathcal{V}_1\otimes \mathcal{Z}_1)^{\perp}\cap(\mathcal{V}\otimes \mathcal{Z})}$ with respect to the basis $X_3\wedge Z_1, ..., X_q\wedge Z_1, X_1\wedge Z_2, ..., X_q\wedge Z_2, X_1\wedge Z_3, ..., X_q\wedge Z_3, ..., X_1\wedge Z_p, ..., X_q\wedge Z_p$ of $(\mathcal{V}_1\otimes \mathcal{Z}_1)^{\perp}\cap(\mathcal{V}\otimes \mathcal{Z})$. 

Let $B$ be the matrix representation of $R_T|_{\Lambda^2\mathcal{V} \oplus \Lambda^2\mathcal{Z}}$ with respect to $X_i\wedge X_j$ with $i<j$ and $Z_\alpha\wedge Z_\beta$ with $\alpha<\beta$. Within $\Lambda^2\mathcal{V}$, the order is $X_1\wedge X_2, X_1\wedge X_3, ..., X_1\wedge X_q$, $X_2\wedge X_3, ..., X_2 \wedge X_q, ..., X_{q-1}\wedge X_q$. The basis of $\Lambda^2\mathcal{Z}$ follows next: $Z_1\wedge Z_2,..., Z_1\wedge Z_p, Z_2\wedge Z_3,..., Z_2\wedge Z_p,...,Z_{p-1}\wedge Z_p$. 

Next, let us describe what each entry of $A$ and $B$ looks like.
\begin{lemma} \label{entries_of_A_and_B}
    The entries of $A$ are of the form $\langle R_T(X_i\wedge Z_\alpha), X_j\wedge Z_\beta\rangle$ (analyzed in Lemma \ref{VZVZ_upper_bound_entrywise}), where $X_i\wedge Z_\alpha$ and $X_j\wedge Z_\beta$ are the basis element of $(\mathcal{V}_1\otimes \mathcal{Z}_1)^{\perp}\cap(\mathcal{V}\otimes \mathcal{Z})$. 

    The entries of $B$ are of the form
    \begin{enumerate} 
    \item $\langle R_T(X_i\wedge X_j), Z_\alpha\wedge Z_\beta \rangle, i<j, \alpha<\beta$ (analyzed in Lemma \ref{VVZZ_upper_bound_entrywise}),
    \item $\langle R_T(X_i\wedge X_j), X_k\wedge X_l\rangle, i<j, k<l$ (analyzed in Lemma \ref{VVVV_upper_bound_entrywise}), and
    \item $\langle R_T(Z_\alpha\wedge Z_\beta), Z_\gamma\wedge Z_\delta\rangle=0$, $\alpha<\beta, \gamma<\delta$.
\end{enumerate}
\end{lemma}
\begin{proof}
     The item 3 follows from Lemma \ref{formulas_for_curv_op} and $\langle \iota(T)(Z_\alpha\wedge Z_\beta), Z_\gamma\wedge Z_\delta\rangle=T(Z_\alpha \wedge Z_\beta \wedge Z_\gamma \wedge Z_\delta )=0$. 
\end{proof}

Now, we start analyzing the matrices $A$ and $B$ separately.

\subsubsection*{Analysis of the matrix A}

As in Lemma \ref{entries_of_A_and_B}, the entry of $A$ is of the form $\langle R_T(X_i\wedge Z_\alpha), X_j\wedge Z_\beta\rangle$, where $X_i\wedge Z_\alpha$ and $X_j\wedge Z_\beta$ are the basis element of $(\mathcal{V}_1\otimes \mathcal{Z}_1)^{\perp}\cap(\mathcal{V}\otimes \mathcal{Z})$. Hence, $j\geq 3$ or $\beta \geq 2$.

\begin{lemma} \label{VZVZ_upper_bound_entrywise}
Let $i,j\in \{1,...,q\}$ and $\alpha, \beta\in \{1,...,p\}$. If $j\geq 3$ or $\beta \geq 2$, then
    $|\langle R(X_i\wedge Z_\alpha), X_j\wedge Z_\beta \rangle| \leq \frac{1}{4}||j||_{\op}\sqrt{||j||_{\Fr}^2-2||j||_{\op}^2}$. In particular, the maximum eigenvalue of $R_T|_{(\mathcal{V}_1\otimes \mathcal{Z}_1)^{\perp}\cap(\mathcal{V}\otimes \mathcal{Z})}$ is less than or equal to $\frac{1}{2}(pq-2)||j||_{\op}\sqrt{||j||_{\Fr}^2-2||j||_{\op}^2}$.
\end{lemma}
\begin{proof}
    By Lemma \ref{formulas_for_curv_op}, we have $\langle R(X_i\wedge Z_\alpha), X_j\wedge Z_\beta \rangle=\frac{1}{4}\langle j(Z_\alpha)X_j, j(Z_\beta)X_i\rangle.$ By Cauchy--Schwarz inequality, we have $|\langle R(X_i\wedge Z_\alpha), X_j\wedge Z_\beta \rangle|=\frac{1}{4}|\langle j(Z_\alpha)X_j, j(Z_\beta)X_i\rangle|\leq \frac{1}{4}|j(Z_\alpha)X_j||j(Z_\beta)X_i|$. We always have $|j(Z_\alpha)X_j|\leq ||j||_{\op}$ and $|j(Z_\beta)X_i|\leq ||j||_{\op}$. Since $j\geq 3$ or $\beta\geq 2$, it follows from Lemma \ref{lemma_key_estimate} that $|j(Z_\alpha)X_j||j(Z_\beta)X_i|\leq ||j||_{\op}\sqrt{||j||_{\Fr}^2-2||j||_{\op}^2}$, and the inequality follows. 
    
    Recall, $A$ is the matrix representation of $R_T|_{(\mathcal{V}_1\otimes \mathcal{Z}_1)^{\perp}\cap(\mathcal{V}\otimes \mathcal{Z})}$ with respect to the basis $X_3\wedge Z_1, ..., X_q\wedge Z_1$, and $X_1\wedge Z_l, ..., X_q\wedge Z_l$, $l\geq 2$ of $(\mathcal{V}_1\otimes \mathcal{Z}_1)^{\perp}\cap(\mathcal{V}\otimes \mathcal{Z})$. All of the basis elements are of the form $X_i\wedge Z_\alpha$, $i\geq 3$ or $\alpha\geq 2$. Thus, if $X_i\wedge Z_\alpha, X_j\wedge Z_\beta$ are the basis elements above, then $|\langle R_T(X_i\wedge Z_\alpha), X_j\wedge Z_\beta \rangle|\leq |\langle R(X_i\wedge Z_\alpha), X_j\wedge Z_\beta \rangle|+|T(X_i\wedge Z_\alpha\wedge X_j\wedge Z_\beta)| \leq \frac{1}{4}||j||_{\op}\sqrt{||j||_{\Fr}^2-2||j||_{\op}^2} + |T(X_i\wedge Z_\alpha\wedge X_j\wedge Z_\beta)|$ as $j\geq 3$ or $\beta\geq 2$. Recall, $T$ is defined in Lemma \ref{lemma_T_for_lower_bound}, as $T=\sum_{i',j',\alpha'} c_{i'1j'\alpha'}(X_{i'}\wedge X_{j'}\wedge Z_1\wedge Z_{\alpha'})^*$, where the sum runs over the indices $i'\in \{1,2\}, j'\in \{3,...,q\}, \alpha' \geq 2$.

    To evaluate $|T(X_i\wedge Z_\alpha\wedge X_j\wedge Z_\beta)|=|T(X_i\wedge X_j\wedge Z_\alpha\wedge Z_\beta)|$, we may assume $i\leq j$ and $\alpha\leq \beta$ by permutation. If $i=j$ or $\alpha=\beta$, then it is zero. Hence, we assume $i<j$ and $\alpha<\beta$.
 Observe that it follows from Lemma \ref{lemma_T_for_lower_bound} that if $|T(X_i\wedge X_j\wedge Z_\alpha\wedge Z_\beta)|\neq 0$, then $i\in \{1,2\}, j\in \{3,...,q\}, \alpha=1$, and $\beta\geq 2$, and $T(X_i\wedge X_j\wedge Z_\alpha\wedge Z_\beta)=c_{i1j\beta}=\langle R(X_i\wedge Z_1), X_j\wedge Z_\beta\rangle$. Since $\beta\geq 2$, by the first part of this lemma, we see $|T(X_i\wedge X_j\wedge Z_\alpha \wedge Z_\beta)|=|\langle R(X_i\wedge Z_1), X_j\wedge Z_\beta\rangle|\leq  \frac{1}{4}||j||_{\op}\sqrt{||j||_{\Fr}^2-2||j||_{\op}^2}$ as well. Thus, $|\langle R_T(X_i\wedge Z_\alpha), X_j\wedge Z_\beta \rangle|\leq |\langle R(X_i\wedge Z_\alpha), X_j\wedge Z_\beta \rangle|+|T(X_i\wedge Z_\alpha\wedge X_j\wedge Z_\beta)| \leq \frac{1}{4}||j||_{\op}\sqrt{||j||_{\Fr}^2-2||j||_{\op}^2} + |T(X_i\wedge Z_\alpha\wedge X_j\wedge Z_\beta)| \rangle|\leq \frac{1}{2}||j||_{\op}\sqrt{||j||_{\Fr}^2-2||j||_{\op}^2}$.
    
    The absolute value of each entry of $A$ is bounded by $\frac{1}{2}||j||_{\op}\sqrt{||j||_{\Fr}^2-2||j||_{\op}^2}$. We have $||A||_{\op}^2\leq ||A||_{\Fr}^2 \leq \frac{1}{4}(pq-2)^2||j||_{\op}^2(||j||_{\Fr}^2-2||j||_{\op}^2),$ where $pq-2$ is the dimension of $(\mathcal{V}_1\otimes \mathcal{Z}_1)^{\perp}\cap(\mathcal{V}\otimes \mathcal{Z})$. Therefore, the maximum eigenvalue $||A||_{\op}$ of $A$ is bounded by $\frac{1}{2}(pq-2)||j||_{\op}\sqrt{||j||_{\Fr}^2-2||j||_{\op}^2}$. This completes the proof.
\end{proof}

\subsubsection*{Analysis of the matrix B.}
\begin{lemma} \label{B_11_computation}
    The $(1,1)$-entry of the matrix $B$ is given by $B_{11}=\langle R_T(X_1 \wedge X_2), X_1\wedge X_2 \rangle=-\frac{3}{4}||j||_{\op}^2$ 
\end{lemma}
\begin{proof}
    We have $B_{11}= \langle R(X_1\wedge X_2), X_1\wedge X_2\rangle + T(X_1\wedge X_2\wedge X_1\wedge X_2).$ The first terms is $K(X_1, X_2)=-\frac{3}{4}|[X_1, X_2]|^2=-\frac{3}{4}||j||_{\op}^2$ by Proposition \ref{general_formula} and Setting \ref{setting_for_the_proof_of_lower_bound_theorem}. The second term is 0 by definition of $T$.
\end{proof}
We will obtain bounds on the entries of $B$ in Lemma \ref{VVZZ_upper_bound_entrywise} and Lemma \ref{VVVV_upper_bound_entrywise}. Then, we control $||B||_{\op}$ in Lemma \ref{upper_bound_of_maximum_eigenvalue_of_B}.

\begin{lemma}\label{VVZZ_upper_bound_entrywise}
    Let $i,j\in \{1,...,q\}$ with $i<j$, and $\alpha,\beta\in \{1,...,p\}$ with $\alpha<\beta$. Then, $|\langle R_T(X_i\wedge X_j), Z_\alpha \wedge Z_\beta \rangle|\leq \frac{3}{4}||j||_{\op}\sqrt{||j||_{\Fr}^2-2||j||_{\op}^2}$.
\end{lemma}

\begin{proof}
    If $\langle \iota(T)(X_i\wedge X_j), Z_\alpha\wedge Z_\beta \rangle=T(X_i\wedge X_j \wedge Z_\alpha \wedge Z_\beta)$ is not zero, then $\alpha=1$ and it is of the form $T(X_i\wedge X_j \wedge Z_\alpha \wedge Z_\beta)=c_{i1j\beta}=\langle R(X_i\wedge Z_1), X_j\wedge Z_\beta\rangle$. Observe,  $\alpha<\beta$ implies $\beta\geq 2$. Since $\beta\geq 2$, it follows from Lemma \ref{VZVZ_upper_bound_entrywise} that $|\langle \iota(T)(X_i\wedge X_j), Z_\alpha\wedge Z_\beta \rangle|\leq \frac{1}{4}||j||_{\op}\sqrt{||j||_{\Fr}^2-2||j||_{\op}^2}$. 

    By Lemma \ref{formulas_for_curv_op}, $\langle R(X_i\wedge X_j), Z_\alpha \wedge Z_\beta \rangle = \frac{1}{4}\langle j(Z_\alpha)X_j, j(Z_\beta)X_i \rangle - \frac{1}{4}\langle j(Z_\alpha)X_i, j(Z_\beta)X_j \rangle$. Then,  $|\langle R(X_i\wedge X_j), Z_\alpha \wedge Z_\beta \rangle| \leq \frac{1}{4}|\langle j(Z_\alpha)X_j, j(Z_\beta)X_i \rangle| + \frac{1}{4}|\langle j(Z_\alpha)X_i, j(Z_\beta)X_j \rangle|\leq \frac{1}{4}(|j(Z_\alpha)X_j||j(Z_\beta)X_i|+|j(Z_\alpha)X_i||j(Z_\beta)X_j|)$ by Cauchy--Schwarz inequality. We always have $|j(Z_\alpha)X_j|\leq ||j||_{\op}$. Since $\beta\geq 2$ it follows from Lemma \ref{lemma_key_estimate} that $|j(Z_\beta)X_i|\leq \sqrt{||j||_{\Fr}^2-2||j||_{\op}^2}$. Thus, $|j(Z_\alpha)X_j||j(Z_\beta)X_i|\leq ||j||_{\op} \sqrt{||j||_{\Fr}^2-2||j||_{\op}^2}$. Similarly, $|\langle j(Z_\alpha)X_i, j(Z_\beta)X_j \rangle|\leq ||j||_{\op} \sqrt{||j||_{\Fr}^2-2||j||_{\op}^2}$. 
    
    Thus, we have $|\langle R(X_i\wedge X_j), Z_\alpha\wedge Z_\beta\rangle|\leq \frac{1}{2}||j||_{\op}\sqrt{||j||_{\Fr}^2-2||j||_{\op}^2}$. Therefore, $|\langle R_T(X_i\wedge X_j), Z_\alpha \wedge Z_\beta \rangle|\leq  |\langle R(X_i\wedge X_j), Z_\alpha \wedge Z_\beta \rangle|+ |\langle \iota(T)(X_i\wedge X_j), Z_\alpha\wedge Z_\beta \rangle| \leq \frac{3}{4}||j||_{\op}\sqrt{||j||_{\Fr}^2-2||j||_{\op}^2}$.
\end{proof}

\begin{lemma}\label{VVVV_upper_bound_entrywise}
   Let $i,j,k,l\in \{1,...,q\}$, $i<j, k<l$. Suppose $l\geq 3$. Then, $|\langle R_T(X_i\wedge X_j), X_k\wedge X_l\rangle|=|\langle X_i\wedge X_j, R_T(X_k\wedge X_l)\rangle| \leq ||j||_{\op}\sqrt{||j||_{\Fr}^2 - 2||j||_{\op}^2}$.
\end{lemma}
\begin{remark} \label{why_everything_except_B_11_is_covered}
    If $l<3$, then $k=1, l=2$. Also, you can switch the roles of $i<j$ and $k<l$. Thus, this lemma provides an estimate of $|\langle R_T(X_i\wedge X_j), X_k\wedge X_l\rangle|$ except $B_{11}=\langle R_T(X_1\wedge X_2), X_1\wedge X_2\rangle=-\frac{3}{4}||j||_{\op}^2$ (see Lemma \ref{B_11_computation}). 
\end{remark}

\begin{proof}
First, $\langle R_T(X_i\wedge X_j), X_k\wedge X_l\rangle=\langle R(X_i\wedge X_j), X_k\wedge X_l\rangle$, because $T(X_i\wedge X_j\wedge X_k\wedge X_l)=0$. By Lemma \ref{formulas_for_curv_op}, we have $\langle R(X_i\wedge X_j), X_k\wedge X_l\rangle =-\frac{1}{2}\langle[X_i, X_j], [X_k, X_l]\rangle + \frac{1}{4}\langle [X_j, X_k], [X_i, X_l]\rangle - \frac{1}{4}\langle [X_i, X_k],[X_j,X_l]\rangle$.

    We analyze $\langle[X_i, X_j], [X_k, X_l]\rangle$. We have $|\langle[X_i, X_j], [X_k, X_l]\rangle|\leq ||j||_{\op}|[X_k,X_l]|$ by Cauchy--Schwarz inequality and Lemma \ref{upper_bound_of_Lie_bracket}. Note, $[X_k,X_l]=\sum_{\alpha=1}^p \langle j(Z_{\alpha})X_k, X_l\rangle Z_{\alpha}=\sum_{\alpha=1}^p j(Z_\alpha)_{lk} Z_{\alpha}$. Hence, $|[X_k,X_l]|^2=|j(Z_1)_{lk}|^2 + \sum_{\alpha=2}^p |j(Z_{\alpha})_{lk}|^2\leq ||j||_{\Fr}^2 - 2||j||_{\op}^2.$
    This inequality follows from $l\geq 3$ and Lemma \ref{lemma_key_estimate}. This implies $|[X_k,X_l]|\leq \sqrt{||j||_{\Fr}^2 - 2||j||_{\op}^2}$. Thus, $|\langle[X_i, X_j], [X_k, X_l]\rangle|\leq ||j||_{\op}\sqrt{||j||_{\Fr}^2 - 2||j||_{\op}^2}.$ Similar computations hold for $\langle [X_j, X_k], [X_i, X_l]\rangle$ and $\langle [X_i, X_k],[X_j,X_l]\rangle$ as $l\geq 3$. 
    
    We have $|\langle R_T(X_i\wedge X_j), X_k\wedge X_l\rangle|=|\langle R(X_i\wedge X_j), X_k\wedge X_l\rangle|\leq (\frac{1}{2}+\frac{1}{4}+\frac{1}{4})||j||_{\op}\sqrt{||j||_{\Fr}^2 - 2||j||_{\op}^2}=||j||_{\op}\sqrt{||j||_{\Fr}^2 - 2||j||_{\op}^2}$.
\end{proof}

Now, we give an upper bound of the maximum eigenvalue of $B$.

\begin{lemma} \label{upper_bound_of_maximum_eigenvalue_of_B}
    The maximum eigenvalue of $B$ is less than or equal to $m||j||_{\op}\sqrt{||j||_{\Fr}^2 - 2||j||_{\op}^2}$, where $m=\dim \Lambda^2 \mathcal{V}\oplus \Lambda^2\mathcal{Z}$.
\end{lemma}

\begin{proof}
    Recall, $B$ is the matrix representation of $R_T|_{\Lambda^2\mathcal{V}\oplus \Lambda^2\mathcal{Z}}$ with respect to $X_i\wedge X_j$ with $i<j$ and $Z_\alpha\wedge Z_\beta$ with $\alpha<\beta$. By Lemma \ref{B_11_computation}, the $(1,1)$-entry of the matrix $B$ is $\langle R_T(X_1\wedge X_2), X_1\wedge X_2\rangle = -\frac{3}{4}||j||_{\op}^2$.
    
    Let $B_1\in M_{m}(\mathbb{R})$ be the matrix whose $(1,1)$-entry is $\langle R_T(X_1\wedge X_2), X_1\wedge X_2\rangle = -\frac{3}{4}||j||_{\op}^2$, and all the other entries are $0$. Let $B_2=B-B_1$. Then, the absolute value of each entry of $B_2$ is bounded by $||j||_{\op}\sqrt{||j||_{\Fr}^2 - 2||j||_{\op}^2}$ by Lemma \ref{VVZZ_upper_bound_entrywise} and Lemma \ref{VVVV_upper_bound_entrywise} (see also Lemma \ref{entries_of_A_and_B} and Remark \ref{why_everything_except_B_11_is_covered}). This implies $||B_2||_{\Fr}^2\leq m^2||j||_{\op}^2(||j||_{\Fr}^2 - 2||j||_{\op}^2)$. Now, given a symmetric matrix $S\in M_m(\mathbb{R})$, let $\lambda^\downarrow(S)=(\lambda_1^\downarrow(S), ..., \lambda_m^\downarrow(S))$ be the ordered tuple of eigenvalues of $S$, where $\lambda_1^\downarrow(S)\geq ... \geq \lambda_m^\downarrow(S)$. We know $\lambda^\downarrow(B_1)=(0,...,0, -\frac{3}{4}||j||_{\op}^2)$. By Weyl's perturbation theorem (\cite[Corollary III.2.6]{Bhatia}), we have $\max_t|\lambda_t^\downarrow(B)-\lambda_t^\downarrow(B_1)|\leq ||B-B_1||_{\op}$. Note, $\lambda_1^\downarrow(B_1)=0$. Thus, we have $|\lambda_1^\downarrow(B)|\leq ||B-B_1||_{\op}=||B_2||_{\op}\leq ||B_2||_{Fr}\leq m||j||_{\op}\sqrt{||j||_{\Fr}^2 - 2||j||_{\op}^2}$. 
\end{proof}

By Lemma \ref{VZVZ_upper_bound_entrywise}, the maximum eigenvalue of $A=R_T|_{(\mathcal{V}_1\otimes \mathcal{Z}_1)^{\perp}\cap(\mathcal{V}\otimes \mathcal{Z})}$ is less than or equal to $\frac{1}{2}(pq-2)||j||_{\op}\sqrt{||j||_{\Fr}^2 - 2||j||_{\op}^2}$. By Lemma \ref{upper_bound_of_maximum_eigenvalue_of_B}, the maximum eigenvalue of $B=R_T|_{\Lambda^2\mathcal{V}\oplus \Lambda^2\mathcal{Z}}$ is less than or equal to $m||j||_{\op}\sqrt{||j||_{\Fr}^2 - 2||j||_{\op}^2}$. Let $C=\max\{m, \frac{1}{2}(pq-2)\}$ so that $\max\{||A||_{\op}, ||B||_{\op}\}\leq C ||j||_{\op}\sqrt{||j||_{\Fr}^2 - 2||j||_{\op}^2}$. Then $\frac{1}{C(p,q)}\leq \frac{1}{16C^2}$. This finishes Step 3.

\subsubsection{Carrying out Step 4.}

Suppose that the maximum eigenvalues of $A$ and $B$ are less than or equal to $C||j||_{\op}\sqrt{||j||_{\Fr}^2-2||j||_{\op}^2}$, where $C>0$ is the constant from Step 3 satisfying $\frac{1}{C(p,q)}\leq \frac{1}{16C^2}$. We want the following conditions to be true:
\begin{align*}
    C||j||_{\op}\sqrt{||j||_{\Fr}^2-2||j||_{\op}^2}\leq \frac{1}{4}||j||_{\op}^2.
\end{align*}
This is equivalent to 
\begin{align*}
    \sqrt{\frac{||j||_{\Fr}^2-2||j||_{\op}^2}{||j||_{\op}^2}} \leq \frac{1}{4C}, \text{ or } \frac{||j||_{\Fr}^2-2||j||_{\op}^2}{||j||_{\op}^2}\leq \frac{1}{16C^2}.
\end{align*}
Since $\frac{1}{C(p,q)}\leq \frac{1}{16C^2}$, the condition $\frac{||j||_{\Fr}^2-2||j||_{\op}^2}{||j||_{\op}^2}\leq \frac{1}{C(p,q)}$ implies $C||j||_{\op}\sqrt{||j||_{\Fr}^2-2||j||_{\op}^2}\leq \frac{1}{4}||j||_{\op}^2$, and the maximum eigenvalues of $A$ and $B$ are less than or equal to $\frac{1}{4}||j||_{\op}^2$. By Step 2, it follows that the maximum eigenvalue of $R_T$ is $\frac{1}{4}||j||_{\op}^2$ if $\frac{||j||_{\Fr}^2-2||j||_{\op}^2}{||j||_{\op}^2}\leq \frac{1}{C(p,q)}$. This finishes Step 4, and Theorem \ref{maintheorem_invariant} and Theorem \ref{maintheorem_open_set} follow.

\section{Various Examples} \label{section_examples}
\subsection{Product Metrics}  \label{section_product_metrics}

In this section, we explain Example \ref{product_of_heis_3_over_r_and_c}.
\begin{lemma}\label{max_and_min_of_product} 
    Let $M\times N$ be a Riemannian product of Riemannian manifolds $M$ and $N$. Then,  $K_{\max}(M\times N)=\max \{K_{\max}(M), K_{\max}(N),0\} $ and 
    $K_{\min}(M\times N)=\min \{K_{\min}(M), K_{\min}(N), 0 \}$.
\end{lemma}
\begin{proof}[Proof of the statements in Example \ref{product_of_heis_3_over_r_and_c}]
    Take $\langle \cdot, \cdot \rangle_{\mathrm{Heis}_3(\mathbb{R})}$ from Example \ref{example_heis_3} with $\lambda=||j||_{\op}=1$ and $\langle \cdot, \cdot \rangle_{\mathrm{Heis}_3(\mathbb{C})}$ from Example \ref{example_heis_3_C_with_Ricci_soliton_metric}, so $K_{\min}=-\frac{3}{4}$ for both metrics (Theorem \ref{minimum_of_sectional_curvature_of_2-step_nilpotent_Lie_groups}). 

    Use the notations of Example \ref{example_heis_3}. For $\frac{1}{a^2}\langle \cdot, \cdot \rangle_{\mathrm{Heis}_3(\mathbb{R})}$, $ae_1, ae_2, ae_3$ form an orthonormal frame, and $j(ae_3)=\begin{bmatrix} 0 & -a \\ a & 0 \end{bmatrix}$, where the matrix representation is with respect to $ae_1, ae_2$. Thus, $||j||_{\op}=a$ for this metric, and $K_{\max}(\mathrm{Heis}_3(\mathbb{R}), \frac{1}{a^2}\langle \cdot, \cdot \rangle_{\mathrm{Heis}_3(\mathbb{R})})=\frac{1}{4}a^2$ and $K_{\min}(\mathrm{Heis}_3(\mathbb{R}), \frac{1}{a^2}\langle \cdot, \cdot \rangle_{\mathrm{Heis}_3(\mathbb{R})})=-\frac{3}{4}a^2$ by Example \ref{example_heis_3}. On the other hand, by Example \ref{example_heis_3_C_with_Ricci_soliton_metric}, $K_{\max}(\mathrm{Heis}_3(\mathbb{C}), \langle \cdot, \cdot \rangle_{\mathrm{Heis}_3(\mathbb{C})})=\frac{1}{2}$ and $K_{\min}(\mathrm{Heis}_3(\mathbb{C}), \langle \cdot, \cdot \rangle_{\mathrm{Heis}_3(\mathbb{C})})=-\frac{3}{4}$. Now the result follows from Lemma \ref{max_and_min_of_product}. 
\end{proof}

The following example shows that a 2-plane that achieves $K_{\max}$ may not be unique up to isometry.

\begin{example} \label{example_non-uniqueness_of_maximal_plane}
    Let $N$ be a simply connected 2-step nilpotent Lie group whose Lie algebra is given by $\mathcal{N}=\mathbb{R}\textrm{-span}\{X_1, X_2, X_3, X_4,X_5,X_6,Z_1,Z_2\}$, where non-trivial Lie bracket relations are given by $ [X_1,X_3]=Z_1, [X_1,X_4]=Z_2, [X_2,X_3]=-Z_2, [X_2,X_4]=Z_1$, and $[X_5,X_6]=\lambda^2 Z_1$. Let $\langle \cdot, \cdot \rangle$ be the left-invariant metric on $N$ such that $X_1,...,X_6,Z_1,Z_2$ form an orthonormal basis at $T_eN\cong \mathcal{N}$. Let $N'=\mathrm{Heis}_3(\mathbb{C})$ with a Ricci soliton left-invariant metric, normalized as $K_{\min}=-\frac{3}{4}$. Let $\mathcal{N}'=\mathrm{Lie}N'$. Let $N\times N'$ be equipped with a product Riemannian left-invariant metric.
\end{example}
\begin{proposition} Let $\lambda>0$ be sufficiently small. Then, $N\times N'$ has $\frac{K_{\min}}{K_{\max}}=-\frac{3}{2}$, and there exist 2-planes $\pi\subseteq \mathcal{N}_1, \pi'\subseteq \mathcal{N}'$ such that the following holds.
\begin{enumerate}
    \item $K(\pi_1)=K(\pi_2)=K_{\max}$.
    \item There does not exist an isometry $\Phi$ of $N\times N'$ with $\Phi(e)=e$ that takes $\pi_1$ to $\pi_2$.
\end{enumerate}
\end{proposition}
\begin{proof} Let $j, j'$ denote the $j$-map of $\mathcal{N}, \mathcal{N}'$, respectively. Let $\mathcal{N}_1=\mathbb{R}\textrm{-span}\{X_1,...,X_4,Z_1,Z_2\}$. Then $\mathcal{N}_1$ is a totally geodesic subalgebra isomorphic to $\mathrm{heis}_3(\mathbb{C})$ with a Ricci soliton inner product.  Indeed, if $A_1$ and $A_2$ are the matrices in Example \ref{example_heis_3_C_with_Ricci_soliton_metric} and $J=\begin{bmatrix} 0 & -1 \\ 1 & 0 \end{bmatrix}$, then $j(Z_1)=\begin{bmatrix} A_1 & 0 \\ 0 & \lambda^2 J \end{bmatrix}$ and $j(Z_2)=\begin{bmatrix} A_2 & 0 \\ 0 & 0 \end{bmatrix}$. Let $\pi_1, \pi_2$ be the 2-planes of $\mathcal{N}_1, \mathcal{N}'$, respectively, such that $K(\pi_1)=\frac{1}{2}=K(\pi_2)$. For $\lambda>0$ sufficiently small, $K_{\max}=\frac{1}{2}$ and $||j||_{\op}=1$. This is because the degeneration $\lambda\to 0$ is a special case of the degeneration considered in Proposition \ref{proposition_key_for_characterization} (see also Lemma \ref{matrix_representation_of_j_lambda}). The item 1 follows. By Theorem \ref{minimum_of_sectional_curvature_of_2-step_nilpotent_Lie_groups} and Lemma \ref{max_and_min_of_product}, we see $\frac{K_{\min}}{K_{\max}}=-\frac{3}{2}$. Next, we show any isometry $\Phi$ with $\Phi(e)=e$ taking $\pi_1$ to $\pi_2$ would need to satisfy $\varphi(\mathcal{N}_1)=\mathcal{N}'$, where $\varphi=d\Phi_e$. 
    
    To see this, let $Z,Z'\in \mathfrak{z}(\mathcal{N}_1)$ and $X,X'\in \mathfrak{z}(\mathcal{N}_1)^{\perp}\cap \mathcal{N}_1$ such that $X+Z, X+Z'$ form an orthonormal basis of $\pi$. Note that $K(\pi)=K_{\max}(\mathcal{N}_1)$. Define $X_1, ..., X_4, Z_1, Z_2$ as in (\ref{basis_equations_in_rigidity}) right below Lemma \ref{lemma_consequence_of_sixth_ineq}. Then $X_1, ..., X_4, Z_1, Z_2$ form an orthonormal basis of $\mathcal{N}_1$ (see Proposition \ref{proposition_key_for_rigidity}). The key observation is that $\varphi(X)+\varphi(Z), \varphi(X')+\varphi(Z')$ also form an orthonormal basis of $\pi'$, and $\varphi(X_1), ..., \varphi(X_4), \varphi(Z_1), \varphi(Z_2)$ form an orthonormal basis of $\mathcal{N}'$. For example, $X_3=\frac{j(Z)X}{|j(Z)X|}$, and $\varphi(X_3)=\frac{j'(\varphi(Z))\varphi(X)}{|j'(\varphi(Z))\varphi(X)|}$, since $\varphi$ is a metric Lie algebra isomorphism.
    
    Given $\varphi(\mathcal{N}_1)=\mathcal{N}'$ and $X_5, X_6\in \mathcal{N}_1^{\perp}$, we have $\varphi(X_5), \varphi(X_6)\in (\mathcal{N}')^{\perp}=\mathcal{N}$. Since $[X_5, X_6]=\lambda^2 Z_1$ and $Z_1\in \mathcal{N}_1$, we have $\lambda^2 \varphi(Z_1)\in \mathcal{N}\cap \mathcal{N}'=0$, a contradiction.
\end{proof}
    
\subsection{Heisenberg Groups}
In this section, we compute the pinching constants of any left-invariant metric on the direct product of $\mathrm{Heis}_{2n+1}(\mathbb{R})$ and the abelian Lie group $\mathbb{R}^m$. The $(2n+1)$-dimensional Heisenberg algebra $\mathrm{heis}_{2n+1}(\mathbb{R})$ is defined in \cite[p. 617, Example 1]{Eb1}. The $(2n+1)$-dimensional Heisenberg group $\mathrm{Heis}_{2n+1}(\mathbb{R})$ is a simply connected Lie group with Lie algebra $\mathrm{heis}_{2n+1}(\mathbb{R})$.

\begin{lemma} \label{classification_1-dimensional_commutator}
   Let $\mathcal{N}$ be a metric 2-step nilpotent Lie algebra with a 1-dimensional commutator ideal $[\mathcal{N}, \mathcal{N}]$. Then, there exist orthogonal ideals $\mathcal{N}_1$ and $\mathcal{N}_2$ of $\mathcal{N}$ such that $\mathcal{N}=\mathcal{N}_1\oplus \mathcal{N}_2$ (an orthogonal direct sum of metric Lie algebras), where  $\mathcal{N}_1$ is isomorphic to $\mathrm{heis}_{2n+1}(\mathbb{R})$ for some $n\geq 1$, and $\mathcal{N}_2$ is abelian, i.e., isomorphic to $\mathbb{R}^m$ for some $m\geq 0$.
    
\end{lemma}
\begin{proof} In general, given any metric 2-step nilpotent Lie algebra $\mathcal{N}$, one can formulate the $j$-map as a linear map $j':[\mathcal{N}, \mathcal{N}]\to \mathrm{so}([\mathcal{N}, \mathcal{N}]^{\perp})$ defined by $\langle j'(Z)X, Y\rangle=\langle Z, [X,Y]\rangle$ for any $Z\in [\mathcal{N}, \mathcal{N}]$ and $X, Y\in [\mathcal{N}, \mathcal{N}]^{\perp}$ (\cite[p. 174]{Eberlein_prescribed_Ricci}). Now, let $\mathcal{N}$ be as in the statement. Let $Z_1$ be a unit vector such that $[\mathcal{N}, \mathcal{N}]=\mathbb{R}\textrm{-span}\{Z_1\}$. Since $j'(Z_1)$ is skew-symmetric, by the spectral theorem, there exists an orthonormal basis $X_1, Y_1, ..., X_n, Y_n, U_1, ..., U_m$ of $[\mathcal{N}, \mathcal{N}]^{\perp}$ and non-zero constants $a_1,...,a_n\in \mathbb{R}\setminus \{0\}$ such that $j'(Z_1)X_i=a_i Y_i, j'(Z_1)Y_i=-a_i X_i$, and $j'(Z_1)U_j=0$ for any $i\in \{1,...,n\}, j\in \{1,...,m\}$. Let $\mathcal{N}_1=\mathbb{R}\textrm{-span}\{X_1, Y_1, ..., X_n, Y_n, Z_1\}$ and $\mathcal{N}_2=\mathbb{R}\textrm{-span}\{U_1, ..., U_m\}$. Then the result follows.
\end{proof}

\begin{proposition} \label{max_and_min_of_heis_2n+1}
    Let $\mathcal{N}$ be a metric 2-step nilpotent Lie algebra with a 1-dimensional commutator ideal $[\mathcal{N}, \mathcal{N}]$. Then, $\frac{K_{\min}}{K_{\max}}=-3$.
\end{proposition}

\begin{proof} Let $\mathcal{N}=\mathcal{N}_1\oplus \mathcal{N}_2$ as in Lemma \ref{classification_1-dimensional_commutator}. Since $K\equiv 0$ on $\mathcal{N}_2$ and $K_{\max}>0$ and $K_{\min}<0$ on $\mathcal{N}_1$, we may assume $\mathcal{N}=\mathcal{N}_1$ by Lemma \ref{max_and_min_of_product}.
    Let $\mathcal{Z}$ be the center of $\mathrm{heis}_{2n+1}(\mathbb{R})$ and $Z_1$ be a unit vector such that $\mathcal{Z}=\mathbb{R}\textrm{-span}\{Z_1\}$. Let $\mathcal{V}$ be the orthogonal complement of $\mathcal{Z}$. First, we have $K_{\min}=-\frac{3}{4}||j||_{\op}^2$ by Theorem \ref{minimum_of_sectional_curvature_of_2-step_nilpotent_Lie_groups}. We will show $K_{\max}=\frac{1}{4}||j||_{\op}^2$. 

    Let $\pi$ be any 2-plane. There exists an orthonormal basis $X+Z, X'+Z'$ of $\pi$ such that $X, X'\in \mathcal{V}$ and $Z, Z'\in \mathcal{Z}$ with $\langle X, X'\rangle =0=\langle Z, Z'\rangle$ by Lemma \ref{rotation_in_2-plane}. Since $\mathcal{Z}$ is 1-dimensional, if $Z\neq 0, Z'\neq 0$, it is impossible to have $\langle Z, Z'\rangle=0$. Thus, we may assume $Z'=0$ by symmetry. By Proposition \ref{general_formula}, $K(X+Z,X'+Z')=\frac{1}{4}|j(Z')X|^2 + \frac{1}{4}|j(Z)X'|^2 -\frac{3}{4}|[X,X']|^2+\frac{1}{2}\langle j(Z)X', j(Z')X \rangle - \langle j(Z)X, j(Z')X' \rangle=\frac{1}{4}|j(Z)X'|^2-\frac{3}{4}|[X,X']|^2\leq \frac{1}{4}|j(Z)X'|^2$. By Lemma \ref{upper_bound_of_Lie_bracket}, we have $K(X+Z,X'+Z')\leq \frac{1}{4}|j(Z)X'|^2\leq \frac{1}{4}||j||_{\op}^2|Z|^2|X'|^2\leq \frac{1}{4}||j||_{\op}^2$. 

    On the other hand, $K_{\max}\geq \frac{1}{4}||j||_{\op}^2$ by the item 3 of Lemma \ref{maximum_eigenspace_and_two_interpretations}. Thus, $K_{\max}= \frac{1}{4}||j||_{\op}^2$, and $\frac{K_{\min}}{K_{\max}}=-3$. This completes the proof.
\end{proof}

\begin{corollary} \label{pinching_constant_of_heis_2n+1_plus_abelian}
    Let $N=\mathrm{Heis}_{2n+1}(\mathbb{R})\times \mathbb{R}^m$ with $n\geq 1$ and $m\geq 0$, i.e., $N$ is a Lie group direct product of $\mathrm{Heis}_{2n+1}(\mathbb{R})$ and an abelian Lie group $\mathbb{R}^m$. (When $m=0$, $N=\mathrm{Heis}_{2n+1}(\mathbb{R})$.) Then, $\frac{K_{\min}}{K_{\max}}=-3$.
\end{corollary}

\subsection{2-step Nilmanifolds of Dimensions 5 or Less}
In this section, we compute the pinching constants of 2-step nilmanifolds of $\dim \leq 5$.
\begin{definition}
    Let $\mathcal{V}=\mathbb{R}\textrm{-span}\{e_1, e_2, e_3\}$ and $\mathcal{Z}=\mathbb{R}\textrm{-span}\{w_1, w_2\}$ be vector spaces and $\mathcal{N}_5=\mathcal{V}\oplus \mathcal{Z}$ be the direct sum of the vector spaces. Define the Lie bracket on $\mathcal{N}_5$ such that $[e_1, e_2]=w_1$, $[e_1, e_3]=w_2$ with the other Lie bracket relations 0.

    Then, $\mathcal{N}_5$ is a 2-step nilpotent Lie algebra with a 2-dimensional center $\mathcal{Z}$, and $[\mathcal{N}_5, \mathcal{N}_5]=\mathcal{Z}$.
\end{definition}

\begin{lemma} \label{classification_up_to_dim_5}
    Let $\mathcal{N}$ be a 2-step nilpotent Lie algebra of dimension 5 or less. Then, $\mathcal{N}$ is isomorphic to $\mathrm{heis}_3(\mathbb{R})$, $\mathrm{heis}_3(\mathbb{R})\oplus \mathbb{R}$, $\mathrm{heis}_3(\mathbb{R})\oplus \mathbb{R}^2$, $\mathrm{heis}_5(\mathbb{R})$, or $\mathcal{N}_5$.
\end{lemma}

\begin{proof}
    One can consult to \cite[p.22]{Console_Sergio_Fino_Evangelia} (see also \cite[Theorem 10]{Homolya_Kowalski}). 
\end{proof}

\begin{proposition} \label{proposition_dim_5}
    Let $N_5$ be a simply connected 2-step nilpotent Lie algebra of dimension 5 whose Lie algebra is $\mathcal{N}_5$. Then, for any left-invariant metric on $N_5$, we have $K_{\max}=\frac{1}{4}||j||_{\op}^2, K_{\min}=-\frac{3}{4}||j||_{\op}^2$, and $\frac{K_{\min}}{K_{\max}}=-3$. 
\end{proposition}
\begin{remark}
    There are uncountably many left-invariant metrics on $N_5$ up to isometry and scaling.
\end{remark}
We split the proof into lemmas. 
\begin{lemma} \label{lemma_dim_5_standard_form}
    There exists an orthonormal basis $X_1, X_2, X_3$ of $\mathcal{V}$ and $Z_1, Z_2$ of $\mathcal{Z}$ such that \begin{align*}
        j(Z_1)=\begin{bmatrix}
            0 & -a & 0 \\
            a & 0 & 0 \\
            0 & 0 & 0
        \end{bmatrix}, j(Z_2)=\begin{bmatrix}
            0 & 0 & -b \\
            0 & 0 & 0 \\
            b & 0 & 0
        \end{bmatrix}
    \end{align*}

    In other words, $[X_1, X_2]=aZ_1, [X_1, X_3]=bZ_2$, and the other bracket relations are zero.

    Moreover, $a=||j||_{\op}$, and $|b|\leq a$.
\end{lemma}
\begin{remark}
     Note, \cite[6 Proposition]{Homolya_Kowalski} also shows that there exists an orthonormal basis with only two non-trivial Lie brackets. However, in our argument, we have $a=||j||_{\op}$, which is useful.
\end{remark}
\begin{proof}[Proof of Lemma \ref{lemma_dim_5_standard_form}.]
    Let $\mathcal{N}=\mathcal{N}_5$. Let $\mathcal{Z}$ be the center of $\mathcal{N}$, and $\mathcal{V}=\mathcal{Z}^{\perp}$. By Lemma \ref{maximum_eigenspace_and_two_interpretations}, there exist unit vectors $X_1, X_2 \in \mathcal{V}$ and $Z_1\in \mathcal{Z}$ such that $|j(Z_1)X_1|=||j||_{\op}$, $j(Z_1)^2X_1=-||j||_{\op}^2X_1$, $X_2=\frac{j(Z_1)X_1}{||j||_{\op}}$, and $[X_1,X_2]=||j||_{\op}Z_1$. Note, $j(Z_1)X_1=||j||_{\op}X_2$ and $j(Z_1)X_2=-||j||_{\op}X_1$. Extend $X_1, X_2$ to an orthonormal basis $X_1, X_2, X_3$ of $\mathcal{V}$. Extend $Z_1$ to an orthonormal basis $Z_1, Z_2$ of $\mathcal{Z}$. Let $\mathcal{N}_1=\mathbb{R}\textrm{-span}\{X_1, X_2, Z_1\}.$ As noted in Lemma \ref{maximum_eigenspace_and_two_interpretations}, $\mathcal{N}_1$ is an aligned totally geodesic subalgebra of $\mathcal{N}$. Let $\mathcal{V}_1=\mathbb{R}\textrm{-span}\{X_1, X_2\}, \mathcal{V}_2=\mathbb{R}\textrm{-span}\{X_3\}, \mathcal{Z}_1=\mathbb{R}\textrm{-span}\{Z_1\}, \mathcal{Z}_2=\mathbb{R}\textrm{-span}\{Z_2\}$. Then, by Lemma \ref{curv_op_and_totally_geodesic_subalgebras},  $j(Z_1)$ leaves $\mathcal{V}_1$ and $\mathcal{V}_2$ invariant, and $j(Z_2)(\mathcal{V}_1)\subseteq \mathcal{V}_2$. 
    
     Together with $j(Z_1)X_1=||j||_{\op}X_2$ and $j(Z_1)X_2=-||j||_{\op}X_1$, the matrix representations of $j(Z_1)$ and $j(Z_2)$ with respect to $X_1, X_2, X_3$ are given by  $j(Z_1)=\begin{bmatrix}
            0 & -||j||_{\op} & 0 \\
            ||j||_{\op} & 0 & 0 \\
            0 & 0 & 0
        \end{bmatrix}, j(Z_2)=\begin{bmatrix}
            0 & 0 & -b \\
            0 & 0 & -c \\
            b & c & 0
        \end{bmatrix}$
    for some $b, c\in \mathbb{R}$. In other words, $[X_1, X_2]=||j||_{\op}Z_1, [X_1, X_3]=bZ_2, [X_2, X_3]=cZ_2$, and the other bracket relations are zero.

    We claim that we may assume $c=0$ by using a different basis. Let $Y_1=(\cos \theta) X_1+(\sin \theta)X_2$ and $Y_2=(-\sin \theta)X_1 + (\cos \theta) X_2$. Then, $[Y_1, Y_2]=||j||_{\op}Z_1$, $[Y_1, X_3]=(b\cos \theta + c\sin \theta)Z_2$, and $[Y_2, X_3]=(-b\sin \theta +c\cos \theta )Z_2$. Thus, we can choose $\theta$ so that $-b\sin \theta + c\cos \theta=0$, i.e., we may assume $c=0$. Let $a=||j||_{\op}$. We have $|b|\leq a$. Indeed, $j(Z_2)X_1=bX_3$, so $|b|=|j(Z_2)X_3|\leq ||j||_{\op}=a$. 
\end{proof}

Now, we will compute the curvature operator $R$, and show that the maximum eigenvalue of $R$ is $\frac{1}{4}a^2$. By Proposition \ref{decomposition_of_the_curvature_operator}, we can write $R=\begin{bmatrix}
        A_{11} & 0 \\
         0 & A_{22}
    \end{bmatrix}$ relative to the subspaces $\Lambda^2\mathcal{V}\oplus \Lambda^2\mathcal{Z}$ and $\mathcal{V} \otimes \mathcal{Z}$. As explained in Remark \ref{remark_block_form_of_curvature_operator}, our order of the basis elements is $X_1 \wedge X_2, X_1 \wedge X_3, X_2\wedge X_3, Z_1\wedge Z_2$ for $\Lambda^2 \mathcal{V}\oplus \Lambda^2\mathcal{Z}$ and $X_1\wedge Z_1, X_2\wedge Z_1, X_3\wedge Z_1, X_1 \wedge Z_2, X_2\wedge Z_2, X_3\wedge Z_2$ for $\mathcal{V}\otimes \mathcal{Z}$.

\begin{lemma}\label{lemma_dim_5_curvature_operator_A11_and_A22}
    The matrix representation of $A_{22}$ with respect to $X_1\wedge Z_1, X_2\wedge Z_1, X_3\wedge Z_1, X_1 \wedge Z_2, X_2\wedge Z_2, X_3\wedge Z_2$ is given by \begin{align}
        A_{11}=\begin{bmatrix}
            -\frac{3}{4}a^2 & 0 & 0 & 0 \\
            0 & -\frac{3}{4}b^2 & 0 & 0 \\
            0 & 0 & 0 & -\frac{1}{4}ab \\
            0 & 0 & -\frac{1}{4}ab & 0
        \end{bmatrix}, A_{22}=\begin{bmatrix}
            \frac{1}{4}a^2 & 0 & 0 & 0 & 0 & 0 \\
            0 & \frac{1}{4}a^2 & 0 & 0 & 0 & 0 \\
            0 & 0 & 0 & 0 & \frac{1}{4}ab & 0 \\
            0 & 0 & 0 & \frac{1}{4}b^2 & 0 & 0 \\
            0 & 0 & \frac{1}{4}ab & 0 & 0 & 0 \\
            0 & 0 & 0 & 0 & 0 & \frac{1}{4}b^2
        \end{bmatrix}.
    \end{align}
\end{lemma}
\begin{proof}[Proof of Lemma \ref{lemma_dim_5_curvature_operator_A11_and_A22}]
    Use Lemma \ref{formulas_for_curv_op} for $A_{11}$, and use the formula in Remark \ref{curv_op_on_V_wedge_Z_using_j_map} for $A_{22}$.
\end{proof}

\begin{proof}[Proof of Proposition \ref{proposition_dim_5}]
    The eigenvalues of $R$ are given by $\frac{1}{4}a^2, \frac{1}{4}b^2, \pm \frac{1}{4}ab, -\frac{3}{4}a^2, -\frac{3}{4}b^2$. Since $|b|\leq a$ and $a=||j||_{\op}$ from Lemma \ref{lemma_dim_5_standard_form}, the maximum eigenvalue $\lambda_{\max}(R)$ of $R$ is $\frac{1}{4}a^2=\frac{1}{4}||j||_{\op}^2$. Thus, it follows from Proposition \ref{bounds_by_curv_op_general_case} (with $T=0$) that $K_{\max}\leq \frac{1}{4}||j||_{\op}^2$. By Lemma \ref{maximum_eigenspace_and_two_interpretations}, $K_{\max}\geq \frac{1}{4}||j||_{\op}^2$. Thus, $K_{\max}=\frac{1}{4}||j||_{\op}^2$. By Theorem \ref{minimum_of_sectional_curvature_of_2-step_nilpotent_Lie_groups}, we have $K_{\min}=-\frac{3}{4}||j||_{\op}^2$. Thus, $\frac{K_{\min}}{K_{\max}}=-3$.
\end{proof}

\begin{corollary} \label{pinching_constant_up_to_dim_5}
    Let $N$ be a 2-step nilpotent Lie group of $\dim N\leq 5$. Then, for any left-invariant metric on $N$, we have $\frac{K_{\min}}{K_{\max}}=-3$.
\end{corollary}

\subsection{Quaternionic Heisenberg Group} \label{section_quaternionic_Heisenberg_group}

The 7-dimensional quaternionic Heisenberg group $N$ has Lie algebra $\mathcal{N}:=\mathrm{Lie}N$ of the form $\mathcal{N}=\mathbb{R}\textrm{-span}\{X_1, X_2, X_3, X_4, Z_1, Z_2, Z_3\}$ with the following non-trivial Lie bracket relations: $[X_1, X_3]=Z_1, [X_2, X_4]=Z
_1, [X_1, X_4]=Z
_2, [X_2,X_3]=-Z_2, [X_1, X_2]=-Z_3$, and $[X_3,X_4]=Z_3$. (If one sets $X_2=-W_1, X_3=Y_1, X_4=-W_1$, then one recovers the case $i=1$ of \cite[p. 618]{Eb1}.) Let $\mathcal{V}=\mathbb{R}\textrm{-span}\{X_1, ..., X_4\}$ and $\mathcal{Z}=\mathbb{R}\textrm{-span}\{Z_1,Z_2,Z_3\}$. The left-invariant metric for which $X_1, ..., X_4, Z_1, Z_2, Z_3$ form an orthonormal frame is of Heisenberg type and hence a Ricci soliton metric (see \cite{Lauret_Ricci_Soliton_Nilmanifolds}). It immediately follows that $||j||_{op}=1$, so $K_{\min}=-\frac{3}{4}$ by Theorem \ref{minimum_of_sectional_curvature_of_2-step_nilpotent_Lie_groups}. Let $\pi=\mathbb{R}\textrm{-span}\left\{\frac{2}{\sqrt{11}}X_3+\sqrt{\frac{7}{11}} Z_2, \frac{2}{\sqrt{11}}X_4+\sqrt{\frac{7}{11}} Z_1 \right\}$. Then, $K(\pi)=\frac{4}{11}$. Hence, $K_{\max}\geq \frac{4}{11}$, and $\frac{K_{\min}}{K_{\max}}\geq -\frac{33}{16}$. 

We show $K_{\max}=\frac{4}{11}$. Let $T_1=-\frac{1}{4}(X_1\wedge Z_1\wedge X_2\wedge Z_2)^*-\frac{1}{4}(X_1\wedge Z_1\wedge X_4\wedge Z_3)^*+\frac{1}{4}(X_2\wedge Z_1\wedge X_3\wedge Z_3)^*+\frac{1}{4}(X_3\wedge Z_1\wedge X_4\wedge Z_2)^*+\frac{1}{4}(X_1\wedge Z_2\wedge X_3\wedge Z_3)+\frac{1}{4}(X_2\wedge Z_2\wedge X_4\wedge Z_3)^*$. Let $T_2=(X_1\wedge X_2\wedge X_3\wedge X_4)^*$. Let $T=t_1T_1+t_2T_2$. Then $\iota(T)$ leaves $\mathcal{V}\otimes \mathcal{Z}$ and $\Lambda^2 \mathcal{V} \oplus \Lambda^2\mathcal{Z}$. When $t_1=\frac{6}{11}$ and $t_2=\frac{49}{44}$, the eigenvalues of $(R+\iota(T))|_{\mathcal{V}\otimes \mathcal{Z}}$ are $\{\frac{4}{11}, \frac{1}{44}\}$, and the eigenvalues of $(R+\iota(T))|_{\Lambda^2 \mathcal{V} \oplus \Lambda^2\mathcal{Z}}$ are $\{\frac{4}{11}, -\frac{49}{22}\}$. Hence, the maximum eigenvalue of $R+\iota(T)$ is $\frac{4}{11}$. By Proposition \ref{bounds_by_curv_op_general_case}, $K_{\max}\leq \frac{4}{11}$. Thus, $K_{\max}=\frac{4}{11}$ and $\frac{K_{\min}}{K_{\max}}=-\frac{33}{16}$.

\bibliographystyle{alpha}
\bibliography{main}{}

\end{document}